\documentclass[twoside,11pt]{article}

\usepackage{jmlr2e}
\usepackage{hyperref,color}
\usepackage{soul}
\usepackage{bbm}
\usepackage{amsmath,latexsym}
\usepackage{mathtools}
\usepackage{epstopdf}
\def\bbR{\mathbb{R}}

\DeclareMathOperator*{\argmin}{arg\,min}

\newcommand{\T}{{\mathcal T}}

\DeclarePairedDelimiter{\ceil}{\lceil}{\rceil}

\usepackage[normalem]{ulem}
\usepackage[ruled,vlined]{algorithm2e}

\jmlrheading{1}{2000}{1-68}{4/00}{10/00}{meila00a}{Kexuan Li, Ruiqi Liu, GanggangXu and Zuofeng Shang}

\ShortHeadings{Nonparametric Inference under B-bits Quantization}{Li, Liu, Xu and Shang}
\firstpageno{1}

\begin{document}

\title{Nonparametric Inference under B-bits Quantization}

\author{\name Kexuan Li \email kexuan.li.77@gmail.com\\
       \addr Global Analytics and Data Sciences\\
       Biogen Inc\\
       Cambridge, MA  02142 USA
       \AND
	   \name Ruiqi Liu \email ruiqliu@ttu.edu\\
       \addr Department of Mathematics and Statistics\\
       Texas Tech University\\
       Lubbock, TX 79409, USA
       \AND
       \name Ganggang Xu \email gangxu@bus.miami.edu \\
       \addr Department of Management Science\\
        University of Miami\\
       Coral Gables, FL 33146, USA
       \AND
       \name Zuofeng Shang \email zshang@njit.edu \\
       \addr Department of Mathematical Sciences\\
       New Jersey Institute of Technology\\
       Newark, NJ 07102, USA
       }

\editor{Rina Barber}

\maketitle

\begin{abstract}
Statistical inference based on lossy or incomplete samples is often needed in research areas such as signal/image processing, medical image storage, remote sensing, signal transmission.  In this paper, we propose a nonparametric testing procedure based on samples quantized to $B$ bits through a computationally efficient algorithm. Under mild technical conditions, we establish the asymptotic properties of the proposed test statistic and investigate how the testing power changes as $B$ increases. In particular, we show that if $B$ exceeds a certain threshold, the proposed nonparametric testing procedure achieves the classical minimax rate of testing \citep{SC15} for spline models. We further extend our theoretical investigations to a nonparametric linearity test and an adaptive nonparametric test, expanding the applicability of the proposed methods. Extensive simulation studies {together with a real-data analysis} are used to demonstrate the validity and effectiveness of the proposed tests.
\end{abstract}


\section{Introduction}
Lossy or incomplete data are commonly encountered in research areas such as machine learning, information theory, and signal processing. To store and process signals in digital devices, \textit{quantization} is a popular procedure that maps the original measurements from a large (often uncountably infinite) set
to a set of possible values. The resulting values are referred to as the \textit{quantized samples}. With the increasing availability of data, it is of great interest to quantify how the data analysis can be affected when
the data are quantized due to storage or communication budget constraint, and how to design quantization schemes to minimize the efficiency loss.  Statistical inference based on quantized samples is challenging because, in addition to the measurement errors, one also needs to account for the information loss due to the quantization errors. In particular, commonly used standard statistical procedures may not be valid when applied to quantized samples if the quantization errors are ignored.

The research on lossy data has attracted increasing attention recently. The first line of works focuses on $b$-bit compressive sensing, which aims at reconstructing a sparse signal from a sequence of $b$-bit quantized outcomes. A $1$-bit compressive sensing model was proposed by \cite{BB08}, and several efficient and provable algorithms have been developed; see, e.g., \cite{GNR10,  GNJN13, PV13, ZYJ14, ZG15}. A signal recovery algorithm {was} proposed in \cite{SL15}, which {extended} the $1$-bit compressive sensing model to a $b$-bit compressive sensing model. The second line of research related to the lossy data is to develop statistical methods based on quantized observations. For example, \cite{lee2001interval} studied the interval estimation of a normal mean process from rounded data, which was further extended to more general likelihood-based statistical estimation problems \citep{vardeman2005likelihood} and nonparametric regression problems \citep{BR06}.  Recently, an increasing number of works aim to quantify the impact of quantization on the statistical properties of the resulting estimators.  For example, \cite{zhang2013information} established lower bounds on the minimax risks for distributed estimation of parametric models under a communication budget constraint. \cite{suresh2017distributed} proposed communication efficient
algorithms for distributed mean estimation without probabilistic assumptions on the data.
 A version of Pinsker's theorem under some storage or communication constraints was developed in \cite{ZL14}, and it was further applied to analyze the convergence rate of nonparametric estimation with {a limited bits budget} by \cite{ZL17}. More recently, a series of works have emerged in investigating the high-dimensional and/or nonparametric regression model estimation in the distributed learning framework with bits constraints, e.g., see \cite{zhu2018distributed, han2018distributed, szabo2020adaptive, cai2021distributed}.

{Despite the abundant existing literature on statistical modeling of quantized data},  research focusing on the nonparametric inference based on quantized data is still lacking.  This paper aims to fill this gap by proposing {a new quantization scheme with a $B$-bits storage or communication budget such that nonparametric estimation and testing based on quantized samples are still valid.} Specifically, we consider the following regression model
\begin{equation}\label{regression:model}
y_i=g_0(i/n)+\sigma\epsilon_i,\,\,i=1,\ldots,n,
\end{equation}
where $g_0(\cdot)$ is a smooth function, $\epsilon_i$'s are \textit{iid} zero-mean errors with  an unit variance, and $\sigma>0$ is an unknown constant. The goal is to (a) estimate $g_0(\cdot)$, and (b)  test the following hypothesis
\begin{eqnarray}\label{H0}
H_0:\,\,\textrm{ $g_0(x)=g_*(x)$ for all $x\in [0,1]$},
\end{eqnarray}
where $g_*(\cdot)$ is a pre-specified deterministic function. 

The above model has been extensively studied in the literature, see, e.g., \cite{SC17}, and is closely related to the well-known Gaussian sequence model and Gaussian white noise model \citep{Tsybakov08nonpara}. However, unlike existing literature,  we consider the case in which the original data, denoted by $y_1,\cdots, y_n$, are generated in machine M, and are quantized as soon as they are generated. The quantized data are then stored in a machine M or  transmitted to another machine M$^*$ for future statistical inferences. We assume that only $B$-bits budget are available for data storage or communication, rendering the necessity for data quantization that may invalidate existing estimation and inference methods. Such a research problem is important for applications where data generation and analysis are carried out at different locations. For example, testing $H_0: g_0(x)=0$ reveals whether the transmitted quantized signals through satellite are pure noises. If $g_*(\cdot)$ is the signal-process from a normally functioning machine, testing (\ref{H0}) using only quantized samples enables us to remotely monitor whether the machine is working properly in real-time.

To meet the $B$-bits requirement, we propose a two-stage quantization procedure: in the first stage we quantize an individual $y_i$ as $Q(y_i)$ with $Q(\cdot)$ being a quantizer, and in the second stage we overwrite these quantized observations by their local averages. See Figure~\ref{figure:diagram} and Algorithm~\ref{alg:distributed} for details. As a result, we obtain a quantized sample of size $c$ for some $c<n$ to be stored or transmitted. We demonstrate that with a carefully chosen $c$ and a well-designed quantizer, the proposed nonparametric estimation and testing procedures are asymptotically valid and efficient even based only on the quantized data.

Our contributions can be summarized as follows. Firstly, we propose a computationally efficient data quantization algorithm to reduce the size of the raw data to meet the $B$-bits constraint, {and at the same time reduce the computational complexity from $O(n^3)$ to $O(c^3)$.} Secondly, we establish sufficient conditions on the bits constraint, i.e., $B$,  that warrants the minimax convergence rate for the resulting spline estimators and {the minimax rates of testing for the proposed testing procedure. In particular, our results show how the asymptotic power of the proposed testing procedure changes as the bits constraint $B$ increases.  Thirdly, we further extend our theoretical investigations to (a) a nonparametric linearity test of the underlying function; (b)  an adaptive nonparametric test when the smoothness of the underlying function is unknown.} To the best of our knowledge, our work is the first to provide a  theoretical investigation on nonparametric inference based on quantized samples.

The rest of the paper is organized as follows. Section~\ref{sec:method} describes the general methodologies we proposed for data quantization, nonparametric estimation, and nonparametric testing using splines. In Section~\ref{sec:asymptotics}, we investigate the theoretical properties of the spline estimator and the nonparametric test statistic based on quantized samples. In Section~\ref{sec:ext}, we study asymptotic properties of the nonparametric linearity test statistic and the adaptive nonparametric test statistic under B-bits constraint. Section~\ref{sec:simulation} gives several simulation studies to evaluate finite sample performances of the proposed methods and Section~\ref{real data} illustrates an application of the proposed methods to the Combined Cycle Power Plant Data.

 \textbf{Notation:} Let $\|\cdot\|$ represent the $L^2$-norm, i.e., $\|f\|^2=\int_0^1 f^2(t)dt$, and define $\|\cdot\|_2$ as the Euclidean Norm of vectors. Let $\|\cdot\|_{\sup}$ denote the supreme norm of a function, i.e., $\|f\|_{\sup}=\sup_{t\in[0,1]}|f(t)|$. For two positive sequences $a_n$ and $b_n$, we denote $a_n\gtrsim b_n$ ($a_n\lesssim b_n$) if there exists a constant $C>0$ such that $a_n\ge C b_n$ ($a_n\le C b_n$) for all $n\ge 1$; denote $a_n\asymp b_n$ if $a_n\gtrsim b_n$ and $a_n\lesssim b_n$; denote $a_n\ll b_n$ if $a_n/b_n\to 0$ as $n\to\infty$ and $a_n\gg b_n$ if $a_n/b_n\to \infty$ as $n\to\infty$.

\section{Methodology}
\label{sec:method}
In this section, we first review some background of the classical smoothing spline regression and then give details on the proposed quantization scheme, nonparametric estimation and testing procedures.
\subsection{Review of Classical Smoothing Spline Regression}
Throughout this paper, we assume that  the underlying true function  $g_0(\cdot)$ belongs to
the $m$-order ($m\geq 1$) periodic Sobolev space  on $\mathbb{I}:=[0, 1]$ defined as
\[
S^m(\mathbb{I})=\left\{\sum_{\nu=1}^\infty \beta_\nu\varphi_\nu(\cdot): \sum_{\nu=1}^\infty \beta_\nu^2\gamma_\nu<\infty\right\},
\]
where
$\varphi_{2k-1}(x)=\sqrt{2}\cos(2\pi kx),\,\,\,\,
\varphi_{2k}(x)=\sqrt{2}\sin(2\pi kx)$ are the trigonometric basis functions, and
$\gamma_{2k-1}=\gamma_{2k}=(2\pi k)^{2m}$  for $x\in\mathbb{I}$ and $k\ge1$. It follows from \cite{W90} and \cite{G13} that $S^m(\mathbb{I})$ is a reproducing kernel Hilbert space (RKHS) endowed with an inner product $J(f,g)=\int_0^1 f^{(m)}(x)g^{(m)}(x)dx$ and a reproducing kernel
\begin{equation*}
K(x, y)=\frac{(-1)^{m-1}}{(2m)!}B_{2m}(|x-y|),
\end{equation*}
where $B_{2m}$ is the Bernoulli polynomial of order $2m$.

Based on the above assumptions on $g_0(\cdot)$, the classic smoothing spline (ss) estimator of
$g_0(\cdot)$ is obtained through the following optimization problem:
\begin{align}
	\widehat{g}^{\textrm{ss}}
&\equiv\argmin_{g\in S^{m}}\frac{1}{n}\sum_{i=1}^n(y_i-g(i/n))^2+\lambda\int_0^1 [g^{(m)}(x)]^2dx.\label{eq:ss:optimization}
\end{align}
For $x\in\mathbb{I}$, we can define a function $K_{x}(\cdot)=K(x,\cdot)$, which belongs to $S^m(\mathbb{I})$.  By the representer Theorem \citep{G13}, the solution to (\ref{eq:ss:optimization}) has the following closed-form
\begin{equation}\label{ss:est:expression}
\widehat{g}^{\textrm{ss}}=\sum_{i=1}^n\theta_i K_{i/n},
\end{equation}
where $\theta=(\theta_1,\ldots,\theta_n)^T=n^{-1}(\Sigma_n+\lambda I_n)^{-1}y$
with  $\Sigma_n=[K(i/n,i'/n)/n]_{1\le i,i'\le n}\in \mathbb{R}^{n\times n}$, $y=(y_1,\ldots,y_n)^T\in \mathbb{R}^n$ and $I_n$ being the $n\times n$ identify matrix.

To conduct hypothesis test for (\ref{H0}), a straightforward idea is to construct a testing statistic based on the distance between $\widehat{g}^{\textrm{ss}}(\cdot)$ and $g_*(\cdot)$. Specifically, we use the $L_2$ norm distance defined as
\begin{eqnarray*}
T^{\textrm{ss}}=\|\widehat{g}^{ss}-g_*\|^2.
\end{eqnarray*}
With an appropriate normalization, it can be shown that $T^{\textrm{ss}}$ is asymptotically normally distributed \citep{SC17,pmlr-v125-yang20a,LSC18}.


\subsection{Two-Stage Quantization}
\label{sec:quan}
The original observations $y_i$'s in (\ref{regression:model}) are real-valued random variables, each of which literally requires an infinite amount of bits to store or transmit. When there are only $B$ available bits, the original observations $y_i$'s may not be directly accessible for estimation or testing, and hence, the classical smoothing spline estimator given in (\ref{ss:est:expression}) is not applicable. This section aims to introduce a two-stage quantization scheme to transform $y_i$'s into the ones whose storage or transmission meets the $B$-bits constraint.
The resulting samples will be further used for optimal inferential purposes in the subsequent sections. The two-stage quantization process is demonstrated in the following Figure~\ref{figure:diagram}.

\begin{figure}[ht!]
  \centering
  \includegraphics[scale=0.7]{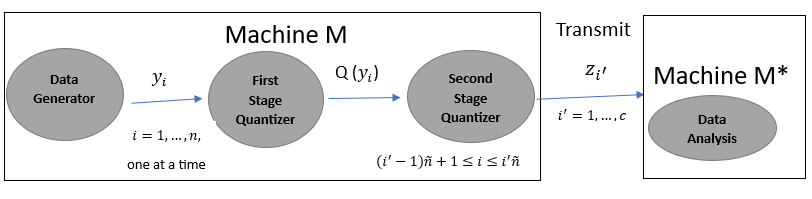}
  \caption{\it Two-stage quantization process.}
  \label{figure:diagram}
\end{figure}

The first-stage quantization is to quantize the data $y_i$'s as soon as they are generated with at most $k$ distinct values. 
For convenience, we use a uniform quantization scheme as follows. We first choose an interval $[t_1,t_{k-1}]$ and choose $t_2<\ldots <t_{k-2}$ as the {equally spaced grid points within} $[t_1, t_{k-1}]$. Denote $t=(t_1,\ldots,t_{k-1})^T\in \mathbb{R}^{k-1}$ and the sub-interval length $C_k(t):=(t_{k-1}-t_1)/(k-2)$. For ease of presentation, we assume that $l_1=t_1/C_k(t)$ is an integer. Define a quantizer $Q(\cdot)$ as follows:
\begin{eqnarray}
    \label{eq:quant}
\textrm{First-stage quantization:}\,\,&&Q(y)=\sum_{j=1}^k\mu_j I[y\in R_j(t)], \text{ with }\mu_j=(l_1+j-1)C_k(t),\nonumber\\
&&\text{ for any } y\in\mathbb{R},
\end{eqnarray}
where $\mu=(\mu_1, \ldots, \mu_k)^T\in \mathbb{R}^k$ consists of the quantized values and $R_1(t)=(-\infty,t_1],R_2(t)=(t_1,t_2],\ldots,R_{k-1}(t)=(t_{k-2},t_{k-1}]$,
$R_k(t)=(t_{k-1},\infty)$ are the corresponding quantized intervals. Clearly, the $R_j(t)$'s form a partition of the real line with assigned marks $\mu_j$'s and $Q$ maps each $y\in\mathbb{R}$ to one of the $k$ marks.
Applying $Q$ to $y_i$'s, we generate $n$ quantized samples $Q({y}_1),\cdots,Q({y}_n)$, each of which takes at most $k$ distinct values.
Storage or transmission of $Q(y_i)$'s thus requires $n\log_2{k}$ bits which might still go beyond the $B$-bits budget when $B=o(n)$.
For this reason, we propose the following second-stage quantization to further reduce the storage or transmission bits through locally averaging the $Q(y_i)$'s. 

The second-stage quantization is to further reduce the number of storage or transmission bits via local average. Specifically, we divide the interval $\mathbb{I}=[0,1]$ into $c$ equally-spaced sub-intervals for some $c\le B/\ceil{\log_2(k+2)}$. For simplicity, we assume that $\widetilde{n}\coloneqq n/c$ is an integer and each sub-interval contains  $\widetilde{n} $ observations. The quantized data from the first-stage quantization, i.e., $Q(y_i)$'s, are further quantized as $z_i=\widetilde l_i C_k(t)$ such that
\begin{equation}
\label{eq:quant:2}
\textrm{Second-stage quantization:}\,\,\,\,\left|z_i - \frac{1}{\widetilde{n}}\sum_{j=(i-1)\widetilde{n}+1}^{i\widetilde{n}}Q(y_j)\right|\le C_k(t),\,\, \textrm{ for $i=1,\ldots,c$.}
\end{equation}
Details of the two-stage quantization algorithm are provided in the following Algorithm~\ref{alg:distributed}.

\begin{algorithm}[h!]
\SetAlgoLined
\textbf{Initialization:} set $\widetilde l_1=\cdots=\widetilde l_c=0$;\\
\KwIn{Data $y_1,\cdots, y_n$ that generated sequentially;}
\textbf{while} $1\le i\le c$ \textbf{do} 
\begin{enumerate}
\item[] Set ${\rm tmp}=0$;
\item[]\textbf{while}  $(i-1)\widetilde{n}+1\le j\le i\widetilde{n}$  \textbf{do}
\begin{enumerate}
    \item First-stage quantization (quantize on the spot): $Q(y_j)=l_j C_k(t)$; 
    \item  Second-stage quantization (averaging): \\
\begin{enumerate}
    \item {\it update} $\widetilde l_i=\widetilde l_i+\text{Integer Part of } l_j/\widetilde{n}$;
    \item   \textbf{if} ${\rm tmp}+\text{sign}(l_j)(|l_j| \text{ mod } \widetilde{n})\ge \widetilde n $ \textbf{ then}  \\\hspace{2em}{\it update} $\begin{cases}
     \widetilde l_i:=\widetilde l_i+1\\   
     {\rm tmp}:={\rm tmp}+\text{sign}(l_j)(|l_j|\text{ mod } \widetilde{n})-\widetilde{n};
    \end{cases}$ 
    \item[] \textbf{else if} ${\rm tmp}+\text{sign}(l_j)(|l_j| \text{ mod } \widetilde{n})\le -\widetilde n $
\textbf{ then}  \\
\hspace{2em}{\it update} $\begin{cases}
     \widetilde l_i:=\widetilde l_i-1\\   
     {\rm tmp}:={\rm tmp}+\text{sign}(l_j)(|l_j| \text{ mod } \widetilde{n})+\widetilde n;
    \end{cases}$
\item[]    
 \textbf{else}   \\\hspace{2em}{\it update} ${\rm tmp}:={\rm tmp}+\text{sign}(l_j)(|l_j|\text{ mod } \widetilde{n})$;\\
    \textbf{end if}  
\end{enumerate}

    \end{enumerate}
    \textbf{end while}
\end{enumerate}
\textbf{end while}

\KwOut{quantized data: $z_1=\widetilde l_1 C_k(t),\cdots,z_c=\widetilde l_c C_k(t)$.}
 \caption{Two-Stage Quantization}
 \label{alg:distributed}
\end{algorithm}

Based on the definition of the quantizer $Q(\cdot)$ in (\ref{eq:quant}), $z_i$ in~\eqref{eq:quant:2} must belong to the interval $[\mu_1-C_k(t), \mu_k+C_k(t)]$ and must be of the form $l C_k(t)$ for some integer $l$. Therefore, there are at most $k+2$ distinct values of $z_i$'s, namely, $l C_k(t)$, for $l=l_1-1, l_1,  l_1+1,\cdots,  (l_1+k-1),l_1+k$. As a result, each $z_i$ requires $b=\ceil{\log_2(k+2)}$ bits to store or transmit,
hence, the entire $z_i$'s require $c\ceil{\log_2(k+2)}\le B$ bits, where  $\ceil{x}$ is the smallest integer greater than $x$.  In the subsequent sections, we will show that, with $c,k$ being properly selected, 
optimal inferences based on $z_i$'s are possible even under $B=o(n)$, comparing to other regression literature which typically needs $B\gtrsim n$ (see \cite{SL16}). For optimal inferences in non-regression settings such as Gaussian sequence model or Gaussian white-noise model, similar findings were made by \cite{cai2021distributed}.



\subsection{$B$-bits Nonparametric Spline Estimation}
\label{sec:est}
Given $B$, let us choose $k,c$ such that $c\ceil{\log_2(k+2)}=B$, i.e.,
our two-stage quantization maximizes the use of the available bits. Based on the quantized samples $z_1,\ldots, z_c$ from Algorithm~\ref{alg:distributed}, a $B$-bits constrained spline estimator is proposed as follows
\begin{align}
\widehat g^{\textrm{B}}_{\mu,t,c}
&\equiv\argmin_{g\in S^m(\mathbb{I})}\frac{1}{c}\sum_{i=1}^c(z_i-g(i/c))^2+\lambda\int_0^1 [g^{(m)}(x)]^2dx. \label{class:smoothing:spline:quantized}
\end{align}
Similar to (\ref{ss:est:expression}), the resulting spline estimator $\widehat{g}^{\textrm{B}}_{\mu,t,c}(\cdot)$
has an explicit expression
\begin{equation*}
\widehat{g}^{\textrm{B}}_{\mu,t,c}=\sum_{i=1}^c\widehat{\theta}_i K_{i/c},
\end{equation*}
where $(\widehat{\theta}_1,\ldots,\widehat{\theta}_c)^T=c^{-1}(\Sigma_c+\lambda I_c)^{-1}z$ with $\Sigma_c=[K(i/c,i'/c)/c]_{1\le i,i'\le c}\in \mathbb{R}^{c\times c}$, $z=(z_1,\ldots,z_c)^T\in \mathbb{R}^c$, and $I_c$ being the $c\times c $ identity matrix.

Notice that the optimization of (\ref{class:smoothing:spline:quantized}) only requires on $c$ quantized observations and the solution only involves computing the inverse of a $c\times c$ matrix $\Sigma_c+\lambda I_c$, which is much less computationally intensive compared to the classical smoothing spline estimator (\ref{ss:est:expression}).  

Finally, the selection of the tuning parameter $\lambda$ is crucial, and can be obtained by minimizing the generalized cross validation (GCV) score as follows
\begin{equation} 
\label{gcv}
\widehat{\lambda}= \argmin_{\lambda>0}\frac{c\|[I_c-\Sigma_c(\Sigma_c+\lambda I_c)^{-1})]z\|_2^2}{[c-\textrm{trace}(\Sigma_c(\Sigma_c+\lambda I_c)^{-1})]^2},
\end{equation}
The GCV has been widely used in the literature and enjoys appealing theoretical properties in various settings, see, e.g., \cite{W90,xu2012asymptotic,G13,xu2018optimal,xu2019distributed}.

\subsection{$B$-bits Nonparametric Testing}\label{sec:test}
In this section, we propose a  test statistic for the null hypothesis~\eqref{H0} based on the $B$-bits spline estimator $\widehat{g}_{\mu,t,c}^{\textrm{B}}(\cdot)$. Without loss of generality, we assume $g_*(\cdot)\equiv 0$ in the null hypothesis~\eqref{H0}. For a nonzero $g_*(\cdot)$, the observed response variables $y_i$'s from model~\eqref{regression:model} can be centered as $y_i^*=y_i-g_*(i/n)$, and the same testing procedure can be applied using $y_i^*$'s instead. To test $H_0: g_0(x)=0$, we consider  test statistic based on the $L_2$ norm distance between $\widehat{g}_{\mu,t,c}^{\textrm{B}}(\cdot)$ and $g_*(\cdot)\equiv 0$ as following
\begin{equation}\label{L2T}
T_{\mu,t,c}=\|\widehat{g}_{\mu,t,c}^{\textrm{B}}\|^2.
\end{equation}
Intuitively, a large value of $T_{\mu,t,c}$ should lead to the rejection of $H_0$. In Theorem~\ref{asymp:distribution}, we shall show that under $H_0$ and mild conditions, it holds that
\[
\frac{cT_{\mu,t,c}-\textrm{trace}(A){\tau}_k^2}{s_c{\tau}_k^2}\overset{d}{\longrightarrow}N(0,1)
\,\,\textrm{as $n, c\rightarrow\infty$,}
\]
where $\tau_k^2=\textrm{Var}(z_1|H_0)$, $A=(\Sigma_c+\lambda I_c)^{-1}\Omega_c(\Sigma_c+\lambda I_c)^{-1}$, $\Omega_c=[K^{\otimes 2}(i/c,i'/c)/c]_{1\le i,i'\le c}$ with {$K^{\otimes 2}(x, x'):= \int_0^1K(x,y)K(y,x')dy$} and  $s_c^2=2\sum_{1\le i\neq i'\le c}a_{i,i'}^2$  with
$a_{i,i'}$ being the $(i,i')$th entry of $A$. In practice, $\tau_k^2$ needs to be estimated based on the quantized data as well. We proposed the following estimator
\[
\widehat{\tau}_k^2= \frac{\widetilde\tau_n^2}{2\widetilde{n}(n-1)},
\]
where $\widetilde\tau_n$ is given in the following Algorithm~\ref{alg:distributed-1} through quantization.
 Intuitively, the above estimator is a re-scaled (by a factor of $\widetilde n^{-1}$) version of the quantized sample error variance $\frac{1}{2(n-1)}\sum_{j=2}^n\left\{Q(y_j)-Q(y_{j-1})\right\}^2$.  It is straightforward to shown that $\widehat{\tau}_k^2=\tau_k^2[1+o_p(1)]$ under mild conditions, see Lemma~\ref{hat_tau_k} of Appendix C for details. Consequently, the decision rule for testing (\ref{H0}) at significance level $\alpha$ can be defined as follows
\begin{equation}\label{eq:test:statistic}
\phi_{c,k}=I(|cT_{\mu,t,c}-\textrm{trace}(A)\widehat{\tau}_k^2|\ge z_{1-\alpha/2}s_c\widehat{\tau}_k^2),
\end{equation}
where $z_{1-\alpha/2}$ is the $(1-\alpha/2)$-percentile of the standard normal distribution.
We reject the null hypothesis (\ref{H0}) if and only if $\phi_{c,k}=1$.

\begin{algorithm}[h!]
\SetAlgoLined
\textbf{Initialization:} set $\widetilde\tau_n^2={\rm tmp}=0$;\\
\KwIn{Data $y_1,\cdots, y_n$ that generated sequentially;}
\textbf{while} $1\le j\leq n$ \textbf{do}
\begin{enumerate}
    \item {Quantize on the spot } $Q(y_j)$;  
    \item \textbf{if } $j=1$  \textbf{ then} {\it update} ${\rm tmp}:=Q(y_1)$; \textbf{ else } 
    \begin{enumerate}
        \item {\it update} $\widetilde\tau_n^2:={\widetilde\tau_n^2+(Q(y_j)-{\rm tmp})^2}$;
        \item {\it update} ${\rm tmp}:=Q(y_j)$;
    \end{enumerate}
    \textbf{end if}
\end{enumerate}
\textbf{end while}

\KwOut{$\widetilde\tau_n^2$.}
 \caption{Quantization Estimation of Variance}
 \label{alg:distributed-1}
\end{algorithm}
By the design of the quantizer $Q(\cdot)$ in (\ref{eq:quant}), we can see that there are at most $k+1$ distinct possible values for each $\left\{Q(y_{j+1})-Q(y_j)\right\}^2$ ranging from $0C_k(t)^2$ to $k^2C_k(t)^2$, $1\le j\le n-1$, yielding the range for $\widetilde\tau^2$ as $[0, (n-1)k^2C_k(t)^2]$. Since $\widetilde\tau^2$ can only take values as $lC_k(t)^2$ for some integer $l$, there are at most $(n-1)k^2+1$ distinct values for $\widetilde\tau^2$, which would cost $\ceil{\log_2\left\{(n-1)k^2+1\right\}}$ bits to store or transmit. Compare to the bit costs for the two-stage quantization $c\ceil{\log_2(k)}\approx B$, the cost to store or transmit $\widetilde\tau^2$ is negligible when $c\to\infty$, hence is ignored in the calculation of the total bit costs for ease of presentation.

\subsection{Practical Choice of $c$ and $k$ Given $B$} \label{sec:select_c_k}
The implementations of Algorithms~\ref{alg:distributed} and~\ref{alg:distributed-1} require a practical choice of $c$ and $k$ for a given bits budget $B$. Based on the discussion in Section~\ref{sec:quan}, Algorithm~\ref{alg:distributed} requires $B=cb$ with $b=\ceil{\log_2\left(k\right)}$. Our theoretical investigations in Section~\ref{sec:power:fixm} require that $b\gg\log_2\left(\sqrt{(nh^{1/2}+n(ch)^{-1})\T_n}\right)$ for some $h\to 0$, and $ch\to\infty$ as $c,n\to\infty$, where $\T_n$ is defined in $\textrm{Condition (\textbf{B})}$. Furthermore, equations~\eqref{rate1} and \eqref{rate2} in Section~\ref{sec:power:fixm} reveal that the optimal choice of $b$ depends on the smoothness of the periodic Sobolev space (i.e., $m$) and the tuning parameter $\lambda$. While the former is typically unknown in practice, the latter needs to be chosen by some data-driven criterion such as GCV based on the quantized data, which is not available until the quantization process is carried out. To simplify the calculation and make the quantization algorithm more practical, we propose to use $b=\ceil{\log_2\left(\sqrt{n\T_n/\sigma^2}\right)}$, which is a valid choice for any $m$ and $h$, and therefore is easy to use in practice. Specifically, given $B$, we find $c$ and $k$ as follows
\begin{equation}
\label{ck}
\begin{split}
      &c^{\dag}=\max\left\{c\in \mathbbm{Z}: c\log_2\left( \sqrt{n\T_n/\sigma^2}\right)\le B\right\}, \\& k^{\dag}=\max\left\{k\in \mathbbm{Z}: c^{\dag}\ceil{\log_2(k)}\le B\right\}.  
\end{split}
\end{equation}

By the definition in $\textrm{Condition (\textbf{B})}$, $\sqrt{\T_n}$ is the quantization range, and $\sigma^2$ is used in the choice of $b$ so that $\T_n/\sigma^2$ is invariant if $y_i$'s are multiplied by  a constant. Under $\textrm{Condition (\textbf{B})}$, the actual choice of $\T_n/\sigma^2$ depends on the distribution of $\epsilon_i$'s in model~\eqref{regression:model}. If $\epsilon_i$'s follow a standard Gaussian distribution, it suffices to take $\T_n/\sigma^2=2.5\log(n)$. Therefore, $\sigma^2$ in~\eqref{ck} does not need to be estimated. See more discussion under $\textrm{Condition (\textbf{B})}$ regarding the choice of $\T_n/\sigma^2$.

\section{Asymptotic Theory}\label{sec:asymptotics}
We now proceed to study asymptotic properties of the $B$-bits spline estimator and the nonparametric test statistic. In this section, we restrict our investigation to the simple case scenario when the order $m$ of the periodic Sobolev space is known and fixed, and the exact form of function~$g_*(\cdot)$ in the null hypothesis~\eqref{H0} is also known. We shall defer theoretical results on more general cases to Section~\ref{sec:ext}.

\subsection{Estimation Convergence Rate}
\label{sec:est:rate}
{We first quantify the convergence rate of $\|\widehat{g}_{\mu,t,c}^{\textrm{B}}-g_0\|^2$.  Even though the main focus of this paper is conducting statistical inference based on quantized samples, it is still of interest to study the asymptotic properties of the spline estimator $\widehat{g}_{\mu,t,c}^{\textrm{B}}(\cdot)$.} 
Define the Sobolev constant
\begin{equation}\label{Sob:const}
c_s\equiv\sup_{g\in S^m(\mathbb{I})} \frac{\|g\|_{\sup}}{\sqrt{J(g,g)}}.
\end{equation}
It is known that $c_s$ is positive finite see \citep{adams2003sobolev}.

For all our theoretical investigations,  we assume that
$\mu_j$'s and $t_j$'s satisfy the following 
\textit{boundedness} condition

\begin{align*}
\textrm{Condition (\textbf{B})}: &\textrm{ Assume that $J(g_0,g_0)\le \rho^2$, and denote } \T_n=min\{t_1^2,t_{k-1}^2\}, \textrm{ it holds that,} \\
&\hspace{4em}\mu_j^2P\left(|\sigma\epsilon_1|+c_s\rho>\sqrt{\T_n}\right)=o\left(\frac{1}{n}\right) \textrm{ for } j=1,k.
\end{align*}

Condition (\textbf{B}) asserts that {the values of $\mu_1, \mu_k$ can not be to large}, and that $t_1,t_k$ should be sufficiently large. Recall that in this paper, we adopt the uniform quantization scheme for which Condition (\textbf{B}) is rather mild.  Since $J(g_0,g_0)\le \rho^2$, by the definition of $c_s$ in~\eqref{Sob:const}, we have that $\|g_0\|_{\sup}\le c_s\rho$, and we shall assume that $\rho$ is finite for our theoretical investigation. Condition (\textbf{B}) essentially assumes that $\T_n$ is sufficiently large so that all observed $y_i$'s fall within the quantization range with a high probability. When $\epsilon_i$'s follow a sub-Gaussian distribution, it suffices to take $\T_n\asymp\log{(n)}$ for Condition (\textbf{B}) to hold. For distributions with heavier tails, the required order for $\min\{t_1^2,t_{k-1}^2\}$ will be larger, e.g., $\T_n\asymp [\log(n)]^2$ for sub-Exponential  distributions. In particular, when $\epsilon_i$'s follow a normal distribution, it suffices to use $\T_n=2.5\sigma^2\log(n)$.

Based on Condition (\textbf{B}), the following theorem establishes an asymptotic upper bound for
the estimation error $E\|\widehat{g}_{\mu,t,c}^{\textrm{B}}-g_0\|^2$.


    

\begin{theorem}\label{upper:bound:hat:bb:f}
If Condition (\textbf{B}) holds, then it follows that
\begin{equation*}
E\|\widehat{g}_{\mu,t,c}^{\textrm{B}}-g_0\|^2
\lesssim (nh)^{-1} + \lambda +  c^{-\min\{2m,3\}} + G_{c,k}(t),
\end{equation*}
where $h = \lambda^{\frac{1}{2m}}$, and $G_{c,k}(t)=4C_k(t)^2+G_{c,k,1}(t)+G_{c,k,2}(t)$,
with
\begin{align*}
G_{c,k,1}(t)&=\frac{2}{n}\sum_{i=1}^{n}\int_{-\infty}^{t_1}(z-\mu_1)^2p\left(\frac{z-g_0(i/n)}{\sigma}\right)\sigma^{-1}dz,\\
G_{c,k,2}(t)&=\frac{2}{n}\sum_{i=1}^{n}\int_{t_{k-1}}^{\infty}(z-\mu_k)^2p\left(\frac{z-g_0(i/n)}{\sigma}\right)\sigma^{-1}dz,
\end{align*}
with $p(\cdot)$ being the distribution of $\epsilon$.
\end{theorem}

The asymptotic error bound for $\widehat{g}_{\mu,t,c}^{\textrm{B}}(\cdot)$ given in Theorem~\ref{upper:bound:hat:bb:f} can be roughly categorized into three parts:
(1) the estimation error of the smoothing spline estimator based on fully observed original data, i.e., $(nh)^{-1}+\lambda$ \citep{W90}; (2) the estimation error attributed to first-stage quantization, i.e., $G_{c,k}(t)$; and (3) the estimation bias introduced by second-stage quantization, i.e., $ c^{-\min\{2m,3\}}$. An extreme case is when
$t_1\to -\infty$, $t_{k-1}\to\infty$ and $C_k(t)\to 0$, i.e., the first-stage quantizer becomes dense enough, in
which case $G_{c,k}$ tends to zero, reducing to the classical nonparametric estimation setting.

Intuitively, if a sufficiently large bits budget $B$, and consequently sufficiently large values $c$ and $k$ can be used, term $(nh)^{-1} + \lambda$ will dominate the upper bound of $E\|\widehat{g}_{\mu,t,c}^{\textrm{B}}-g_0\|^2$. As a result, the convergence rate of $\|\widehat{g}_{\mu,t,c}^{\textrm{B}}-g_0\|^2$ coincides with that of the classical smoothing spline estimator based on original observations without quantization \citep{W90}. A sufficient condition is given in the following corollary.



\begin{corollary} \label{coro2}
 Assume that Condition (\textbf{B}) holds, and that (1) $C_k(t)^2 \lesssim n^{-2m/(2m+1)}$; (2) as $T\rightarrow\infty$, $p(z)$ satisfies $\int_{|z|\ge T} z^2p(z)dz=O(\exp(-T^{d}))$ where $d \geq \frac{4m}{2m+1}$; (3)  $\T_n \asymp \log(n)$; and that (4) $c\asymp n^{\frac{\max\{1,2m/3\}}{2m+1}}$, $\lambda \asymp n^{-\frac{2m}{2m+1}}$. Then it follows that $E\|\widehat{g}_{\mu,t,c}^{\textrm{B}}-g_0\|^2=O(n^{-\frac{2m}{2m+1}})$, which achieves the optimal convergence rate of smoothing splines without quantization.
\end{corollary}

Recall the definition $C_k(t)=(t_{k-1}-t_1)/(k-2)$, under conditions of Corollary~\ref{coro2}, the minimum order of $k$ to achieve the optimal convergence rate is $n^{m/(2m+1)}\log(n)$, leading to a  required $b=\ceil{\log_2(k)}\asymp \frac{m}{2m+1}\log_2(n)$.  Therefore, the total bits budget $B=cb\asymp n^{\frac{\max\{1,2m/3\}}{2m+1}}\log(n)$. Recently, \cite{zhu2018distributed} propose a quantization scheme for the Gaussian sequence model that achieves the same optimal estimation rate with a bits budget $B\asymp n^{\frac{1}{2m+1}}$. Although their bits budget is lower than our proposed method, \cite{zhu2018distributed} achieve this goal by essentially only quantizing the first $n^{\frac{1}{2m+1}}$ Fourier coefficients of the function $g_0(\cdot)$ and discarding the remaining Fourier coefficients as $0$'s. It is unclear how can this approach be extended to making valid nonparametric inferences for  $g_0(\cdot)$, which is the main focus of our work. The proposed quantization scheme in Section~\ref{sec:quan} is in spirit closer to the quantization algorithms proposed in~\cite{SL16} and references therein, although these works are mainly focused on the estimation of the parametric linear regression model. In the following subsections, we shall investigate the impacts of the bits budget on the asymptotic properties of the proposed nonparametric testing procedure.

\subsection{Asymptotic Distribution of the Test Statistic under $H_0$}\label{sec:bbit:testing}
In this section, we proceed to derive the asymptotic distribution of the test statistic $T_{\mu, t, c}$ under $H_0$. From now on, we will use $h=\lambda^{1/(2m)}$ without repeating its definition.
\begin{theorem}\label{asymp:distribution}
Suppose that Condition (\textbf{B}) holds, and it holds that $h\to 0$, $ch\to \infty$, $b\gg\log_2\left(\sqrt{(nh^{1/2}+n(ch)^{-1})\T_n}\right)$ and  $E([\widetilde{n}^{-1}\sum_{j=1}^{\widetilde{n}}Q(\epsilon_j)]^4)=O(c^2n^{-2})$ as $n,c\to\infty$. Then under $H_0$, it follows that
\begin{equation}\label{clt:T:t}
\frac{cT_{\mu,t,c}-\textrm{trace}(A)\widehat{\tau}_k^2}{s_c\widehat{\tau}_k^2}\overset{d}{\longrightarrow}N(0,1),
\,\,\textrm{as $n, c\rightarrow\infty$},
\end{equation}
where $T_{\mu,t,c}$, $A$, $\widehat{\tau}_k^2$ and $s_c$ are as defined in Section~\ref{sec:test}.
\end{theorem}

Theorem \ref{asymp:distribution} states that under some regularity conditions, the null distribution of the nonparametric test statistic $T_{\mu,t, c}$ for $H_0$ in~\eqref{H0} is asymptotically normal. The proof relies on Stein's exchangeable pair method and is given in the Appendix.

 We remark that the conditions in Theorem \ref{asymp:distribution} are rather mild. Specifically, the first condition $h\to 0$ requires the tuning parameter $\lambda$ to shrink to zero and the second condition $ch \to \infty$ implies the number of quantized data, ie., $c$, should be sufficiently large. The only condition that needs more discussion is the last condition $E([\widetilde{n}^{-1}\sum_{j=1}^{\widetilde{n}}Q(\epsilon_j)]^4)=O(c^2n^{-2})$, which involves jointly controlling the moment of $\epsilon_i$'s and the first-stage quantizer $Q(\cdot)$.  Proposition \ref{tau_k^2:lemma} below provides a sufficient condition to for this assumption.

\begin{proposition}\label{tau_k^2:lemma}
Suppose that  Condition (\textbf{B}) holds. If $E(\epsilon_1^4)=O(nc^{-1})$, $C_k(t) = o(1)$ and
$\mu_j^4P(\sigma\epsilon_1\in R_j(t))=O(nc^{-1})$ for $j=1$ and $k$, then it follows that $E([\widetilde{n}^{-1}\sum_{j=1}^{\widetilde{n}}Q(\epsilon_j)]^4)=O(c^2n^{-2})$.
\end{proposition}

Using Theorem~\ref{asymp:distribution} and Proposition~\ref{tau_k^2:lemma}, the validity of the proposed nonparametric testing procedure requires the quantized sample size $c$ to be sufficiently large, in particular, $ch=c\lambda^{1/(2m)}\to\infty$.  Recall that the proposed quantization scheme in Section~\ref{sec:quan} requires a total bits budget $B= cb$ with $b=\ceil{\log_2(k)}$. As a result, for Theorem~\ref{asymp:distribution} to hold, the required bits budget $B\gg c\log_2\left(\sqrt{(nh^{1/2}+n(ch)^{-1})\T_n}\right)$, for which the lower bound is determined by the tuning parameter $\lambda$ (or $h$). In the next subsection, we shall investigate the impacts of $\lambda$ on the asymptotic testing power against local alternatives, which can be used to study optimal asymptotic power achievable with a given bits budget $B$. For example, we shall show that to achieve the minimax rate of testing, one needs $B \gtrsim n^{\frac{3}{4m+1}}\log_2\left(n^{\frac{2m}{4m+1}}\sqrt{\T_n}\right)$.

\subsection{Asymptotic Power of the Nonparametric Test}
\label{sec:power:fixm}
We now proceed to examine the asymptotic power of the proposed nonparametric test. For a fixed constant $\rho>0$, let $S_\rho^m(\mathbb{I})=\{f\in S^m(\mathbb{I}): J(f, f)\le\rho^2\}$ be the $\rho$-ball in the periodic Sobolev space with a radius $\rho$. We consider the following alternative hypothesis
\begin{equation}\label{eq:h1}
H_1: g_0\in S_\rho^m(\mathbb{I})\backslash\{0\}.
\end{equation}

Based on the definition of the quantized data $z_i$ in \eqref{class:smoothing:spline:quantized}, its unquantized counterpart can be defined as $\widetilde{y}_i=\frac{1}{\widetilde n}\sum_{j=(i-1)\widetilde n+1}^{i\widetilde n}y_j$ for $i=1,\cdots,c$. Under $H_1$, one has that $E\widetilde{y}_i=\frac{c}{n}\sum_{j=(i-1)\widetilde{n}+1}^{i\widetilde{n}}g_0(j/n)$ for $i=1,\cdots,c$. To facilitate our theoretical investigation, we introduce the following function

\begin{equation}\label{eq:f0}
f(x)=\frac{1}{2\Delta}\int_{\max(x-\Delta, 0)}^{\min(x+\Delta, 1)}g_0(s)ds, \text{ where } x\in\mathbb{I}, \text{ and }\Delta=\frac{1}{c}.
\end{equation}
It is straightforward to show that $\max_{1\le i\le c}|f(i/c)-E\widetilde{y}_i|=O(n^{-1})$ and that as $\Delta\to 0$, $\sup_{x\in\mathbb{I}}|g_0(x)-f(x)|\to 0$. Theorem \ref{power:thm} below states that, under some regularity conditions,
our proposed nonparametric test can achieve arbitrary high power provided that $H_0$ and $H_1$
are sufficiently separated.
\begin{theorem}\label{power:thm}
Suppose that Condition (\textbf{B}) holds. If it holds that $h\to 0$, $ch\to \infty$, $b\gg\log_2\left(\sqrt{(nh^{1/2}+n(ch)^{-1})\T_n}\right)$, and  $E([\widetilde{n}^{-1}\sum_{j=1}^{\widetilde{n}}Q(\epsilon_j)]^4)=O(c^2n^{-2})$, then for any $\eta>0$, there exists positive constants $C_\eta$ and $N_\eta$  such that
for any $c\ge N_\eta$,
\[
\inf_{\substack{g\in S_\rho^m(\mathbb{I})\\
\|g\|_c\ge C_\eta\delta_{n,c,\lambda}}} P(\textrm{reject $H_0$}|\textrm{$H_1$ is true})\ge1-\eta,
\]
where
$\delta_{n,c,\lambda}=\sqrt{(nh^{1/2})^{-1}+\lambda+(nch^2)^{-1}}$ and $\|g\|_c=\sqrt{\sum_{i=1}^c f^2(i/c)/c}$ with function $f(\cdot)$ as defined in~\eqref{eq:f0}.
\end{theorem}

The separation rate $\delta_{n,c,\lambda}$ represents the smallest rate of deviation from the $H_0$ that can be consistently detected by the proposed test statistic~\eqref{L2T}, given sufficiently large $n$ and $c$. The first part of  $\delta_{n,c,\lambda}$, namely, $(nh^{1/2})^{-1}+\lambda$, coincides with the separation rate of the classical spline-based nonparametric test using original observations without quantization, see, e.g.,  \citet{SC13, CS15, SC15, SC17}. The remaining part of $\delta_{n,c,\lambda}$, namely, $(nch^2)^{-1}$, is an  additional term due to the two-stage quantization errors.   For a given $n$ and $c$, the separation rate $\delta_{n,c,\lambda}$ can be minimized by choosing an appropriate value of the tuning parameter $\lambda$, subject to the constraint $c\lambda^{1/2m} \to\infty$. Specifically, by some straightforward algebra,  one can show that

\begin{equation}
\label{rate1}
  \inf_{\lambda\gg c^{-2m}}\delta_{n,c,\lambda}=\begin{cases}
n^{-\frac{2m}{4m+1}},& \text{ if } c\gtrsim n^{\frac{3}{4m+1}} \text{ with } \lambda\asymp n^{-\frac{4m}{4m+1}};\\
(nc)^{-\frac{m}{2(m+1)}},& \text{ if }  n^{\frac{1}{2m+1}} \ll c\lesssim n^{\frac{3}{4m+1}} \text{ with } \lambda \asymp (nc)^{-\frac{m}{m+1}};\\
\lambda^{1/2},& \text{ if }  c\lesssim n^{\frac{1}{2m+1}} \text{ with } \lambda \gg c^{-2m}.\\
\end{cases}
\end{equation}

Recall that the total bits needed for the proposed quantization scheme in Section~\ref{sec:quan} is $B=cb$, for which Theorem~\ref{power:thm} requires that $h\to 0$ and $b\gg\log_2\left(\sqrt{(nh^{1/2}+n(ch)^{-1})\T_n}\right)$. By plugging the optimal smoothing parameter back to the lower bound of $b$, we have the following

\begin{equation}
\label{rate2}
  \inf_{\lambda\gg c^{-2m}}\delta_{n,c,\lambda}=\begin{cases}
n^{-\frac{2m}{4m+1}},& \text{ if } c\gtrsim n^{\frac{3}{4m+1}}, b\gg \log_2\left(n^{\frac{2m}{4m+1}}\sqrt{\T_n}\right);\\
(nc)^{-\frac{m}{2(m+1)}},& \text{ if }  n^{\frac{1}{2m+1}} \ll c\lesssim n^{\frac{3}{4m+1}},b\gg \log_2\left(n^{\frac{2m+3}{4(m+1)}}c^{-\frac{2m+1}{4(m+1)}}\sqrt{\T_n}\right);\\
\lambda^{1/2},& \text{ if }  c\lesssim n^{\frac{1}{2m+1}},b\gg \log_2\left(\sqrt{n\T_n}\right)  \text{ with } \lambda \gg c^{-2m}.\\
\end{cases}
\end{equation}

 From~\eqref{rate2}, we can see that when $B$ is sufficiently large, i.e., $B \gtrsim n^{\frac{3}{4m+1}}\log_2\left(n^{\frac{2m}{4m+1}}\sqrt{\T_n}\right)$, the minimal separation rate $n^{-\frac{2m}{4m+1}}$ achieves the minimax rate of testing~\citep{SC13, SC17, shang:ejs:2020}, implying lossless asymptotic testing power using only quantized samples. In this case, the minimal number of bits for each data point, i.e., $b$, does not depend on $c$ but is determined by the smoothness of the function and the tail bound $\T_n$ of the error distribution. When $B$ is between $ n^{\frac{1}{2m+1}}\log_2\left(\sqrt{n\T_n}\right)$ and $n^{\frac{3}{4m+1}} \log_2\left(n^{\frac{2m}{4m+1}}\sqrt{\T_n}\right)$, the minimax rate of testing is no longer achievable, but the minimal separation rate still decays polynomially as the original sample size $n$ increases. Furthermore, in this intermediate phase of $B$, the lower bound of $b$ decreases as $c$ increases, implying that increasing $c$ rather than $b$ when allocating the total bits budget $B$ will more effectively improve the testing power. Finally, when $B$ is less than  $n^{\frac{1}{2m+1}}\log_2 \left(\sqrt{n\T_n}\right)$, the asymptotic lower bound for the minimal rate of separation is (roughly) of the order $c^{-m}$ with $c=B/b$, the number of quantized measurements that can be transmitted or stored, provided that $b\gg \log_2\left(\sqrt{n\T_n}\right)$. 

\section{Extensions}
\label{sec:ext}
Our prior investigations in Section~\ref{sec:asymptotics} assume that the hypothesized function $g_*(\cdot)$ in (\ref{H0}) and the order $m$ of the periodic Sobolev space $S^m(\mathbb{I})$ are both known. In reality, it might be interesting to test other hypotheses, e.g., whether $g_0$ has a parametric expression such as a linear function. Meanwhile, the order $m$ is often unknown. 
We will extend the prior works to such settings.

\subsection{Nonparametric Testing for Linearity of $g_0(\cdot)$} \label{sec: linearity}
In some applications, we are interested in testing whether $g_0(\cdot)$ resides in a parametric family. In this section, as an illustrative example, we consider testing the linearity of $g_0(\cdot)$:
\begin{equation}\label{eq:H0:linear}
H_0^{\textrm{linear}}: g_0 \in \mathcal{L}(\mathbb{I}) \textrm{ vs. } H_1^{\textrm{linear}}: g_0\in S_\rho^m(\mathbb{I})\backslash\{\mathcal{L}(\mathbb{I})\},
\end{equation}
where $\mathcal{L}(\mathbb{I})$ denotes the class of liner functions over $\mathbb{I}:=[0, 1]$. Testing the hypothesis that $g_0(\cdot)$ belongs to other parametric families governed by a finite number of parameters can be conducted in the same way with minor modifications.

To test~\eqref{eq:H0:linear}, we first obtain the least-square estimator $\widehat{g}(x)$, $x\in\mathbb{I}$ based on $Q(y_j)$'s, i.e., $\widehat{g}(x) = \argmin _{g\in \mathcal{L}(\mathbb{I})} \sum_{i=1}^n \left[g(x_i) - Q(y_i)\right]^2$. Subsequently, we define the new data as $y^{\textrm{linear}} = (Q(y_1) - \widehat{y}_1, \ldots, Q(y_n) - \widehat{y}_n)^T$, where $\widehat{y}_i = \widehat{g}(i/n)$. By applying the two-stage quantization Algorithm~\ref{alg:distributed} to $y^{\textrm{linear}}$, we can then obtain the quantized data $ z_\textrm{linear} = ( z_{\textrm{linear},1}, \ldots,  z_{\textrm{linear}, c})^T$. Following the same estimation procedure in Section~\ref{sec:est}, we can obtain a spline estimator $\widehat{g}^{\textrm{B}}_{\textrm{linear},\mu,t,c}$ based on the quantized data $z_\textrm{linear}$. 


The resulting test statistic is then defined as $T_{\textrm{linear},\mu,t,c} = \|\widehat{g}_{\textrm{linear},\mu,t,c}^{\textrm{B}}\|^2$, whose limiting distribution under  $H^\textrm{linear}_0$ is given by the following theorem.

\begin{theorem} \label{corollary linear}
Suppose that Condition (\textbf{B}) holds. If as $n,c\to\infty$, it holds that $h\to 0$, $ch\to \infty$, $b\gg\log_2\left(\sqrt{(nh^{1/2}+n(ch)^{-1})\T_n}\right)$, and $E([\widetilde{n}^{-1}\sum_{j=1}^{\widetilde{n}}Q(\epsilon_j)]^4)=O(c^2n^{-2})$, then under $H^\textrm{linear}_0$, one has that
\begin{equation}\label{clt:T:t1}
\frac{cT_{\textrm{linear},\mu,t,c}-\textrm{trace}(A)\widehat\tau_k^2}{s_c\widehat\tau_k^2}\overset{d}{\longrightarrow}N(0,1),
\,\,\textrm{as $n, c\rightarrow\infty$,}
\end{equation}
where $T_{\mu,t,c}$, $A$, $\widehat{\tau}_k^2$ and $s_c$ are as defined in Section~\ref{sec:test} but based on $y^{\textrm{linear}}$.
\end{theorem}

Theorem~\ref{corollary linear} is an immediate extension of Theorem~\ref{asymp:distribution} to testing the linearity of $g_0(\cdot)$ using only quantized samples, indicating that the proposed nonparametric linearity test is valid under mild conditions. To investigate the power of the proposed linearity test against the alternative $H_1^{\textrm{linear}}$, we define the distance between $g_0(\cdot)$ and the linear function space $\mathcal{L}(\mathbb{I})$ as $\|g_0 - \mathcal{P}_{\mathcal{L}(\mathbb{I})}(g_0)\|$, where $\mathcal{P}_{\mathcal{L}(\mathbb{I})}(g_0) = \argmin_{f\in \mathcal{L}(\mathbb{I})} \|g_0 - f\|^2$ is the projection of $g_0(\cdot)$ to $\mathcal{L}(\mathbb{I})$. The magnitude of $\|g_0 - \mathcal{P}_{\mathcal{L}(\mathbb{I})}(g_0)\|$ characterizes how far the true function $g_0(\cdot)$ deviates from any linear function in $\mathcal{L}(\mathbb{I})$. Note that under null hypothesis $H_0^{\textrm{linear}}$, one has that $\|g_0 - \mathcal{P}_{\mathcal{L}(\mathbb{I})}(g_0)\| = 0$.

The following theorem describes the asymptotic power of the proposed nonparametric linearity test.

\begin{theorem}\label{power:linearity}
Suppose that Condition (\textbf{B}) hold. If as $n,c\to\infty$, it holds that $h\to 0$, $ch\to \infty$, $b\gg\log_2\left(\sqrt{(nh^{1/2}+n(ch)^{-1})\T_n}\right)$, and $E([\widetilde{n}^{-1}\sum_{j=1}^{\widetilde{n}}Q(\epsilon_j)]^4)=O(c^2n^{-2})$, then for any $\eta>0$, there exists positive constants $C_\eta$ and $N_\eta$  such that
for any $c\ge N_\eta$,

\[
\inf_{\substack{g\in S_\rho^m(\mathbb{I})\\
\|g-\mathcal{P}_{\mathcal{L}(\mathbb{I})}(g)\|_c\ge C_\eta\delta_{n,c,\lambda}^{\textrm{linear}}}} P(\textrm{reject $H_0^{\textrm{linear}}$}|\textrm{$H_1^{\textrm{linear}}$ is true})\ge1-\eta,
\]
where $\delta_{n,c,\lambda}^{\textrm{linear}}=\sqrt{(nh^{1/2})^{-1}+\lambda+(nch^2)^{-1}}$ and $\|g\|_c=\sqrt{\sum_{i=1}^c f^2(i/c)/c}$ with function $f(\cdot)$ as defined in~\eqref{eq:f0}.
\end{theorem}

Based on Theorem~\ref{power:linearity}, we can see that for a given quantized sample of size $c$, the same separation rate for testing can be achieved by the proposed nonparametric linearity test as described in~\eqref{rate1}. Furthermore, the proofs of Theorems~\ref{corollary linear}-\ref{power:linearity} are similar to those of  Theorem~\ref{asymp:distribution} and Theorem~\ref{power:thm} by recognizing the fact that the least square estimator $\widehat{g}(\cdot)$ satisfies that $\sup_{x\in\mathbb{I}}|\widehat{g}(x)-g_0(x)|=O_p(n^{-1/2})$, whose impact is negligible for a nonparametric spline estimator. It is therefore trivial to extend Theorems~\ref{corollary linear}-\ref{power:linearity} to testing whether $g_0(\cdot)$ resides in other parametric families as long as an uniformly root-n consistent parametric estimator $\widehat{g}(\cdot)$ is available.

\subsection{Adaptive Nonparametric Test When $m$ is Unknown} \label{sec: adaptive}
From~\eqref{rate1}, we can see that the power of the proposed nonparametric test depends crucially on the order $m$ of the periodic Sobolev space where the underlying true function $g_0(\cdot)$ resides. However, the order $m$ may be unknown in practice.
One popular strategy is to set $m=2$ regardless of the underlying truth, which may lead to sub-optimal testing power. In this section, we construct an optimal adaptive nonparametric testing procedure based on quantized samples that doesn't require $m$.

Let $m_*$ denote the unknown true order of the Sobolev space to which $g_0(\cdot)$ belongs, and assume that $m_*$ is an integer between two known integers $m_l$ and $m_u$. For instance, one can set $m_l=1$ and $m_u=\textrm{poly}(\log{n})$ so that, as $n$ diverges, $m_*$ is guaranteed to belong to $[m_l,m_u]$. For any given integer $m$, we can calculate the test statistics $T_m := T_{\mu, t, c}$ defined in (\ref{L2T})  with the tuning parameter $\lambda_m=a_n^{2m}n^{-4m/(4m+1)}\log(m_u)^{2m/(4m+1)}$ where $a_n$ may depend on $n$ but is free of $m$. We remark that the upper bound $m_u$ may be slowly diverging as $n\to\infty$. Our adaptive nonparametric testing procedure is summarized as follows.\\

Step 1. For any $m_l \leq m \leq m_u \rightarrow \infty$, calculate the standardized testing statistic
\[
\xi_m = \frac{cT_{m}-\textrm{trace}(A_m)\widehat\tau_k^2}{s_{c,m}\widehat\tau_k^2},
\]
\hskip 5em where $T_m$, $A_m$, $\widehat{\tau}_k^2$ and $s_{c,m}$ are as defined in Section~\ref{sec:test}.\\

Step 2. Calculate the maximum of $\xi_m$'s, i.e., $\xi_{\textrm{max}} = \max_{m_l \leq m \leq m_u} \xi_m$.\\

Step 3. Standardize $\xi_{\textrm{max}}$ as following
\[\xi_* = C_n(\xi_{\textrm{max}}-C_n),
\]
\hspace{5.5em} where $C_n$ satisfies $2\pi C_n^2\exp(C_n^2)=m_u^2$.\\

For the validity of the proposed adaptive nonparametric test, we assume that the following Condition (\textbf{C}) holds.
\begin{align*}
\textrm{Condition (\textbf{C})}: &\text{ (a) } m_u \lesssim \log^{d_0}(n) \text{ for some }d_0 \in (0, 1/2), \quad \text{(b) }a_nn^{-2/(4m_u+1)}[\log(n)]^2 \rightarrow 0,\\ \, &\text{ and (c) }n^{2/(4m_l+1)}\log (n)/(ca_n) \rightarrow 0, \text{ as } n,c\to\infty.
\end{align*}
\textrm{Condition (\textbf{C})} requires that the searching range for $m$ can not be too large by imposing a slowly diverging uppper bound on $m_u$. In addition, the total number of quantized samples, i.e., $c$, that need to be transmitted or stored can not be too small compared to $n$, and is jointly determined by $m_l,m_u$ and the tuning parameter $a_n$. These conditions are rather mild and have been used in the literature, see, e.g., \cite{shang:colt2019,LSC18}. The following theorem describes the asymptotic behavior of $\xi_*$ under $H_0$.
\begin{theorem}\label{adaptive:size}
Suppose that both Conditions (\textbf{B}) and (\textbf{C}) hold, $E([\widetilde{n}^{-1}\sum_{j=1}^{\widetilde{n}}Q(\epsilon_j)]^4)=O(c^2n^{-2})$, and that
\[
b\gg \max\left\{\log_2\left(\sqrt{(nh_{m_l}^{1/2}+n(ch_{m_l})^{-1})\T_n}\right), \log_2\left(\sqrt{(nh_{m_u}^{1/2}+n(ch_{m_u})^{-1})\T_n}\right)\right\}.
\]
Then, under $H_0$ given in~\eqref{H0}, for any $\alpha \in (0, 1)$, it holds that
\begin{equation}
P(\xi_*\leq q_\alpha) \rightarrow 1 - \alpha, \,\,\textrm{as $n,c \rightarrow\infty$},
\end{equation}
where $q_\alpha = -\log(-\log(1-\alpha))$.
\end{theorem}

The intuition behind Theorem~\ref{adaptive:size} is straightforward: under $H_0$, the limiting distribution of each $\xi_m$ is normal, which suggests that the asymptotic distribution of the maxima $\xi_*$ should be close to the extreme value distribution. We use the techniques developed in~\cite{adaptiveapproximation} to formalize the proof.

Next, we investigate the asymptotic power of the proposed adaptive nonparametric test under the alternative $H_1: g_0\in S_\rho^{m^*}(\mathbb{I})\backslash\{\mathcal{L}(\mathbb{I})\}.
$
\begin{theorem}\label{adaptive:power}
Suppose that both Conditions (\textbf{B}) and (\textbf{C}) hold, $E([\widetilde{n}^{-1}\sum_{j=1}^{\widetilde{n}}Q(\epsilon_j)]^4)=O(c^2n^{-2})$, and that 
\[
b\gg \max\left\{\log_2\left(\sqrt{(nh_{m_l}^{1/2}+n(ch_{m_l})^{-1})\T_n}\right), \log_2\left(\sqrt{(nh_{m_u}^{1/2}+n(ch_{m_u})^{-1})\T_n}\right)\right\}.
\]
Then, for any $\eta>0$, there exists positive constants $C_\eta$ and $N_\eta$ such that
for any $c\ge N_\eta$,
\[
\inf_{\substack{g\in S_\rho^{m_*}(\mathbb{I})\\
\|g\|_c\ge C_\eta\delta_{n, c, a_n}}} P(\textrm{reject $H_0$}|\textrm{$H_1$ is true})\ge1-\eta,
\]
where $\delta_{n,c,a_n}=n^{-\frac{2m_*}{4m_*+1}}[\log(m_u)]^{\frac{m_*}{4m_*+1}}\sqrt{a_n^{-\frac{1}{2}}+ c^{-1}a_n^{-2}n^{\frac{3}{4m_*+1}}[\log(m_u)]^{-\frac{2(m_*+1)}{4m_*+1}} + a_n^{2m_*}}$ and $\|g\|_c=\sqrt{\sum_{i=1}^c f^2(i/c)/c}$ with function $f(\cdot)$ as defined in~\eqref{eq:f0}.
\end{theorem}

Based on the form of separation rate $\delta_{n,c,a_n}$ obtained in Theorem~\ref{adaptive:power}, it is straightforward to show that the minimal separation rate is obtained when $a_n=a_0$ for some constant $a_0>0$,  provided that $$c\gg \max\{n^{2/(4m_l+1)}\log(n),  n^{3/(4m_*+1)}\}[\log (m_u)]^{-1/(4m_*+1)}$$ so that Condition~(\textbf{C}) is met and the second term inside the square-root part of $\delta_{n,c,a_n}$ is negligible. Specifically, if $c\gg \max\{n^{2/(4m_l+1)}\log(n),  n^{3/(4m_*+1)}\}[\log (m_u)]^{-1/(4m_*+1)}$, one has that
\begin{equation}
\label{opt-ada}
\inf_{a_n>0}\delta_{n,c,a_n}\asymp n^{-\frac{2m_*}{4m_*+1}}\left[\log(m_u)\right]^{\frac{m_*}{4m_*+1}}.
\end{equation}

The minimal separation rate~\eqref{opt-ada} is the same as the one obtained in~\cite{shang:colt2019,LSC18} and is minimax for the adaptive nonparametric test. This suggests that with the quantized samples, the proposed adaptive test can still achieve the optimal testing power if the bits budget satisfies $$B \gg \max\{n^{2/(4m_l+1)}\log(n),  n^{3/(4m_*+1)}\}[\log (m_u)]^{-1/(4m_*+1)}\log(n),$$ and we take $b=\log_2\left( \sqrt{n\T_n/\sigma^2}\right)\asymp \log(n)$ as suggested in Section~\ref{sec:select_c_k}. Compared to the minimax rate of testing when $m_*$ is known, which is given in~\eqref{rate1}, the minimal separation rate~\eqref{opt-ada} is only inflated by a factor of $[\log (m_u)]^{4m_*/(4m_*+1)}$. This is the price to pay for searching $m$ over $m_l\le m\le m_u$. Furthermore, we wish to remark that the lower bound of the bits budget $B$ depends not only on the true order $m_*$ but also on the smallest guess of the order, i.e., $m_l$. This can be interpreted by the fact that $\xi_{m_l}$ in Step 1 of the adaptive test is constructed based on an under-smoothed spline estimator, which may have a larger order of estimation bias. In practice, it is convenient to set $m_l=1$ as suggested by~\cite{LSC18}. However, a more accurate guess of $m_l$ may lead to a smaller bits budget $B$ required to achieve the minimax rate of testing.

\section{Simulation Studies}\label{sec:simulation}
In this section, we evaluate the finite sample performance of the proposed methods through a set of simulation studies. For all simulation settings except for Section~\ref{sim:linearity}, the data are generated from the following model
\begin{equation}
\label{eq:simdata}
    y_i = r\beta_{3, 2}(x_i) + \epsilon_i,\quad \text{ with } x_i=\frac{i}{n}, i=1,\cdots,n,
\end{equation}
where $\beta_{3,2}(\cdot)$ is the density function of the beta distribution with parameters $3$ and $2$, $\epsilon_i$'s are independent random errors. Two types of errors were considered:
(1) $\epsilon\sim N(0,1)$; (2) $\epsilon\sim N(0,1.5^2)$. We consider $r$ from $0$ to $1$, and various sample sizes $n$. In particular, $r=0$ is used to examine the empirical size of the proposed test  under $H_0$, and other values of $r$ are used to check the empirical powers against alternatives. The target significance level was chosen as $\alpha=0.1$.

For all simulation studies, we consider the uniform quantization scheme outlined in Section~\ref{sec:quan}. Specifically, for the data quantization step, for a given bits budget $B$, we choose $c, k$ following the approach suggested in Section \ref{sec:select_c_k} with a $\T_n/\sigma^2=2.5\log(n)$. For each simulation, the quantization ranges $t_1, t_{k-1}$ are defined as $t_1 = \mu_0 - \sqrt{ 2.5\sigma^2\log(n)}, t_{k-1} = \mu_0 + \sqrt{ 2.5\sigma^2\log(n)}$, where $\mu_0=\int_0^1g_0(x)dx$ with $g_0(\cdot)$ being the regression function in model (\ref{regression:model}). The use of $\mu_0$ and $\sigma$ is of limited importance and can be replaced with any reasonable alternatives such as setting $\mu_0=0$ or using estimates based on historical data. Summary statistics from each simulation setting were based on $1000$ independent simulation runs. Except for Section \ref{sim:adaptive}, we considered periodic Sobolev space of order $m=2$ with kernel function $K(x, y)=-B_{4}(|x-y|)/24$, where $B_4$ is the Bernoulli polynomial of order $4$. The tuning parameter $\lambda$ was set as $\lambda=\widehat{\lambda}_{\textrm{GCV}}/\log(c)$ with $\widehat{\lambda}_{\textrm{GCV}}$ being picked by GCV.

\subsection{Estimation Performance of $\widehat{g}^{\textrm{B}}_{\mu,t,c}(\cdot)$}\label{sim:est}
In this section, we first evaluate the estimation performance of the spline estimator $\widehat{g}^{\textrm{B}}_{\mu, t, c}(\cdot)$ defined in (\ref{class:smoothing:spline:quantized}) that is based on only quantized samples. We generated data from model~\eqref{eq:simdata} with  $r=0.5, 1$ and sample sizes $n=1000, 2000, 3000, 5000, 10000$. For each $n$, we gradually increase the bits budget $B$ from $30$ to $1000$.   The estimation accuracy was evaluated by the mean squared errors $({\rm RMSE})$ defined as $\|\widehat{g}_{\mu,t,c}^{\textrm{B}}-g_0\|$. The simulation results were summarized in Figure \ref{figure:mse}, which suggests that the MSEs decrease as $n$ increases in all considered settings. Moreover, as $B$ increases, the MSEs first decreases rapidly at the beginning and then stabilize at some levels. This observation is consistent with our theoretical results established in Section~\ref{sec:est:rate}, which state that increasing $B$ (or equivalently, $c$ and $b$) will diminish the impact of information loss due to the data averaging and data quantization, and as a result $\widehat{g}^{\textrm{B}}_{\mu,t,c}(\cdot)$ becomes more accurate. Furthermore, we can also observe after $B$ exceeds a certain threshold, the MSEs of $\widehat{g}^{\textrm{B}}_{\mu,t,c}(\cdot)$ stabilize, which supports the findings in Corollary~\ref{coro2}. Specifically, when $B$ is sufficiently large, the  MSEs of $\widehat{g}^{\textrm{B}}_{\mu,t,c}(\cdot)$ reaches the estimation error lower bound of the classical spline estimator based on the complete data.

\begin{figure}[ht!]
  \centering
  \includegraphics[width=6.3in]{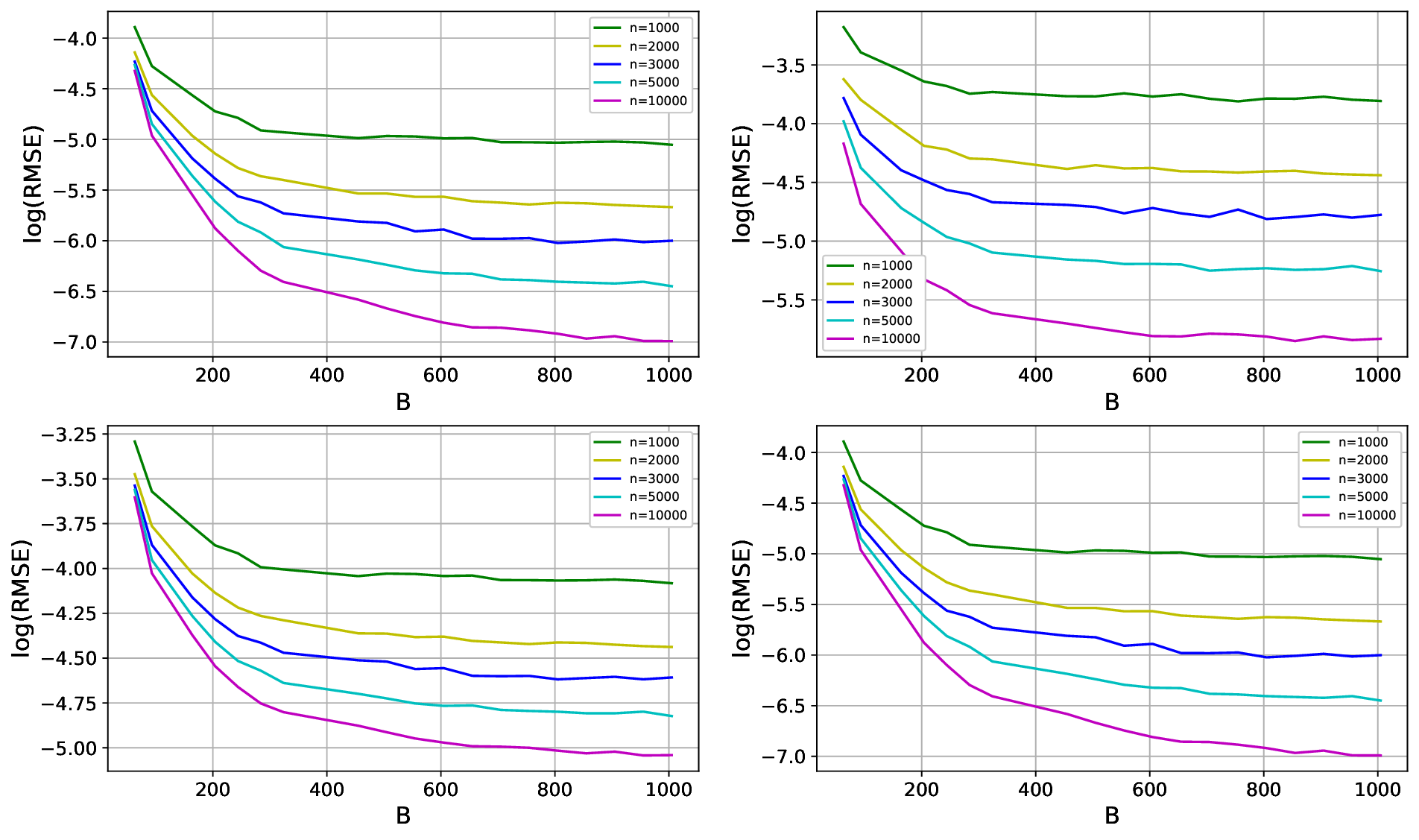}
  \caption{\it MSEs of spline estimators based on quantized data: $r=0.5$ for left 2 panels and $r=1$ for right 2 panels;
  $\epsilon\sim N(0, 1)$ for top 2 panels and $\epsilon\sim N(0, 1.5^2)$ for bottom 2 panels.}
  \label{figure:mse}
\end{figure}


\subsection{Nonparametric Test with $g_*(\cdot)\equiv 0$ and $m=2$}\label{sim:test}
In this section, we investigate the empirical sizes and powers of the nonparametric test proposed in Section~\ref{sec:test}, when $g_*(\cdot)\equiv0$ in the null hypothesis (\ref{H0}) and $m=2$ treated as  known. The data was generated from the model~\eqref{eq:simdata} with various $r$ and sample sizes $n$.

Figure \ref{figure:power1} reports the empirical sizes of the proposed nonparametric test  when $r=0$ and the empirical powers when $r>0$, respectively.
Specifically, in all case scenarios, the empirical sizes of the proposed test are close to the target nominal level $0.1$ as the sample size $n$ increases.  When either $r$ or $n$ increases, we observe that the empirical powers of the proposed test gradually approach one, which suggests that the proposed testing procedure is consistent for the alternative hypothesis that has a sufficiently large deviation (relative to the sample size $n$) from the $H_0$. Furthermore, after the bits budget $B$ exceeds a certain threshold, the empirical powers of the proposed nonparametric test are rather close to each other, which supports our theoretical findings in Section~\ref{sec:power:fixm}.


\begin{figure}[ht!]
  \centering
  \includegraphics[width=6.3 in]{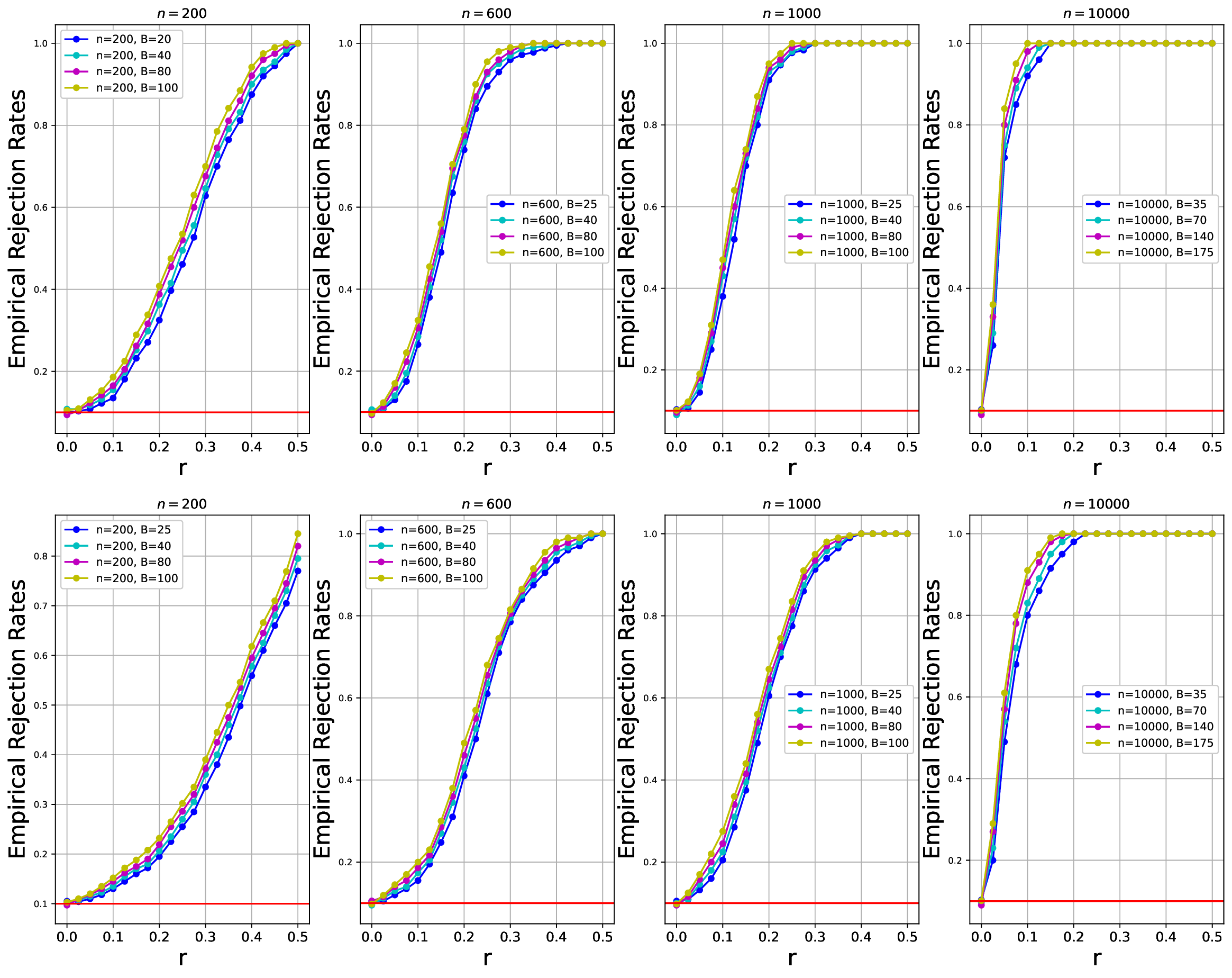}
 \vspace{-2em}
   \caption{\it Empirical rejection rates at $0.1$ significance level: $\epsilon\sim N(0, 1)$ for the top 4 panels; $N(0, 1.5^2)$ for the bottom 4 panels.}
  \label{figure:power1}
\end{figure}

\subsection{Adaptive Nonparametric Test with an Unknown $m$}\label{sim:adaptive}
In this section, we investigate the validity and the empirical power of the adaptive nonparametric test proposed in Section \ref{sec: adaptive}, for which the order parameter $m$ is searched from $m_l=1$ to $m_u=\sqrt{\log n}$. Figure \ref{figure:adaptive} shows the empirical rejection rates of the proposed nonparametric adaptive test at the $0.1$ significance level. We can observe that when $r=0$, the empirical rejection rates are rather close to the nominal level $0.1$.  For any given $r>0$, we can see that the empirical rejection rates increase as the sample size $n$ increases. For a fixed $n$, as $r$ increases, the empirical rejection rates increase steadily and eventually reach the $100\%$ when $n=600, n=1000$ and $n=10000$. Finally, as long as the bits budget $B$ exceeds a certain threshold, the empirical rejection rates are rather similar in most settings. All these observations are consistent with our theoretical findings in Theorem~\ref{adaptive:power}. Furthermore, the empirical rejection rates are smaller than the nonparametric test (non-adaptive) in Section \ref{sim:test} under the same setting, which is the price paid for adaptivity in $m$.

\begin{figure}[ht!]
\centering
  \includegraphics[width=6.3 in]{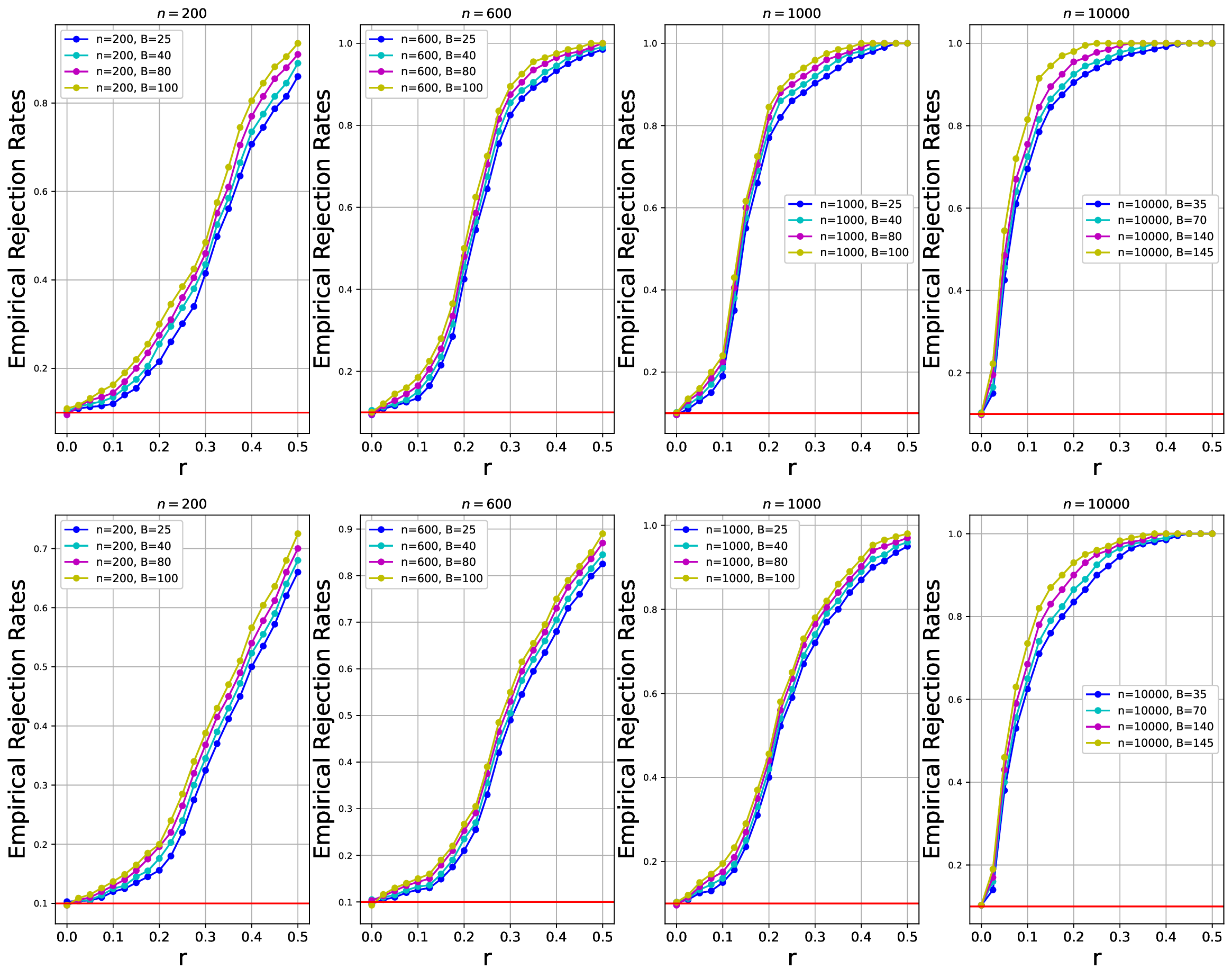}
       \caption{\it Empirical rejection rates of the nonparametric adaptive test at the $0.1$ significance level: $\epsilon\sim N(0, 1)$ for the top 4 panels; $N(0, 1.5^2)$ for the bottom 4 panels.}
  \label{figure:adaptive}
  \end{figure}

\subsection{Nonparametic Linearity Test with $m=2$}\label{sim:linearity}
In this section, we study the empirical performance of the proposed nonparametric linearity test. The data is generated from the following model
\[
y_i = r\beta_{3, 2}(x_i) + 3x_i + 2 + \epsilon_i,\quad \text{ with } x_i=\frac{i}{n}, i=1,\cdots,n,
\]
where $\beta_{3,2}(\cdot)$ is the density function of the beta distribution with parameters $3$ and $2$, and  $\epsilon_i$'s are independent random errors. Two types of errors were considered:
(1) $\epsilon\sim N(0,1)$; (2) $\epsilon\sim N(0,1.5^2)$. When $r=0$, the model satisfies the null hypothesis $H_0^{\textrm{linear}}$, and as $r$ increases, the departure from the linear model becomes increasingly larger.
\begin{figure}[ht!]
\centering
  \includegraphics[width=6.3 in]{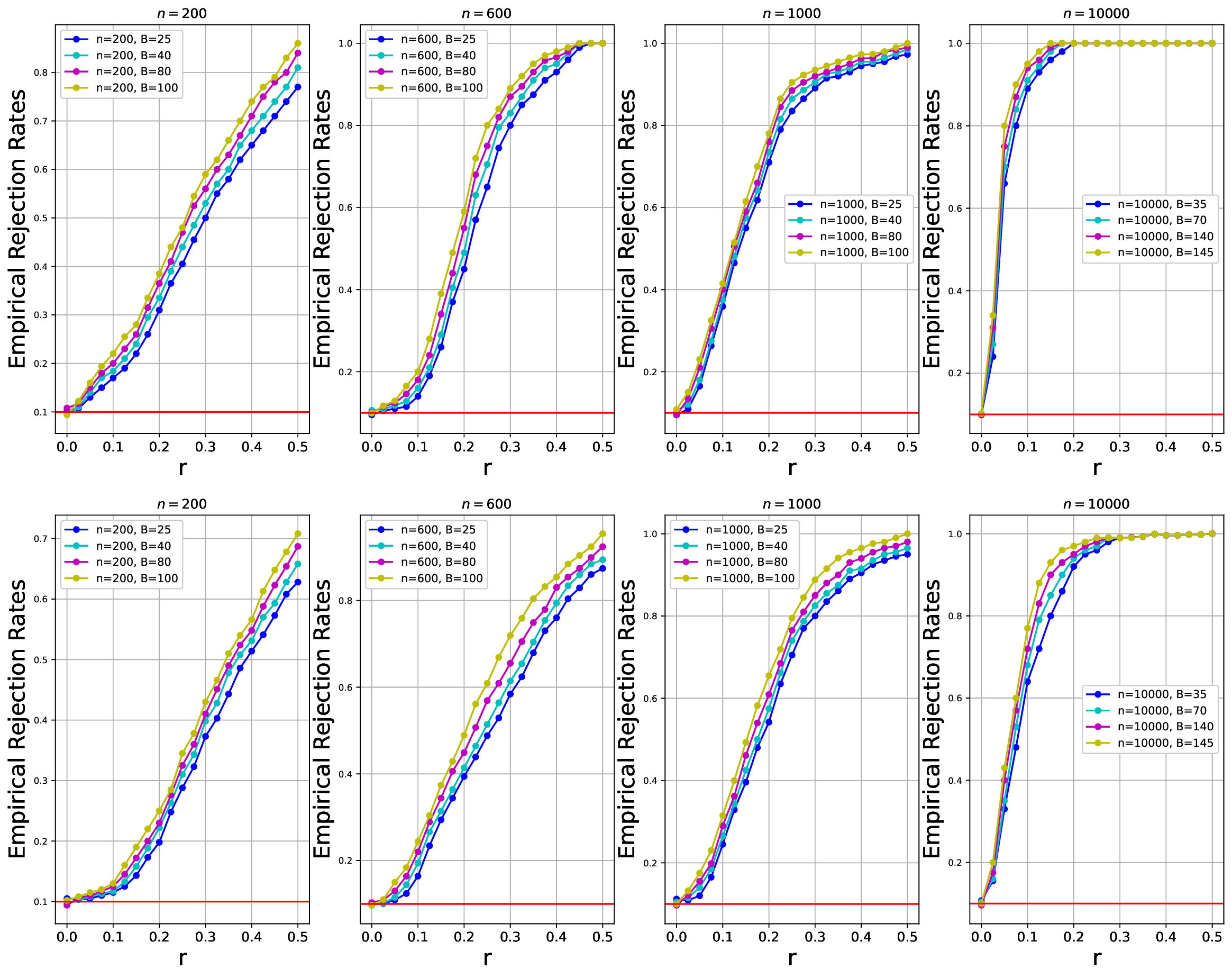}
     \caption{\it Empirical rejection rates of the nonparametric linearity test at $0.1$ significance level: $\epsilon\sim N(0, 1)$ for the top 4 panels; $N(0, 1.5^2)$ for the bottom 4 panels.}
  \label{figure:Linear}
  \end{figure}
  
 Figure \ref{figure:Linear} reports the empirical rejection rates of the nonparametric linearity test proposed in Section~\eqref{sec: linearity} at the significance level $0.1$ . It is straightforward to see that, when $r=0$, the empirical rejection rates are close to the nominal size, indicating the validity of the test asserted by Theorem~\ref{corollary linear}. For any given $r>0$, we can see that the empirical rejection rates increase as the sample size $n$ increases. For a fixed $n$, as $r$ increases, the empirical rejection rates increase steadily and eventually reach the $100\%$. Finally, as long as the bits budget $B$ exceeds a certain threshold, the empirical rejection rates are rather similar in most settings. All these observations are consistent with our theoretical findings in Theorem~\ref{power:linearity}.

\section{Real Data Analysis} \label{real data}
In this section, we apply the proposed methods to the Combined Cycle Power Plant Data \citep{RealData1,RealData2}, which can be downloaded at \url{http://https://archive.ics.uci.edu/ml/datasets/Combined+Cycle+Power+Plant}. The data set consists of $n=9568$ observations from a Combined Cycle Power Plant over 6 years (2006-2011).
The purpose of our analysis is to explore the relationship between the net hourly electrical energy output of the plant between three environmental factors:  temperature, ambient pressure, and relative humidity.

Figure \ref{figure:real data} displays the estimated curve based on $B$-bits quantizations ($B=35,70,140,175$) and full data, for which the periodic spline of order $m=2$ was used. For the quantization step, we choose $\T_n = 2.5\times\hat{\sigma}^2\log(n)$, where $\hat{\sigma}^2$ is the standard deviation of the observated data, and $c, k$ are determined by Section \ref{sec:select_c_k}. We can observe that
the spline estimator based on quantized data with $B=35$, i.e., the green curve, is rather different from the other curves in the two analyzes. When the bits budget $B$ increases to more than $70$, such differences quickly diminish. This observation demonstrates the effectiveness of the proposed $B$-bits quantization scheme.

\begin{figure}[ht!]
  \centering
  \includegraphics[width=6.5 in]{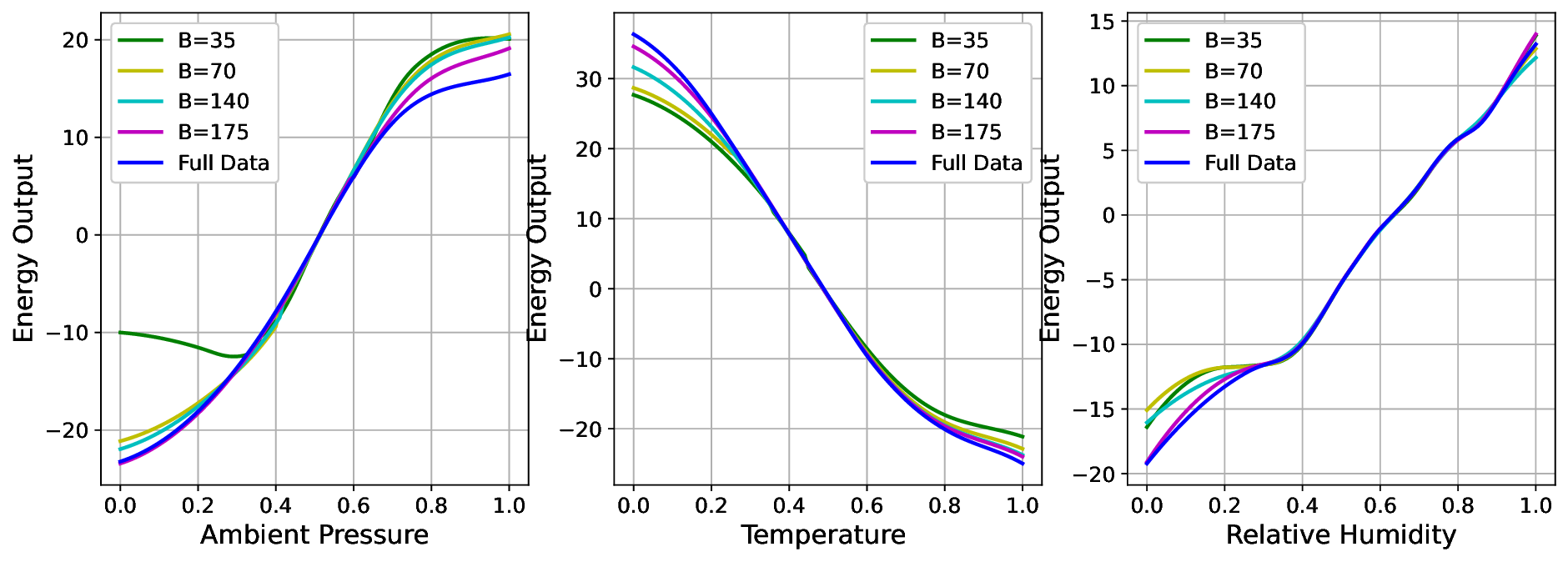}
   \caption{\it Spline estimators based on $B$-bits quantizations ($B$ = 35,70,140,175) and full data. Sample size is $n=9568$.}
  \label{figure:real data}
\end{figure}

{ Next, we conduct some hypothesis tests for the relationship between the net hourly
electrical energy output and other three environmental factors. The first test is to test whether there is an association between the energy output and three environmental factors. We consider both non-adaptive and adaptive nonparametric tests. For the non-adaptive nonparametric test, $m=2$ is used. The p-values are all close to zero, implying strong rejections of the null hypothesis. This is not surprising based on the shapes of the spline estimators illustrated in Figure~\ref{figure:real data}.

Next, there appears to be a strong linear association between relative humidity and the energy output in Figure~\ref{figure:real data}. Based on this conjecture, we proceed to test whether the associations between these three environmental factors and the energy output are linear or nonlinear, using the nonparametric linearity test proposed in Section~\ref{sec: linearity}. The p-values for the first two environmental factors, i.e., ambient pressure and temperature,  are both close to zero, indicating strong rejections of the null hypothesis. Figure~\ref{figure:p_val} illustrates the p-values of the nonparametric linearity test for the relationship between relative humidity and energy output as a function of the bits budget $B$.  We can see that the nonparametric linearity test based on quantized data fails to reject the null hypothesis,  which echos our conjecture based on Figure \ref{figure:real data}.}

\begin{figure}[ht!]
  \centering
  \includegraphics[width=4 in]{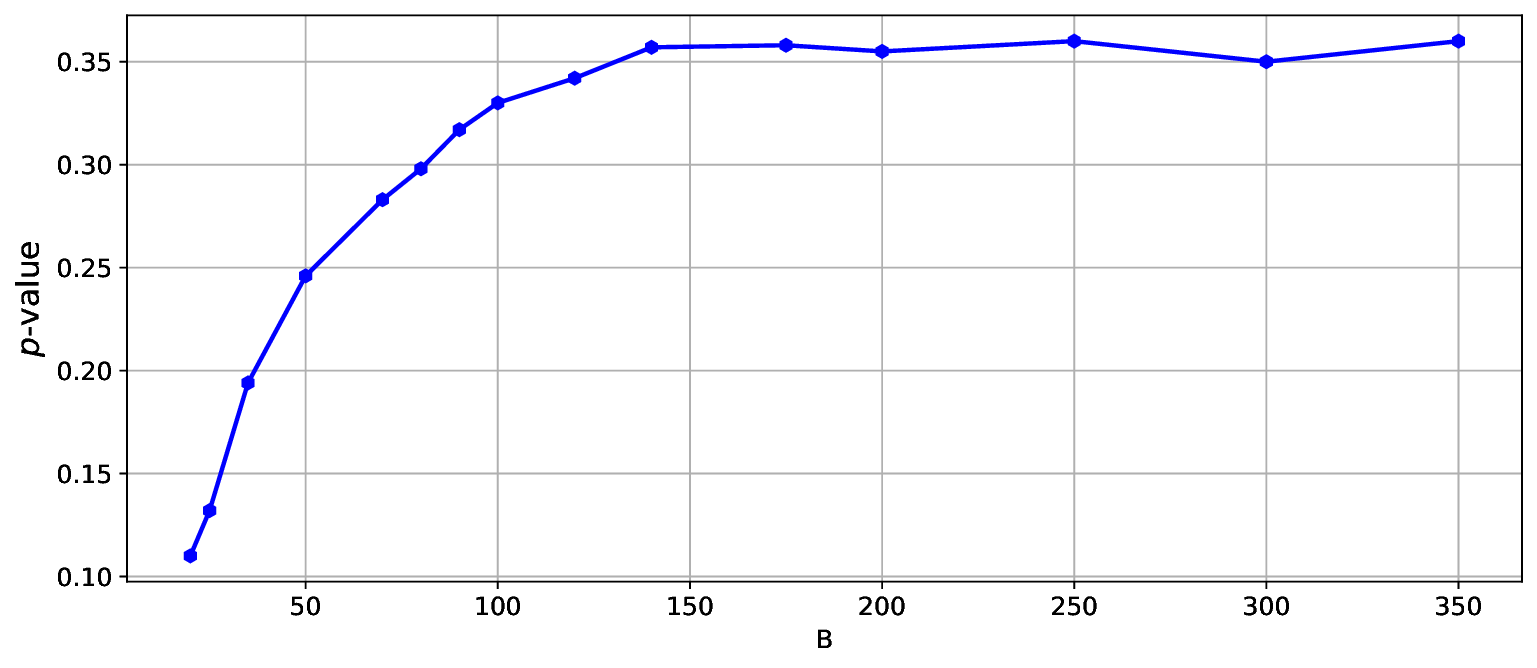}
   \caption{\it  $P$-values of the nonparametric linearity test for the relationship between relative humidity and energy output as a function of $B$.}
  \label{figure:p_val}
\end{figure}

\section{Discussion}
In this paper, we propose a set of non-parametric testing procedures based on quantized observations, including the non-adaptive nonparametric test, the nonparametric linearity test, and the adaptive nonparametric test. The proposed tests are easy-to-use based on $L_2$-metric between the quantization spline
estimators and the hypothesized function. We investigate the asymptotic validity and testing powers of the proposed tests and show how the asymptotic testing powers changes as the bits budget $B$ increases.

In the end, we discuss two additional extensions. First, the present paper only
deals with periodic splines. It is interesting to extend our results to more general splines
or even kernel ridge regression. The special periodic spline largely reduces the difficulty
level of the technical proofs. Indeed, the majority of the proofs can be accomplished by
exact calculations based on trigonometric series. For general RKHS, exact calculations may not be possible, and more involved proofs are needed. Second, the nonparametric linearity test can be easily extended to testing general composite null hypotheses such as $H_0^{\textrm{general}}: g_0(x)=h_0(x, \theta)$ for some function $h_0$ governed by parameters $\theta\in \mathbb{R}^p$ with a fixed $p$. However, when $p$ is diverging as $n$ increases, it will be more challenging to investigate the asymptotic behavior of the proposed test statistic and will be an interesting future research topic. 

\bibliography{references}

\newpage 
\appendix
\section{Structure of the proofs}
In this section, we outline the high-level structure of the proofs for the main theorems. 
\begin{itemize}
    \item The proof of Theorem \ref{upper:bound:hat:bb:f} is mainly based on Lemma \ref{a:preliminary:thm}.
\begin{itemize}
\item In Lemma \ref{a:preliminary:thm}, we provide an upper bound for the difference between two smoothing spline estimators.
\end{itemize}
      \item The proof of Theorem \ref{asymp:distribution} relies on Stein’s exchangeable pair method. Specifically, we first prove that the asymptotic normality of $\frac{cT_{\mu^\star,t,c}-\textrm{trace}(A)\tau^{\star2}_k}{s_c\tau^{\star2}_k}$ based on $z^\star_i$'s, where $z^\star_i$'s are the quantized samples corresponding to $\mu_j=\mu^\star_j$ for $1\le j\le k$, $\tau^{\star2}_k = Var(z^\star_i|H_0)$, and 
    \begin{equation}
\mu^\star_j=\frac{\sum_{i=1}^n E\{y_iI(y_i\in R_j(t))\}}{\sum_{i=1}^nP(y_i\in R_j(t))}. \end{equation}
Next, we prove that 
\[
\frac{cT_{\mu^\star,t,c}-\textrm{trace}(A)\tau^{\star2}_k}{s_c\tau^{\star2}_k} - \frac{cT_{\mu,t,c}-\textrm{trace}(A)\widehat{\tau}_k^2}{s_c\widehat{\tau}_k^2} = o_p(1).
\]  
\begin{itemize}
    \item In Lemma \ref{hat_tau_k} and Lemma \ref{hat_tau_k_linear}, we prove the error rate introduced by quantization of variance using Algorithm 2, which are needed for the proof of Theorem~\ref{asymp:distribution}.
        \item In Lemma \ref{power:thm:lemma:1}, we quantify the difference of quantized sample under $H_1$ and $H_0$.
    \end{itemize}
\item In the proof of Theorem \ref{power:thm}, we first decompose the test statistic into two parts, 
\[
\frac{cT_{\mu,t}-\textrm{trace}(A)\widehat{\tau}_k^2}{s_c\widehat{\tau}_k^2}\nonumber=\frac{z^TAz-(z^0)^TAz^0}{s_c\widehat{\tau}_k^2}+\frac{(z^0)^TAz^0-\textrm{trace}(A)\widehat{\tau}_k^2}{s_c\widehat{\tau}_k^2},
\]
where $z^0$ is the vector of quantized sample under $H_0: g_0=0$. Under Theorem \ref{asymp:distribution}, we know the second term is $O_p(1)$. In the first term, it is straightforward to see that $z^TAz-(z^0)^TAz^0 = (z-z^0)^TA(z-z^0)+2(z-z^0)^TAz^0$.
\begin{itemize}
    \item In Lemma \ref{power:thm:lemma:2}, we establish a lower bound for $(z-z^0)^TAz^0$.
    \item In Lemma \ref{power:thm:lemma:3} and Lemma \ref{power:thm:lemma:4}, we establish the lower bound for $(z-z^0)^TA(z-z^0)$.
\end{itemize}

\item In the proof of Theorem \ref{adaptive:size}, observe that the test statistic for each $m$
\[\xi_m = \frac{cT_{m}-\textrm{trace}(A_m)\widehat{\tau}_k^2}{s_{c,m}\widehat{\tau}_k^2}
=\frac{(z^0)^TA_mz^0-\textrm{trace}(A_m)\widehat{\tau}_k^2}{s_{c,m}\widehat{\tau}_k^2}
\]
is in a quadratic form. 
\begin{itemize}
    \item Lemma \ref{adaptive} proves that the maximum of the quadratic form follows an extreme value distribution.
    \item Lemma \ref{adaptive_rate} provides the rate conditions such that Lemma \ref{adaptive} holds.
\end{itemize}
\item The idea of Theorem \ref{adaptive:power} is similar to the proof of Theorem \ref{power:thm} and Theorem \ref{power:linearity}.
\end{itemize}
\section{Notation}
In this section, we first summarize some notations which are frequently used through out the paper for the reader's convenience.
\begin{table}[htb]
\begin{tabular}{|l|l|}
\hline
\textbf{Symbol} & \textbf{Description} \\
$c$      & number of groups.\\
 $\widetilde{n}$ & number of observations in each group which is defined as  $\widetilde{n} = n/c$.\\
 $(\mu_1, \ldots, \mu_k)^T$ & quantized value.\\
 $(t_1, \ldots, t_{k-1})^T$ & cut-off points of quantized intervals.\\
 $y= (y_1, \ldots, y_n)^T$    &         vector of response .\\
 $\widetilde{y} = (\widetilde{y}_1,\ldots, \widetilde{y}_c)^T$ & average of the response which is defined as $\widetilde{y}_i = \frac{1}{\widetilde{n}}\sum_{j=(i-1)\widetilde{n}+1}^{i\widetilde{n}}y_j$.\\
 $z=(z_1, \ldots, z_c)^T$ & vector of quantized sample.\\
 $z^0=(z^0_1, \ldots, z^0_c)^T$ & vector of quantized sample under $H_0: g_0=0$.\\
 $\widetilde{z}=(\widetilde{z}_1, \ldots, \widetilde{z}_c)^T$ & vector of truncated quantized sample,\\
  & where $\widetilde{z}_i=z_i\mathbbm{1}(c_s\rho+\sigma|\epsilon_j|\le\sqrt{\T_n}) \textrm{ for all }j=(i-1)\widetilde{n}+1, \dots i\widetilde{n} )$.\\
 $\widetilde{z}^0=(\widetilde{z}^0_1, \ldots, \widetilde{z}^0_c)^T$ & vector of truncated quantized response under $H_0: g_0=0$,\\
 & where $\widetilde{z}_i^0=z_i^0\mathbbm{1}(c_s\rho+\sigma|\epsilon_j|\le\sqrt{\T_n}) \textrm{ for all }j=(i-1)\widetilde{n}+1, \dots i\widetilde{n} )$.\\
 $y^\textrm{linear} = (y_1^\textrm{linear}, \ldots, y_n^\textrm{linear})^T$    &   new defined data for testing the linearity of $g_0$, which is defined as\\
 &  $y_i^\textrm{linear} = Q(y_i) - \widehat{g}(i/n)$, and $\widehat{g}(i/n)$ is  the least-square estimator of $g$.\\
 $z^\textrm{linear} = (z_1^\textrm{linear}, \ldots, z_c^\textrm{linear})^T$ & vector of quantized value of $y_i^\textrm{linear}$.\\
 $z_{0}^\textrm{linear} = (z_{i,0}^\textrm{linear}, \ldots, z_{i,c}^\textrm{linear})^T$ & quantized value of $y_i^\textrm{linear}$ under $H_0^\textrm{linear}: g_0 \textrm{ is linear}.$ \\
 $\lambda$ & smoothing parameter.\\
 $\{\varphi_i(x)\}_{i=1}^\infty$ & trigonometric basis functions.\\
 $K(\cdot, \cdot)$ & kernel function.\\
 $\Sigma_c$ & kernel matrix defined as $\Sigma_c=[K(i/c,i'/c)/c]_{1\le i,i'\le c}$.\\
 $\Omega_c$ & ``tensor'' of $K(\cdot, \cdot)$ defined as $\Omega_c=[K^{\otimes 2}(i/c,i'/c)/c]_{1\le i,i'\le c}$.\\
 $A$  & $A=(\Sigma_c+\lambda I_c)^{-1}\Omega_c(\Sigma_c+\lambda I_c)^{-1}$.\\
 $\zeta$ &  approximation error of Riemann sum and integral.\\
 $c_s$ & Sobolev constant defined as $c_s =\sup_{f\in S^m(\mathbb{I})} \frac{\|f\|_{\sup}}{\sqrt{J(f)}}$.\\
 $C_k(t)$ & maximum length of quantization interval.\\
 \hline
\end{tabular}
\caption{Table that lists some of the useful notations that are frequently used throughout the paper.}
\end{table}
\newpage
\section{Useful Lemmas}
The proofs of the theorems require some preliminary lemmas. In this section, we summarize these useful lemmas. Throughout the proof, we let $\widetilde{y}_i = \frac{1}{\widetilde{n}}\sum_{j=(i-1)\widetilde{n}+1}^{i\widetilde{n}}y_j, i=1, \ldots, c$ and we denote $\widehat{g}^{\textrm{ss}}$ as the canonical smoothing spline based on the full dataset; $\widehat{g}^{\textrm{ss}}_c$ as the smoothing spline based on the averaged responses {$\{\widetilde{y}_1,\cdots,\widetilde{y}_c\}$}, and $\widehat{g}^{\textrm{B}}_{\mu,t,c}$ as the desired $B$-bits estimator, i.e.,
\begin{align*}
	\widehat{g}^{\textrm{ss}} &=\argmin_{g\in S^{m}(\mathbb{I})}\frac{1}{n}\sum_{i=1}^n(y_i-g(i/n))^2+\lambda\int_0^1 [g^{(m)}(x)]^2dx, \\
    \widehat{g}^{\textrm{ss}}_c &=\argmin_{g\in S^m(\mathbb{I})}\frac{1}{c}\sum_{i=1}^c(\widetilde{y}_i-g(i/c))^2 + \lambda\int_0^1 [g^{(m)}(x)]^2dx, \,\,\, \widetilde{y}_i = \frac{1}{\widetilde{n}}\sum_{j=(i-1)\widetilde{n}+1}^{i\widetilde{n}}y_j,\\
    \widehat{g}^{\textrm{B}}_{\mu,t,c} &=\argmin_{g\in S^m(\mathbb{I})}\frac{1}{c}\sum_{i=1}^c(z_i-g(i/c))^2+\lambda\int_0^1 [g^{(m)}(x)]^2dx.
\end{align*}
The following lemma describes that the distance between
$\widehat{g}^{\textrm{B}}_{\mu,t,c}$ and $\widehat{g}^{\textrm{ss}}_c$ can be well controlled by {carefully choosing} quantization parameters $\mu, t$ and $c$.
\begin{lemma}\label{a:preliminary:thm}
For any $\mu=(\mu_1,\ldots,\mu_k)^T\in\bbR^k$ and $t=(t_1,\ldots,t_{k-1})^T\in\bbR^{k-1}$,
it holds that
\begin{equation}\label{upper:bound:diff}
\|\widehat{g}^{\textrm{B}}_{\mu,t,c}
-\widehat{g}^{\textrm{ss}}_c\|^2\le c^{-1}\sum_{i=1}^c(z_i-\widetilde{y}_i)^2.
\end{equation}
\end{lemma}
\begin{proof}
Recall that $\widehat{g}^{\textrm{B}}_{\mu,t,c}=\sum_{i=1}^c\widehat{\theta}_i K_{i/c}$, where $(\widehat{\theta}_1,\ldots,\widehat{\theta}_c)^T=c^{-1}(\Sigma_c+\lambda I_c)^{-1}z$ with $\Sigma_c=[K(i/c,i'/c)/c]_{1\le i,i'\le c}\in \mathbb{R}^{c\times c}$, $z=(z_1,\ldots,z_c)^T\in \mathbb{R}^c$, and $K(\cdot, \cdot)$ is the kernel function. Similarly, $\widehat{g}^{\textrm{ss}}_c = \sum_{i=1}^c\widetilde{\theta}_i K_{i/c}$, where  $(\widetilde{\theta}_1,\ldots,\widetilde{\theta}_c)^T=c^{-1}(\Sigma_c+\lambda I_c)^{-1}\widetilde{y}$ with $\widetilde{y}=(\widetilde{y}_1,\ldots,\widetilde{y}_c)^T$. Let $\widehat{\theta}=(\widehat{\theta}_1,\ldots,\widehat{\theta}_c)^T, \widetilde{\theta}=(\widetilde{\theta}_1,\ldots,\widetilde{\theta}_c)^T$. By direct calculations, we have
\begin{eqnarray*}
\widehat{g}_{\mu,t,c}^{\textrm{B}}
-\widehat{g}^{\textrm{ss}}_c&=&\sum_{\nu=1}^\infty\sum_{i=1}^c(\widehat{\theta}_i-\widetilde{\theta}_i)
\frac{\varphi_\nu(i/c)}{\gamma_\nu}\varphi_\nu =\sum_{\nu=1}^\infty\frac{\Phi_\nu^T(\widehat{\theta}-\widetilde{\theta)}}{\gamma_\nu}\varphi_\nu,
\end{eqnarray*}
where $\varphi_{2k-1}(x)=\sqrt{2}\cos(2\pi kx),\,\,\,\,
\varphi_{2k}(x)=\sqrt{2}\sin(2\pi kx)$ are the trigonometric basis functions, $\gamma_{2k-1}=\gamma_{2k}=(2\pi k)^{2m}$, and $\Phi_\nu=(\varphi_\nu(1/c),\varphi_\nu(2/c),\ldots,\varphi_\nu(c/c))^T$. So
\begin{eqnarray}\label{thm1:eqn1}
\|\widehat{g}^{\textrm{B}}_{\mu,t,c}
-\widehat{g}^{\textrm{ss}}_c\|^2&=&\sum_{\nu=1}^\infty\frac{|\Phi_\nu^T(\widehat{\theta}-\widetilde{\theta})|^2}{\gamma_\nu^2}
\nonumber\\
&=&(\widehat{\theta}-\widetilde{\theta})^T\sum_{\nu=1}^\infty\frac{\Phi_\nu\Phi_\nu^T}{\gamma_\nu^2}(\widehat{\theta}-\widetilde{\theta})
\nonumber\\
&=&c^{-1}(z-\widetilde{y})^T(\Sigma_c+\lambda I_c)^{-1}\Omega_c(\Sigma_c+\lambda I_c)^{-1}(z-\widetilde{y}).
\end{eqnarray}
We now look at $\Sigma_c$ and $\Omega_c$.
To ease our calculations, for $0\le l\le c-1$, we first define the following two notations,
\begin{eqnarray*}
d^{\prime}_l&=&\frac{2}{c}\sum_{k=1}^\infty\frac{\cos(2\pi k l/c)}{(2\pi k)^{2m}}, d_l=\frac{2}{c}\sum_{k=1}^\infty\frac{\cos(2\pi k l/c)}{(2\pi k)^{4m}}.
\end{eqnarray*}
Since $d^{\prime}_l=d^{\prime}_{c-l}$, $d_l=d_{c-l}$ for $l=1,2,\ldots,c-1$, we know
$\Sigma_c,  \Omega_c$ are both symmetric circulant of order $c$. Furthermore, $\Sigma_c$ and $\Omega_c$ share the same normalized eigenvectors as
\[
x_r=\frac{1}{\sqrt{c}}(1,\varepsilon^r,\varepsilon^{2r},\ldots,\varepsilon^{(c-1)r})^T,
\,\,r=0,1,\ldots,c-1,
\]
where $\varepsilon=\exp(2\pi\sqrt{-1}/c)$. Let $M=(x_0,x_1,\ldots,x_{c-1})$, and $M^\ast$ be the conjugate transpose of $M$.
Clearly, $MM^\ast=I_c$
and $\Sigma_c,\Omega_c$ admit the following decomposition
\begin{equation}\label{Sigma:Sigma2:decompose}
\Sigma_c=M\Lambda_{d^{\prime}} M^\ast,\,\,\Omega_c=M\Lambda_dM^\ast,
\end{equation}
where $\Lambda_{d^{\prime}}=\textrm{diag}(\lambda_{d^{\prime},0},\lambda_{d^{\prime},1},\ldots,\lambda_{d^{\prime},c-1})$
and $\Lambda_d=\textrm{diag}(\lambda_{d,0},\lambda_{d,1},\ldots,\lambda_{d,c-1})$
with $\lambda_{d^{\prime},l}=d^{\prime}_0+d^{\prime}_1\varepsilon^l+\ldots+d^{\prime}_{c-1}\varepsilon^{(c-1)l}$
and $\lambda_{d,l}=d_0+d_1\varepsilon^l+\ldots+d_{c-1}\varepsilon^{(c-1)l}$.

Direct calculations show that
\begin{equation}\label{formula:lambda:lc}
\lambda_{d^{\prime},l}=\left\{
\begin{array}{cc}
2\sum_{k=1}^\infty\frac{1}{(2\pi k c)^{2m}},& l=0, \nonumber\\
\sum_{k=1}^\infty\frac{1}{[2\pi(kc-l)]^{2m}}+\sum_{k=0}^\infty\frac{1}{[2\pi(kc+l)]^{2m}}, & 1\le l\le c-1
\end{array}\right.
\end{equation}
\begin{equation}\label{formula:lambda:ld}
\lambda_{d,l}=\left\{
\begin{array}{cc}
2\sum_{k=1}^\infty\frac{1}{(2\pi k c)^{4m}},& l=0, \nonumber\\
\sum_{k=1}^\infty\frac{1}{[2\pi(kc-l)]^{4m}}
+\sum_{k=0}^\infty\frac{1}{[2\pi(kc+l)]^{4m}},
& 1\le l\le c-1.
\end{array}\right.
\end{equation}
It is easy to examine that
\begin{equation}\label{expression:lambda:0}
\lambda_{d^{\prime},l}=\left\{
\begin{array}{cc}
2\bar{d^{\prime}}_m(2\pi c)^{-2m},& l=0,\\
\frac{1}{[2\pi(c-l)]^{2m}}+\frac{1}{(2\pi l)^{2m}} + \sum_{k=2}^\infty\frac{1}{[2\pi(kc-l)]^{2m}}
+\sum_{k=1}^\infty\frac{1}{[2\pi(kc+l)]^{2m}}, & 1\le l\le c-1,
\end{array}\right.
\end{equation}

\begin{equation}\label{expression:lambda:l}
\lambda_{d,l}=\left\{
\begin{array}{cc}
2\bar{d}_m(2\pi c)^{-4m},& l=0,\\
\frac{1}{[2\pi(c-l)]^{4m}}+\frac{1}{(2\pi l)^{4m}}+\sum_{k=2}^\infty\frac{1}{[2\pi(kc-l)]^{4m}}
+\sum_{k=1}^\infty\frac{1}{[2\pi(kc+l)]^{4m}}, & 1\le l\le c-1,
\end{array}\right.
\end{equation}
\begin{eqnarray*}
&&\underline{d}^{\prime}_m(2\pi c)^{-2m}\le
\sum_{k=2}^\infty\frac{1}{[2\pi(kc-l)]^{2m}}\le \bar{d^{\prime}}_m(2\pi c)^{-2m}, \\
&&\underline{d}^{\prime}_m(2\pi c)^{-2m}\le
\sum_{k=1}^\infty\frac{1}{[2\pi(kc+l)]^{2m}}\le\bar{d}^{\prime}_m(2\pi c)^{-2m},\\
&&\underline{d}_m(2\pi c)^{-4m}\le
\sum_{k=2}^\infty\frac{1}{[2\pi(kc-l)]^{4m}}\le \bar{d}_m(2\pi c)^{-4m}, \\
&&\underline{d}_m(2\pi c)^{-4m}\le
\sum_{k=1}^\infty\frac{1}{[2\pi(kc+l)]^{4m}}\le\bar{d}_m(2\pi c)^{-4m},
\end{eqnarray*}
where
\begin{align} \label{cmdm}
\bar{d}^{\prime}_m:=\sum_{k=1}^\infty k^{-2m},
\underline{d}^{\prime}_m:=\sum_{k=2}^\infty k^{-2m},
\bar{d}_m:=\sum_{k=1}^\infty k^{-4m},
\underline{d}_m:=\sum_{k=2}^\infty k^{-4m}.
\end{align}
It follows from (\ref{expression:lambda:0}) and (\ref{expression:lambda:l}) that
$\lambda_{d,l}\le \lambda_{d^{\prime},l}^2$ for $0\le l\le c-1$.
Therefore,
\begin{eqnarray*}
&&(z-\widetilde{y})^T(\Sigma_c+\lambda I_c)^{-1}\Omega_c(\Sigma_c+\lambda I_c)^{-1}(z-\widetilde{y})\\
&=&(z-\widetilde{y})^TM\textrm{diag}\left(\frac{\lambda_{d,0}}{(\lambda+\lambda_{d^{\prime},0})^2},
\ldots,\frac{\lambda_{d,c-1}}{(\lambda+\lambda_{d^{\prime},c-1})^2}\right)M^\ast(z-\widetilde{y})\\
&\le& (z-\widetilde{y})^TMM^\ast(z-\widetilde{y})=
(z-\widetilde{y})^T(z-\widetilde{y}).
\end{eqnarray*}
Therefore, it follows by (\ref{thm1:eqn1}) that
\begin{equation}
\begin{split}
\|\widehat{g}^{\textrm{B}}_{\mu,t,c}
-\widehat{g}^{\textrm{ss}}_c\|^2\le c^{-1}(z-\widetilde{y})^T(z-\widetilde{y})=c^{-1}\sum_{i=1}^c(z_i-\widetilde{y}_i)^2.
\label{approxerr}
\end{split}
\end{equation}
This completes the proof.
\end{proof}

\begin{lemma}\label{hat_tau_k}
Suppose Condition (\textbf{B}) holds true, and it holds that $C_k(t)\rightarrow{0}$, then we have $\tau_k^2 = Var(z_1|H_0) = O(\widetilde{n}^{-1} + C_k(t)^2)$ and $\widehat{\tau}_k^2= \frac{\widetilde\tau_n^2}{2\widetilde{n}(n-1)} = \tau_k^2\left[1+O_p(n^{-1/2}+C_k(t))\right] = \tau_k^2[1+o_p(1)]$.
\end{lemma}
\begin{proof}
By the definition of $\tau_k^2$ and (\ref{eq:quant:2}) we have
\begin{align}\label{eq_tau_k}
\tau_k^2 &=  Var(z_1|H_0) = \frac{1}{\widetilde n ^2}Var(\sum_{i=1}^{\widetilde{n}}Q(y_i)|H_0) + O(C_k(t)^2) = \frac{1}{\widetilde n }Var(Q(\sigma\epsilon_1)) + O(C_k(t)^2) \nonumber \\
&=\frac{1}{\widetilde n }\sum_{j=1}^k\mu_j^2P\left(Q(\sigma\epsilon_1)=\mu_j\right) - \frac{1}{\widetilde n }\left(\sum_{j=1}^k\mu_jP\left(Q(\sigma\epsilon_1)=\mu_j\right)\right)^2 + O(C_k(t)^2)\nonumber\\
&=\frac{1}{\widetilde n }\sum_{j=1}^k\mu_j^2P\left(\sigma\epsilon_1\in R_j(t)\right) - \frac{1}{\widetilde n }\left(\sum_{j=1}^k\mu_jP\left(\sigma\epsilon_1\in R_j(t)\right)\right)^2 + O(C_k(t)^2) \nonumber \\
&=R_1 + R_2 + O(C_k(t)^2).
\end{align}

Assume that for $2 \leq s \leq k-1, t_1<t_2<\cdots<t_{s-1} \leq 0<t_s<\cdots<t_{k-1}$ and let $p(\epsilon)$ be the density function of $\epsilon_1$. Then we have
$$
\begin{aligned}
\sum_{j=2}^{s-1} \mu_j^2 P\left(\sigma \epsilon_1 \in R_j(t)\right) & \leq \sum_{j=2}^{s-1} \int_{t_{j-1} / \sigma}^{t_j / \sigma} p(\epsilon) d \epsilon t_{j-1}^2 \\
& \leq 2 \sum_{j=2}^{s-1} \int_{t_{j-1} / \sigma}^{t_j / \sigma} p(\epsilon) d \epsilon\left(t_j^2+C_k\left(t\right)^2\right) \\
& \leq 2 \sigma^2 \sum_{j=2}^{s-1} \int_{t_{j-1} / \sigma}^{t_j / \sigma} \epsilon^2 p(\epsilon)d\epsilon+2 C_k(t)^2 \sum_{j=2}^{s-1} \int_{t_{j-1} / \sigma}^{t_j / \sigma} p(\epsilon) d \epsilon
\end{aligned}
$$
and 
$$
\begin{aligned}
\sum_{j=s+1}^{k-1} \mu_j^2 P\left(\sigma \epsilon_1 \in R_j(t)\right) & \leq \sum_{j=s+1}^{k-1} \int_{t_{j-1} / \sigma}^{t_j / \sigma} p(\epsilon) d \epsilon t_j^2 \\
& \leq 2 \sum_{j=s+1}^{k-1} \int_{t_{j-1} / \sigma}^{t_j / \sigma} p(\epsilon) d \epsilon\left(t_{j-1}^2+C_k(t)^2\right) \\
& \leq 2 \sigma^2 \sum_{j=s+1}^{k-1} \int_{t_{j-1} / \sigma}^{t_j / \sigma} \epsilon^2 p(\epsilon) d \epsilon+2 C_k(t)^2 \sum_{j=s+1}^{k-1} \int_{t_{j-1} / \sigma}^{t_j / \sigma} p(\epsilon) d \epsilon .
\end{aligned}
$$

The fact that $\left|\mu_s\right| \leq C_k(t)$ and the above inequalities lead
$$
\begin{aligned}
R_1 = & \frac{1}{\widetilde{n}}\Bigg\{\sum_{j=2}^{s-1} \mu_j^2 P\left(\sigma \epsilon_1 \in R_j(t)\right)+\sum_{j=s+1}^{k-1} \mu_j^2 P\left(\sigma \epsilon_1 \in R_j(t)\right) \\
& +\mu_1^2 P\left(\sigma \epsilon_1 \in R_1(t)\right)+\mu_k^2 P\left(\sigma \epsilon_1 \in R_k(t)\right)+\mu_s^2 P\left(\sigma \epsilon_1 \in R_s(t)\right)\Bigg\} \\
\leq & \frac{1}{\widetilde{n}}\left\{2 \sigma^2 \int_{\mathbb{R}} \epsilon^2 p(\epsilon) d \epsilon+3 C_k(t)^2+O(1)\right\} .
\end{aligned}
$$
This proves $R_1 \lesssim \frac{1}{\widetilde{n}}$.
On the other hand, by $t_1^2>\sigma^2$ and $t_{s-1}=O\left(C_k(t)\right)=o(1)$, we have
$$
\begin{aligned}
R_1 & \geq \frac{1}{\widetilde{n}}\sum_{j=2}^{s-1} \mu_j^2 P\left(\sigma \epsilon_1 \in R_j(t)\right) \\
& \geq \frac{1}{\widetilde{n}}\sum_{j=2}^{s-1} t_j^2 \int_{t_{j-1} / \sigma}^{t_j / \sigma} p(\epsilon) d \epsilon \\
& \geq \frac{1}{\widetilde{n}}\sum_{j=2}^{s-1}\left(t_{j-1}^2 / 2-C_k(t)^2\right) \int_{t_{j-1} / \sigma}^{t_j / \sigma} p(\epsilon) d \epsilon \\
& \geq \frac{1}{\widetilde{n}}\left\{\frac{\sigma^2}{2} \sum_{j=2}^{s-1} \int_{t_{j-1} / \sigma}^{t_j / \sigma} \epsilon^2 p(\epsilon) d \epsilon-C_k(t)^2 \sum_{j=2}^{s-1} \int_{t_{j-1} / \sigma}^{t_j / \sigma} p(\epsilon) d \epsilon \right\}\\
& \geq \frac{1}{\widetilde{n}}\left\{\frac{\sigma^2}{2} \int_{t_1 / \sigma}^{t_{s-1} / \sigma} \epsilon^2 p(\epsilon) d \epsilon-C_k(t)^2\right\} = O(\widetilde{n}^{-1}) .
\end{aligned}
$$
This proves $R_1 \gtrsim \frac{1}{\widetilde{n}}$, which implies $R_1 = O(\widetilde{n}^{-1})$. Using a similar approach, we can prove $R_2 = O(\widetilde{n}^{-1})$. From (\ref{eq_tau_k}), we get $\tau_k^2 = O(\widetilde{n}^{-1} + C_k(t)^2)$.

Now we prove $\widehat{\tau}_k^2 = \frac{\widetilde\tau_n^2}{2\widetilde{n}(n-1)}=\tau_k^2[1+o_p(1)]$. By the definition of $\widetilde\tau_n$, we have
\begin{align}
    \frac{\widetilde\tau_n^2}{2\widetilde{n}(n-1)}&
=    \frac{1}{2\widetilde{n}(n-1)}\sum_{i=2}^n(Q(y_i)- 
                         Q(y_{i-1}))^2\nonumber \\
                         &=  \frac{1}{2\widetilde{n}(n-1)}\sum_{i=2}^n\left\{Q(g_0(i/n)+\sigma\epsilon_i)-Q(g_0((i-1)/n)+\sigma\epsilon_{i-1}) \right\}^2\nonumber\\
                         &=\frac{1}{2\widetilde{n}(n-1)}\sum_{i=2}^n\left\{g_0(i/n)-g_0((i-1)/n)+Q(\sigma\epsilon_i)-Q(\sigma\epsilon_{i-1}) + \widetilde{\delta}_i\right\}^2\nonumber\\
                         &=\frac{1}{2\widetilde{n}(n-1)}\sum_{i=2}^n\left\{Q(\sigma\epsilon_i)-Q(\sigma\epsilon_{i-1})\right\}^2+\frac{1}{2\widetilde{n}(n-1)}\sum_{i=2}^n\left\{g_0(i/n)-g_0((i-1)/n) + \widetilde{\delta}_i\right\}^2\nonumber\\
                         &\qquad+\frac{1}{\widetilde{n}(n-1)}\sum_{i=2}^n\left\{Q(\sigma\epsilon_i)-Q(\sigma\epsilon_{i-1})\right\}\left\{g_0(i/n)-g_0((i-1)/n) + \widetilde{\delta}_i\right\},\nonumber
\end{align}
where $|\widetilde\delta_i|=O(C_k(t))$ for $i=1,\cdots,n$. Note that, by the central limit theorem, it holds that
\[
\frac{1}{2(n-1)}\sum_{i=2}^n\left\{Q(\sigma\epsilon_i)-Q(\sigma\epsilon_{i-1})\right\}^2= Var\left[Q(\sigma\epsilon_1)\right]+O_p(n^{-1/2}),
\]
and that $g_0(i/n)-g_0((i-1)/n)=O(1/n)$ by the smoothness of function $g_0$, we have that
\begin{align}
\frac{\widetilde\tau_n^2}{2\widetilde{n}(n-1)}&=\frac{1}{\widetilde n }\left[Var(Q(\sigma\epsilon_1))+O(n^{-2}+C_k(t)^2) + O_p(n^{-1/2})+O(n^{-1}+C_k(t))\right]\nonumber\\&=\tau_k^2\left[1+O_p(n^{-1/2}+C_k(t))\right], \label{eq_proof_variance_first}
\end{align}
which completes the proof.
\end{proof}

\begin{lemma}\label{hat_tau_k_linear}
Suppose Condition (\textbf{B}) holds true, and it holds that $C_k(t)\rightarrow{0}$. Let $\widehat{\tau}_k^2$ be the quantied variance based on $y^{\textrm{linear}}$, then we have that $\widehat{\tau}_k^2= \frac{\widetilde\tau_n^2}{2\widetilde{n}(n-1)} = \tau_k^2\left[1+O_p(n^{-1/2}+C_k(t))\right] = \tau_k^2[1+o_p(1)]$.
\end{lemma}
\begin{proof}
By the definition of $\widehat{\tau}_k^2 $, we have
\begin{align}
    \frac{\widetilde\tau_n^2}{2\widetilde{n}(n-1)}&
=    \frac{1}{2\widetilde{n}(n-1)}\sum_{i=2}^n(Q(y_i^{\textrm{linear}})- 
                         Q(y^{\textrm{linear}}_{i-1}))^2\nonumber \\
                         &=  \frac{1}{2\widetilde{n}(n-1)}\sum_{i=2}^n\left\{Q[Q(g_0(i/n)+\sigma\epsilon_i) - \widehat{y}_i] -Q[Q(g_0((i-1)/n)+\sigma\epsilon_{i-1})] -\widehat{y}_{i-1} \right\}^2\nonumber\\
                         &=\frac{1}{2\widetilde{n}(n-1)}\sum_{i=2}^n\left\{g_0(i/n)-g_0((i-1)/n) + \widehat{y}_i -\widehat{y}_{i-1} + Q(\sigma\epsilon_i)-Q(\sigma\epsilon_{i-1}) + \widetilde{\delta}_i\right\}^2,\nonumber
\end{align}
where $|\widetilde\delta_i|= O(C_k(t))$ for $i=1,\cdots,n$. Note that $\widehat{g} \in \mathcal{L}(\mathbb{I})$, one has that $|\widehat{y}_{i} - \widehat{y}_{i-1}| = O_p(1/n)$ by the smoothness of $\widehat{g}$. Similar to the proof in Lemma \ref{hat_tau_k}, we get $\widehat{\tau}_k^2= \frac{\widetilde\tau_n^2}{2\widetilde{n}(n-1)} = \tau_k^2\left[1+O_p(n^{-1/2}+C_k(t))\right] = \tau_k^2[1+o_p(1)]$.
\end{proof}

To ease calculation, we define some useful notations. Let $z_i^0$ be the quantized data conditional on $g_0(x) = 0$ and $z^0 = (z_1^0,\ldots,z_c^0)^T$. According to (\ref{eq:quant:2}), we have

\begin{equation} \label{eq:quant_z0}
\left|z_i^0 - \frac{1}{\widetilde{n}}\sum_{j=(i-1)\widetilde{n}+1}^{i\widetilde{n}}Q(\sigma\epsilon_j)\right|\le C_k(t),\,\, \textrm{ for $i=1,\ldots,c$.}
\end{equation}

Furthermore, we let $\widetilde{z}^0=(\widetilde{z}_1^0,\ldots,\widetilde{z}_c^0)^T, \widetilde{z}=(\widetilde{z}_1,\ldots,\widetilde{z}_c)^T$, where for $i=1, \ldots, c,$
\begin{eqnarray*}
\widetilde{z}_i^0=z_i^0\mathbbm{1}(c_s\rho + \sigma|\epsilon_j|\le\sqrt{\T_n} \textrm{ for all }j=(i-1)\widetilde{n}+1, \dots i\widetilde{n} ), \\
\widetilde{z}_i=z_i\mathbbm{1}(c_s\rho + \sigma|\epsilon_j|\le\sqrt{\T_n} \textrm{ for all }j=(i-1)\widetilde{n}+1, \dots i\widetilde{n} ). 
\end{eqnarray*}

\begin{lemma}\label{power:thm:lemma:1} Suppose $g$ is the regression function generating the samples. Suppose Condition (\textbf{B}) holds
and $\sigma|\epsilon_j|+c_s\rho\le\sqrt{\T_n}$ holds for all $j=1,\ldots,n$. Then
for any $g\in S^m(\mathbb{I})$ with $J(g)\le\rho^2$, it holds that
$|\widetilde{z}_i-\widetilde{z}_i^0-f(i/c)|\le 4C_k(t) + \zeta, i=1, \ldots, c$,
where $f$ is the corresponding integral equation defined in (\ref{eq:f0}), $\zeta= \max\limits_{i = 1,\ldots,c}|f(i/c) - \frac{1}{\widetilde{n}}\sum_{j=(i-1)\widetilde{n}+1}^{i\widetilde{n}}g(j/n)| = \max\limits_{i = 1,\ldots,c}|\frac{1}{2\Delta}\int_{\max(i/c-\Delta, 0)}^{\min(i/c+\Delta, 1)}g(s)ds - \frac{1}{\widetilde{n}}\sum_{j=(i-1)\widetilde{n}+1}^{i\widetilde{n}}g(j/n)|$ and $\Delta=\frac{1}{c}$.
\end{lemma}

\begin{proof}\textbf{\hspace*{-1mm}:}
Suppose $\sigma\epsilon_i\in R_j(t)$ for some $1\le j\le k$. Since $\min\{t_1^2,t_{k-1}^2\}=\T_n$ and
$c_s\rho + \sigma|\epsilon_i|\le\sqrt{\T_n}$, we must have $2\le j\le k-1$.
Suppose that $g(i/n)+\sigma\epsilon_i\in R_l(t)$ for some $1\le l\le k$.
Since $\min\{t_1^2,t_{k-1}^2\}=\T_n$ and by (\ref{Sob:const}) implying
$|g(i/n)|\le c_s\rho$, we have
\[
|y_i| = |g(i/n)+\sigma\epsilon_i|\le |g(i/n)|+|\sigma\epsilon_i|\le c_s\rho+|\sigma\epsilon_i|\le\sqrt{\T_n} = \min\{|t_1|,|t_{k-1}|\}.
\]
Therefore, $2\le l\le k-1$.
Since
\[
t_{j-1}\le\sigma\epsilon_i<t_j,\,\,\,\,
t_{l-1}\le g(i/n) + \sigma\epsilon_i<t_l,
\]
\[
t_{j-1}\le\mu_j<t_j,\,\,\,\,
t_{l-1}\le \mu_l<t_l,
\]
we have
\[
t_{l-1}-t_j<\mu_l-\mu_j<t_l-t_{j-1},
\]
\[
t_{l-1}-t_j<g(i/n)<t_l-t_{j-1}.
\]
Hence it holds that
\begin{equation}\label{eq:tmp1}
|Q(y_i) - Q(\sigma\epsilon_i) - g(i/n)| = |\mu_l-\mu_j-g(i/n)| \le|t_l-t_{l-1}|+|t_j-t_{j-1}|\le 2C_k(t).    
\end{equation}
Since $c_s\rho + \sigma|\epsilon_j|\le\sqrt{\T_n}$ for all $j=1, \ldots, n$, the result follows from (\ref{eq:quant:2}) and (\ref{eq:quant_z0}) that
\begin{align*}
|\widetilde{z}_i-\widetilde{z}_i^0-f(i/c)| =& \Big|z_i - \frac{\sum_{j=(i-1)\widetilde{n}+1}^{i\widetilde{n}}Q(y_j)}{\widetilde{n}} - \left(z_i^0 - \frac{\sum_{j=(i-1)\widetilde{n}+1}^{i\widetilde{n}}Q(\sigma\epsilon_j)}{\widetilde{n}}\right)\\
&+  \frac{\sum_{j=(i-1)\widetilde{n}+1}^{i\widetilde{n}}Q(y_j)}{\widetilde{n}} - \frac{\sum_{j=(i-1)\widetilde{n}+1}^{i\widetilde{n}}Q(\sigma\epsilon_j)}{\widetilde{n}} - f(i/c)\Big|\\
\le&\frac{\left|\sum_{j=(i-1)\widetilde{n}+1}^{i\widetilde{n}}\left((Q(y_j) - Q(\sigma\epsilon_j)\right)-\widetilde{n}f(i/c)\right|}{\widetilde{n}}+2C_k(t)\\
=& \frac{\left|\sum_{j=(i-1)\widetilde{n}+1}^{i\widetilde{n}}\left\{Q(y_j) - Q(\sigma\epsilon_j) - g(j/n)\right\} + \sum_{j=(i-1)\widetilde{n}+1}^{i\widetilde{n}}g(j/n)-\widetilde{n}f(i/c)\right|}{\widetilde{n}}+2C_k(t)\\
\le&\frac{\left|\sum_{j=(i-1)\widetilde{n}+1}^{i\widetilde{n}}\left\{Q(y_j) - Q(\sigma\epsilon_j) - g(j/n)\right\}\right| + \left|\sum_{j=(i-1)\widetilde{n}+1}^{i\widetilde{n}}g(j/n)-\widetilde{n}f(i/c)\right|}{\widetilde{n}}+2C_k(t)\\
\le& 4C_k(t) + \zeta,
\end{align*}
where the last inequality follows from (\ref{eq:tmp1}) and the definition of $\zeta$.
\end{proof}

\begin{lemma}\label{power:thm:lemma:2}
Suppose Condition (\textbf{B}) holds, and  $h\rightarrow0$, $ch\rightarrow\infty$. Then for any $g\in S^m(\mathbb{I})$ with $J(g)\le\rho^2$, we have
\begin{equation}\label{lemma:2:result}
\sup_{g\in S_\rho^m(\mathbb{I})}\frac{E\{|(\widetilde{z}-\widetilde{z}^0)^TAz^0|^2\}}
{(1+(ch^2)^{-1})(\tau_k^2+ 4C_k(t)^2)\sum_{i=1}^c(|f(i/c)|+4C_k(t) + \zeta)^2}\le8,\,\,\textrm{as $c\rightarrow\infty$},
\end{equation}
where $\zeta = \max\limits_{i = 1,\ldots,c}|\frac{1}{2\Delta}\int_{\max(i/c-\Delta, 0)}^{\min(i/c+\Delta, 1)}g(s)ds - \frac{1}{\widetilde{n}}\sum_{j=(i-1)\widetilde{n}+1}^{i\widetilde{n}}g(j/n)|$ and $\Delta=\frac{1}{c}$.
\end{lemma}
\begin{proof}\textbf{\hspace*{-1mm}:}
For convenience, let $\omega_i=\widetilde{z}_i-\widetilde{z}_i^0$.
From Lemma \ref{power:thm:lemma:1} and the fact that $\widetilde{z}_i - \widetilde{z}_i^0= 0$ if $c_s\rho+\sigma|\epsilon_j|>\sqrt{\T_n}$ for some $(i-1)\widetilde{n} + 1\le j \le i\widetilde{n}$, it holds that
\begin{equation}\label{power:lemma:2:eqn:1}
\textrm{$|\omega_i|\le|f(i/c)|+4C_k(t)+\zeta$ for all $1\le i\le c$.}
\end{equation}
According to (\ref{eq:quant_z0}) and the fact that 
\[
\left|E\left(\frac{1}{\widetilde{n}}\sum_{j=(i-1)\widetilde{n}+1}^{i\widetilde{n}}Q(\sigma\epsilon_j)\right)\right| \le E\left(\frac{1}{\widetilde{n}}\sum_{j=(i-1)\widetilde{n}+1}^{i\widetilde{n}}\sigma\epsilon_j\right) + C_k(t),
\]
one has that
\[
|E(z_i^0)| \le \left|E\left(\frac{1}{\widetilde{n}}\sum_{j=(i-1)\widetilde{n}+1}^{i\widetilde{n}}Q(\sigma\epsilon_j)\right)\right| + C_k(t) \le 2C_k(t),
\]
which further implies that 
\[
E(|z_i^0|^2) = Var(z_i^0) + |E(z_i^0)|^2 \le \tau_k^2 + 4C_k(t)^2.
\]
For any $g\in S^m(\mathbb{I})$ with $J(g)\le\rho^2$,
we have
\begin{eqnarray*}
&&E\{|(\widetilde{z}-\widetilde{z}^0)^TAz^0|^2\}\\
&=&E\{(\sum_{u,v=1}^c a_{u,v}\omega_u z_v^0)^2\}\\
&=&\sum_{1\le u_1,u_2,v_1,v_2\le c}a_{u_1,v_1}a_{u_2,v_2}E\{\omega_{u_1}\omega_{u_2}z_{v_1}^0z_{v_2}^0\}\\
&=&\sum_{\substack{u_1,u_2,v_1\\\textrm{are mutually different}}}
a_{u_1,v_1}a_{u_2,v_1}E\{\omega_{u_1}\}E\{\omega_{u_2}\}E\{|z_{v_1}^0|^2\}
+\sum_{u_1\neq v_1}a_{u_1,v_1}^2E\{\omega_{u_1}^2\}E\{|z_{v_1}^0|^2\}\\
&&+\sum_{u_1\neq u_2}a_{u_1,u_1}a_{u_2,u_2}E\{\omega_{u_1}z_{u_1}^0\}E\{\omega_{u_2}z_{u_2}^0\}
+\sum_{u_1\neq u_2}a_{u_1,u_1}a_{u_2,u_1}E\{\omega_{u_1}|z_{u_1}^0|^2\}E\{\omega_{u_2}\}\\
&&+\sum_{u_1}a_{u_1,u_1}^2E\{\omega_{u_1}^2|z_{u_1}^0|^2\}\equiv T_1+T_2+T_3+T_4+T_5.
\end{eqnarray*}
To complete the proof, we will analyze the above terms $T_1$ through $T_5$.

For $T_1$, we have
\begin{eqnarray*}
T_1&=&(\tau_k^2+ 4C_k(t)^2)\sum_{u_1\neq u_2}\sum_{v_1\neq u_1,u_2}a_{u_1,v_1}a_{u_2,v_1}E\{\omega_{u_1}\}E\{\omega_{u_2}\}\\
&\le&(\tau_k^2+ 4C_k(t)^2)\sum_{u_1,u_2=1}^n\left(\sum_{v_1\neq u_1,u_2}a_{u_1,v_1}a_{u_2,v_1}\right)E\{\omega_{u_1}\}E\{\omega_{u_2}\}\\
&=&(\tau_k^2+ 4C_k(t)^2)E\{x\}^T(A-A_0)^2E\{x\}\\
&\le&2(\tau_k^2+ 4C_k(t)^2)E\{x\}^T(A^2+A_0^2)E\{x\},
\end{eqnarray*}
where recall $A_0=\textrm{diag}(a_{1,1},\ldots,a_{c,c})$.
Since $A\le I_c$ and $a_{1,1}=\cdots=a_{c,c}\asymp 1/(ch)=o(1)$,
we have $A^2+A_0^2\le 2I_c$ (as $c\rightarrow\infty$), which, together with (\ref{power:lemma:2:eqn:1}), further leads to
\begin{eqnarray*}
T_1&\le& 4(\tau_k^2+ 4C_k(t)^2) E\{x\}^TE\{x\}\\
&\le& 4(\tau_k^2+ 4C_k(t)^2)\sum_{i=1}^c(|f(i/c)|+4C_k(t)+\zeta)^2\\&\le& 4(1+(ch^2)^{-1})(\tau_k^2+ 4C_k(t)^2)\sum_{i=1}^c(|f(i/c)|+4C_k(t)+\zeta)^2.
\end{eqnarray*}
For $T_2$, we have
\begin{eqnarray*}
T_2&=&(\tau_k^2+ 4C_k(t)^2)\sum_{u_1=1}^c\left(\sum_{v_1\neq u_1}a_{u_1,v_1}^2\right)E\{\omega_{u_1}^2\}\\
&\le&(\tau_k^2+ 4C_k(t)^2)\sum_{u_1=1}^c\left(\sum_{v_1=1}^ca_{u_1,v_1}^2\right)(|f_{u_1}|+4C_k(t)+\zeta)^2\\
&\le&(1+(ch^2)^{-1})(\tau_k^2+ 4C_k(t)^2)\sum_{i=1}^c(|f(i/c)|+4C_k(t)+\zeta)^2,
\end{eqnarray*}
where the last inequality follows from
\begin{eqnarray*}
\sum_{v=1}^ca_{i,v}^2=\frac{1}{c}\sum_{r=0}^{c-1}\frac{\lambda_{d,r}^2}{(\lambda+\lambda_{d^\prime,r})^4}
\lesssim (ch)^{-1}\rightarrow0.
\end{eqnarray*}
Here the above ``$\lesssim$'' is uniformly of $1\le i\le c$.

For $T_3$,  Cauchy inequality implies that
\begin{eqnarray*}
T_3&\le&\left(\sum_{i=1}^ca_{i,i}E\{\omega_i z_i^0\}\right)^2\\
&\le&c\sum_{i=1}^ca_{i,i}^2|E\{\omega_i z_i^0\}|^2\\
&\le&c(\tau_k^2+ 4C_k(t)^2)\sum_{i=1}^ca_{i,i}^2(|f(i/c)|+4C_k(t)+\zeta)^2\\
&\le&(1+(ch^2)^{-1})(\tau_k^2+ 4C_k(t)^2)\sum_{i=1}^c(|f(i/c)|+4C_k(t)+\zeta)^2,
\end{eqnarray*}
where the last inequality follows from
$ca_{1,1}^2=\ldots=ca_{c,c}^2\asymp (ch^2)^{-1}$, as $c\rightarrow\infty$.

For $T_4$, we have
\begin{eqnarray*}
T_4&=&\sum_{i\neq v}a_{i,i}a_{v,i}E\{\omega_i|z_i^0|^2\}E\{\omega_v\}\\
&\le&\sum_{i=1}^c a_{i,i} E\{|\omega_i|\cdot |z_i^0|^2\}\sum_{v\neq i}|a_{v,i}| E\{|\omega_v|\}\\
&\lesssim&\frac{\tau_k^2+ 4C_k(t)^2}{ch}\sum_{i=1}^c(|f(i/c)|+4C_k(t)+\zeta)\sqrt{\sum_{v=1}^c a_{v,i}^2}
\sqrt{\sum_{i=1}^c(|f(i/c)|+4C_k(t)+\zeta)^2}\\
&\lesssim&\frac{\tau_k^2+ 4C_k(t)^2}{(ch)^{3/2}}\sqrt{c}\sum_{i=1}^c(|f(i/c)|+4C_k(t)+\zeta)^2\\
&\lesssim&\frac{\tau_k^2+ 4C_k(t)^2}{ch^{3/2}}\sum_{i=1}^c(|f(i/c)|+4C_k(t)+\zeta)^2\\
&\lesssim&\frac{\tau_k^2+ 4C_k(t)^2}{ch^{2}}\sum_{i=1}^c(|f(i/c)|+4C_k(t)+\zeta)^2\\
&\lesssim&(1+(ch^2)^{-1})(\tau_k^2+ 4C_k(t)^2)\sum_{i=1}^c(|f(i/c)|+4C_k(t)+\zeta)^2.
\end{eqnarray*}

For $T_5$, it holds that
\begin{eqnarray*}
T_5&=&\sum_{i=1}^c a_{i,i}^2E\{\omega_i^2|z_i^0|^2\}\\
&\lesssim&\frac{\tau_k^2+ 4C_k(t)^2}{(ch)^2}\sum_{i=1}^c(|f(i/c)|+4C_k(t)+\zeta)^2\\
&\le& (1+(ch^2)^{-1})(\tau_k^2+ 4C_k(t)^2)\sum_{i=1}^c(|f(i/c)|+4C_k(t)+\zeta)^2.
\end{eqnarray*}
From the above analysis of $T_1$ through $T_5$, we get that as $c\rightarrow\infty$, 
for any $g\in S^m(\mathbb{I})$ with $J(g)\le\rho^2$, it follows that
\[
E\{|(\widetilde{z}-\widetilde{z}^0)^TAz^0|^2\}\le 8(1+(ch^2)^{-1})(\tau_k^2+ 4C_k(t)^2)\sum_{i=1}^c(|f(i/c)|+4C_k(t)+\zeta)^2.
\]
This proves the desired result.
\end{proof}
For $\nu=1,2,\ldots,c$, define $\Phi_\nu=(\varphi_\nu(1/c),\varphi_\nu(2/c),\ldots,\varphi_\nu(c/c))^T$.
Let $\varepsilon=\exp(2\pi\sqrt{-1}/c)$,
\[
x_r=\frac{1}{\sqrt{c}}(1,\varepsilon^r,\varepsilon^{2r},\ldots,\varepsilon^{(c-1)r})^T,
\,\,r=0,1,\ldots,c-1,
\]
and $x_r^\ast $ be the conjugate transpose of $x_r$.
\begin{lemma}\label{power:thm:lemma:3}
For $0\le r\le c-1$ and $1\le v\le c-1$, one has that
\begin{eqnarray*}
x_r^\ast\Phi_{2(pc+v)-1}&=&\sqrt{\frac{c}{2}}\left(\varepsilon^v I(r=v)+\varepsilon^{-v}I(r+v=c)\right),\\
x_r^\ast\Phi_{2(pc+v)}&=&\sqrt{-\frac{c}{2}}\left(\varepsilon^v I(r=v)-\varepsilon^{-v}I(r+v=c)\right);
\end{eqnarray*}
and
\begin{eqnarray*}
x_r^\ast\Phi_{2(pc+c)-1}&=&\sqrt{\frac{c}{2}}I(r=0),\\
x_r^\ast\Phi_{2(pc+c)}&=&0.
\end{eqnarray*}
\end{lemma}
\begin{proof}\textbf{\hspace*{-1mm}:}
The proof can be accomplished by direct calculations.
For instance, the first case holds by following arguments.
For $0\le r\le c-1$ and $1\le v\le c-1$,
\begin{eqnarray*}
x_r^\ast\Phi_{2(pc+v)-1}&=&\frac{1}{\sqrt{2c}}\sum_{i=1}^c\varepsilon^{-r(i-1)}\cos\left(\frac{2\pi v i}{c}\right)\\
&=&\frac{1}{\sqrt{2c}}\sum_{i=1}^c\varepsilon^{-r(i-1)}(\varepsilon^{vi}+\varepsilon^{-vi})\\
&=&\frac{1}{\sqrt{2c}}\sum_{i=0}^{c-1}\varepsilon^{-(r-v)i}\varepsilon^v
+\frac{1}{\sqrt{2c}}\sum_{i=0}^{c-1}\varepsilon^{-(r+v)i}\varepsilon^{-v}\\
&=&\sqrt{\frac{c}{2}}\left(\varepsilon^v I(r=v)+\varepsilon^{-v}I(r+v=c)\right).
\end{eqnarray*}
The proof of other cases is similar.
\end{proof}
Let $M=(x_0,x_1,\ldots,x_{c-1})$ and $M^\ast \textbf{f}=(e_0(f),e_1(f),\ldots,e_{c-1}(f))^T$,
where $\textbf{f}=(f(1/c),\ldots,f(c/c))^T$.
Recall $M^\ast$ is the conjugate transpose of $M$. Suppose $f\in S^m(\mathbb{I})$
admits Fourier expansion $f=\sum_{\nu=1}^\infty f_\nu\varphi_\nu$.
\begin{lemma}\label{power:thm:lemma:4}
There exists a universal constant $\varrho>0$ s.t.
for any $f\in S^m(\mathbb{I})$,
\[
\textbf{f}^T (I_c-A)\textbf{f}\le \varrho c(\lambda+c^{-2m})J(f).
\]
\end{lemma}

\begin{proof}\textbf{\hspace*{-1mm}:}
For simplicity, denote $e_r=e_r(f)$.
For $1\le r\le c/2$, we have
\begin{eqnarray*}
\lambda_{d^{\prime},r}^2-\lambda_{d,r}&\le&\left((2\pi r)^{-2m}+(2\pi(c-r))^{-2m}+2\bar{d}^{\prime}_m(2\pi c)^{-2m}\right)^2
-(2\pi r)^{-4m}\\
&\le&\left((2\pi r)^{-2m}+(1+2^{1-2m}\bar{d}^{\prime}_m)(\pi c)^{-2m}\right)^2-(2\pi r)^{-4m}\\
&=&2(1+2^{1-2m}\bar{d}^{\prime}_m)(2\pi r)^{-2m}(\pi c)^{-2m}+(1+2^{1-2m}\bar{d}^{\prime}_m)^2(\pi c)^{-4m}.
\end{eqnarray*}
Therefore, it follows that
\begin{eqnarray}
&&1-\frac{\lambda_{d,r}}{(\lambda+\lambda_{d^{\prime},r})^2}\nonumber\\
&=&\frac{\lambda^2+2\lambda\lambda_{d^{\prime},r}+\lambda_{d^{\prime},r}^2-\lambda_{d,r}}{(\lambda+\lambda_{d^{\prime},r})^2}\nonumber\\
&\le&\frac{2\lambda}{\lambda+\lambda_{d^{\prime},r}}+\frac{\lambda_{d^{\prime},r}^2-\lambda_{d,r}}{(\lambda+\lambda_{d^{\prime},r})^2}\nonumber\\
&\le&\left(2\lambda+2(1+2^{1-2m}\bar{d}^{\prime}_m)(\pi c)^{-2m}
+2^{2m}(1+2^{1-2m}\bar{d}^{\prime}_m)^2(\pi c)^{-2m}\right)(2\pi r)^{2m}\nonumber\\
&\le&\varrho_m'(\lambda+c^{-2m})(2\pi r)^{2m},\label{point:0}
\end{eqnarray}
where $\varrho_m'=\max\{2,(2(1+2^{1-2m}\bar{d}^{\prime}_m)+2^{2m}(1+2^{1-2m}\bar{d}^{\prime}_m)^2)\pi^{-2m}\}$, and $\bar{d}^{\prime}_m$ is defined in (\ref{cmdm}).

By Lemma \ref{power:thm:lemma:3} and direct calculations, for $1\le r\le c-1$, we have
\begin{eqnarray*}
e_r&=&\sum_{p=0}^\infty\sum_{v=1}^c f_{2(pc+v)-1} x_r^\ast\Phi_{2(pc+v)-1}
+\sum_{p=0}^\infty\sum_{v=1}^\infty f_{2(pc+v)}x_r^\ast\Phi_{2(pc+v)}\\
&=&\sum_{p=0}^\infty\sum_{v=1}^{c-1}
f_{2(pc+v)-1}\left(\sqrt{\frac{c}{2}}\varepsilon^v I(r=v)+\sqrt{\frac{c}{2}}\varepsilon^{-v}I(v+r=c)\right)\\
&&+\sum_{p=0}^\infty\sum_{v=1}^{c-1} f_{2(pc+v)}
\left(\sqrt{-\frac{c}{2}}\varepsilon^v I(r=v)-\sqrt{-\frac{c}{2}}\varepsilon^{-v}I(r+v=c)\right)\\
&=&\varepsilon^r\sqrt{\frac{c}{2}}\sum_{p=0}^\infty\left(
f_{2(pc+r)-1}+f_{2(pc+c-r)-1}+\sqrt{-1}f_{2(pc+r)}-\sqrt{-1}f_{2(pc+c-r)}\right).
\end{eqnarray*}
Therefore, it holds that
\begin{eqnarray*}
|e_r|^2&=&\frac{c}{2}\bigg|\sum_{p=0}^\infty\left(f_{2(pc+r)-1}+f_{2(pc+c-r)-1}\right)\bigg|^2
+\frac{c}{2}\bigg|\sum_{p=0}^\infty\left(f_{2(pc+r)}-f_{2(pc+c-r)}\right)\bigg|^2.
\end{eqnarray*}
It is easy to see that
\[
|e_r|^2=|e_{c-r}|^2,\,\,r=1,\ldots,c-1.
\]
For $1\le r\le c/2$, we have
\begin{eqnarray}
|e_r|^2&\le&c\sum_{p=0}^\infty f_{2(pc+r)-1}^2(2\pi(pc+r))^{2m}\sum_{p=0}^\infty(2\pi(pc+r))^{-2m}\nonumber\\
&&+c\sum_{p=0}^\infty f_{2(pc+c-r)-1}^2(2\pi(pc+c-r))^{2m}\sum_{p=0}^\infty(2\pi(pc+c-r))^{-2m}\nonumber\\
&&+c\sum_{p=0}^\infty f_{2(pc+r)}^2(2\pi(pc+r))^{2m}\sum_{p=0}^\infty(2\pi(pc+r))^{-2m}\nonumber\\
&&+c\sum_{p=0}^\infty f_{2(pc+c-r)}^2(2\pi(pc+c-r))^{2m}\sum_{p=0}^\infty
(2\pi(pc+c-r))^{-2m}\nonumber\\
&\le&\left(c\sum_{p=0}^\infty f_{2(pc+r)-1}^2(2\pi(pc+r))^{2m}+c\sum_{p=0}^\infty
f_{2(pc+c-r)-1}^2(2\pi(pc+c-r))^{2m}\right.\nonumber\\
&&\left.+c\sum_{p=0}^\infty f_{2(pc+r)}^2(2\pi(pc+r))^{2m}+c\sum_{p=0}^\infty f_{2(pc+c-r)}^2(2\pi(pc+c-r))^{2m}\right)\nonumber\\
&&\times \frac{2m}{2m-1}(2\pi r)^{-2m},\label{point:1}
\end{eqnarray}
where (\ref{point:1}) follows by an elementary inequality
\begin{equation*}
\sum_{p=0}^\infty (2\pi(pc+r))^{-2m}\le\frac{2m}{2m-1}(2\pi r)^{-2m},\,\,1\le r\le c/2.
\end{equation*}
Meanwhile,  a similar analysis leads to
\begin{eqnarray}
|e_0|^2&=&\frac{c}{2}\bigg|\sum_{p=0}^\infty f_{2(pc+c)-1}\bigg|^2\nonumber\\
&\le&\frac{c}{2}\sum_{p=0}^\infty f_{2(pc+c)-1}^2(2\pi(pc+c))^{2m}\sum_{p=0}^\infty(2\pi(pc+c))^{-2m}\nonumber\\
&=&\frac{c^{1-2m}}{2}\sum_{p=0}^\infty f_{2(pc+c)-1}^2(2\pi(pc+c))^{2m}
\times \sum_{p=1}^\infty(2\pi p)^{-2m}.\label{point:2}
\end{eqnarray}

Now it follows from (\ref{point:0}), (\ref{point:1}) and (\ref{point:2}),
and elementary facts $\lambda_{d^{\prime},r}=\lambda_{d^{\prime},c-r}$
and $\lambda_{d,r}=\lambda_{d,c-r}$, for $1\le r\le c-1$, that
\begin{eqnarray*}
&&\textbf{f}^T(I_c-A)\textbf{f}\\
&=&\sum_{r=0}^{c-1}\left(1-\frac{\lambda_{d,r}}{(\lambda+\lambda_{d^{\prime},r})^2}\right)|e_r|^2\\
&=&\left(1-\frac{\lambda_{d,0}}{(\lambda+\lambda_{d^{\prime},0})^2}\right)|e_0|^2+
\sum_{1\le r\le c/2}\left(1-\frac{\lambda_{d,r}}{(\lambda+\lambda_{d^{\prime},r})^2}\right)|e_r|^2
+\sum_{c/2<r\le c-1}\left(1-\frac{\lambda_{d,r}}{(\lambda+\lambda_{d^{\prime},r})^2}\right)|e_r|^2\\
&\le&|e_0|^2+2\sum_{1\le r\le c/2}\left(1-\frac{\lambda_{d,r}}{(\lambda+\lambda_{d^{\prime},r})^2}\right)|e_r|^2\\
&\le&\frac{c^{1-2m}}{2}\sum_{p=0}^\infty f_{2(pc+c)-1}^2(2\pi(pc+c))^{2m}
\times \sum_{p=1}^\infty(2\pi p)^{-2m}\\
&&+2\varrho_m'c(\lambda+c^{-2m})\frac{2m}{2m-1}
\sum_{1\le r\le c/2}
\left(\sum_{p=0}^\infty f_{2(pc+r)-1}^2(2\pi(pc+r))^{2m}\right.\\
&&\left.+\sum_{p=0}^\infty
f_{2(pc+c-r)-1}^2(2\pi(pc+c-r))^{2m}
+\sum_{p=0}^\infty f_{2(pc+r)}^2(2\pi(pc+r))^{2m}\right.\\
&&\left.+\sum_{p=0}^\infty f_{2(pc+c-r)}^2(2\pi(pc+c-r))^{2m}\right)\\
&\le& \varrho_m c(\lambda+c^{-2m})\sum_{\nu=1}^\infty\left(f_{2\nu-1}^2+f_{2\nu}^2\right)(2\pi\nu)^{2m}\\
&=&\varrho_m c(\lambda+c^{-2m})J(f),
\end{eqnarray*}
where $\varrho_m=\max\{\sum_{p=1}^\infty(2\pi p)^{-2m}/2,4m \varrho_m'/(2m-1)\}$. It is straightforward to see $\varrho_m$ is a decreasing function with respect to $m$, therefore, we choose $\varrho = \varrho_{m=1}$. This proves Lemma \ref{power:thm:lemma:4}.
\end{proof}

The proof of Theorem \ref{adaptive:size} requires some recent Gaussian approximation result, i.e., Theorem 3.1 in \cite{adaptiveapproximation}.

\begin{lemma}\label{adaptive}
For each $c \in \mathbb{N}$, let $\boldsymbol{\Psi}_{c}$ be an $c$-dimensional centered Gaussian vector with covariance matrix $\Sigma_{c}=\left(\Sigma_{c}(m, m')\right)_{1 \leq m, m' \leq c}$ and $m_{u} \geq 2$ be an integer. Also, for each $m=m_l, \ldots, m_{u}$, let $A_{m}$ be an $c \times c$ symmetric matrix and $Z_{c}=\left(Z_{c, m_l}, \ldots, Z_{c, m_{u}}\right)^{\top}$ be an $m_{u}-m_l+1$ -dimensional centered Gaussian vector with covariance matrix $\mathfrak{C}_{c}=\left(\mathfrak{C}_{c}(m, m')\right)_{m_l \leq m, m' \leq m_{u}} .$ Set $F_{c, m}:=\boldsymbol{\Psi}_{c}^{\top} A_{m} \boldsymbol{\Psi}_{c}-E\left[\boldsymbol{\Psi}_{c}^{\top} A_{m} \boldsymbol{\Psi}_{c}\right]$ and suppose that the following conditions are satisfied:
\begin{enumerate}
    \item There is a constant $b>0$ such that $\mathfrak{C}_{c}(m, m) \geq b$ for every $c$ and every $m=m_l, \ldots, m_{u} .$
    \item $\max _{m_l \leq m \leq m_{u}}\left(E\left[F_{c, m}^{4}\right]-3 E\left[F_{c, m}^{2}\right]^{2}\right) \log ^{6} m_{c} \rightarrow 0$ as $c \rightarrow \infty$.
    \item $\max _{m_l \leq m, m' \leq m_{u}}\left|\mathfrak{C}_{c}(m, m')-E\left[F_{c, m} F_{c, m'}\right]\right| \log ^{2} m_{c} \rightarrow 0$ as $c \rightarrow \infty$.
\end{enumerate}

Then we have
$$
\sup _{x \in \mathbb{R}}\left|P\left(\max _{m_l \leq m \leq m_{u}} F_{c, m} \leq x\right)-P\left(\max _{m_l \leq m\leq m_{u}} Z_{c, m} \leq x\right)\right| \rightarrow 0, \,\,\textrm{as $c\rightarrow\infty$.}
$$
\end{lemma}
\begin{proof}
This is Theorem 3.1 in \cite{adaptiveapproximation}.
\end{proof}

The proof of Theorem \ref{adaptive:power} requires some rate conditions which are summarized in the following lemma.
\begin{lemma} \label{adaptive_rate}
Suppose $\lambda_m=a_n^{2m}n^{-4m/(4m+1)}\log(m_u)^{2m/(4m+1)}$, then for any $m_l < m < m_u \rightarrow \infty$, under Condition (\textbf{C}), the following rate conditions hold:
\[
h_{m}\log^2(m_u) \rightarrow 0; \, m_u\log^2(m_u)/ch_{m} \rightarrow 0;
\]
\[
\left(h_{m'} / h_{m}\right)^{-1 / 2}m_u^2\log^2(m_u) \rightarrow 0,\textrm{ for any } m_l\leq m<m' < m_u;
\]
\[
ch_m/\log(c)\rightarrow\infty; \, h_m^{1/2}\log(c)\rightarrow 0,
\]
where $h_m = \lambda^{\frac{1}{2m}} = a_nn^{-2/(4m+1)}\log(m_u)^{1/(4m+1)}$.
\end{lemma}

\begin{proof}
It is easy to see $h_{m_l}<h_m<h_{m_u}$. Therefore
\[
h_{m}\log^2(m_u) < h_{m_u}\log^2(m_u)\lesssim a_nn^{-2/(4m_u+1)}[\log(n)]^2 \rightarrow 0.
\]
\[
m_u\log^2(m_u)/ch_{m} <  m_u\log^2(m_u)/ch_{m_l} \lesssim n^{2/(4m_l+1)}\log (n)/(ca_n) \rightarrow 0.
\]
\begin{align*}
\left(h_{m'} / h_{m}\right)^{-1 / 2}m_u^2\log^2(m_u) &= n^{-\frac{4(m'-m)}{(4m+1)(4m'+1)}}(\log(m_u))^{-\frac{2m'-2m}{(4m+1)(4m'+1)}}m_u^2\log^2(m_u)\\
&\lesssim n^{-\frac{4}{(4m_u+1)^2}}(\log(m_u))^{-\frac{2}{(4m_u+1)^2}}m_u^2\log^2(m_u) \rightarrow 0.
\end{align*}
where the last ``$\rightarrow 0$'' follows from the assumption $m_u \lesssim \log^{d_0}(n)$ for some $d_0 \in (0, 1/2)$. For the last two terms, one has that
\[
ch_m/\log(c) > ch_{m_l}/\log(c) \gtrsim \frac{ca_n}{n^{2/(4m_l+1)}\log(n)}\rightarrow\infty.
\]
\[
h_m^{1/2}\log(c) < h_{m_u}^{1/2}\log(c) \lesssim a_n^{1/2}n^{-1/(4m_u+1)}\log(n) \rightarrow 0.
\]
\end{proof}

\section{Proofs for main theorems}
\begin{proof}\textbf{of Theorem \ref{upper:bound:hat:bb:f}:}
\newline
It holds that $\|\widehat{g}^{\textrm{B}}_{\mu, t, c}-g_0\|^2 \le 2\|\widehat{g}^{\textrm{B}}_{\mu, t, c}-\widehat{g}^{\textrm{ss}}_c\|^2 +
2\|\widehat{g}^{\textrm{ss}}_c-g_0\|^2$, and we analyze these two terms separately. We first analyze $\|\widehat{g}^{\textrm{B}}_{\mu, t, c}-\widehat{g}^{\textrm{ss}}_c\|^2$. Because
\begin{align*}
|Q(y_j) - y_j| \le C_k(t)\mathbbm{1}\left\{y_j\in \cup_{j=2}^{k-1}R_j(t) \right\} + |y_j - \mu_1|\mathbbm{1}\left\{y_j\in R_1(t)\right\} + |y_j - \mu_k|\mathbbm{1}\left\{y_j\in R_k(t)\right\},
\end{align*}
we have
\begin{align*}
(z_i - \widetilde{y}_i)^2 =&\left\{z_i - \frac{1}{\widetilde{n}}\sum_{j=(i-1)\widetilde{n}+1}^{i\widetilde{n}}Q(y_j) + \frac{1}{\widetilde{n}}\sum_{j=(i-1)\widetilde{n}+1}^{i\widetilde{n}}Q(y_j) - \widetilde{y_i}\right\}^2\\
\le& 2\left\{\frac{1}{\widetilde{n}}\sum_{j=(i-1)\widetilde{n}+1}^{i\widetilde{n}}\big(Q(y_j) - y_j\big)\right\}^2 + 2C_k(t)^2\\
\le& 2\Bigg\{\frac{1}{\widetilde{n}}\sum_{j=(i-1)\widetilde{n}+1}^{i\widetilde{n}}C_k(t)\mathbbm{1}\left\{y_j\in \cup_{j=2}^{k-1}R_j(t) \right\} \\
&+ |y_j - \mu_1|\mathbbm{1}\left\{y_j\in R_1(t)\right\} + |y_j - \mu_k|\mathbbm{1}\left\{y_j\in R_k(t)\right\}\Bigg\}^2 + 2C_k(t)^2\\
\le& 4C_k(t)^2 + \frac{2}{\widetilde{n}}\sum_{j=(i-1)\widetilde{n}+1}^{i\widetilde{n}}  (y_j - \mu_1)^2\mathbbm{1}\left\{y_j\in R_1(t)\right\} + \frac{2}{\widetilde{n}}\sum_{j=(i-1)\widetilde{n}+1}^{i\widetilde{n}}(y_j - \mu_k)^2\mathbbm{1}\left\{y_j\in R_k(t)\right\}.
\end{align*}

Therefore, from Lemma \ref{a:preliminary:thm}, we have
\begin{eqnarray}\label{thm:Gnk:eqn1}
E\|\widehat{g}^{\textrm{B}}_{\mu, t, c}
-\widehat{g}^{\textrm{ss}}_c\|^2 &\le& c^{-1}\sum_{i=1}^cE\left\{(z_i-\widetilde{y}_i)^2\right\}\nonumber\\
&\le& 4C_k(t)^2 + \frac{2}{n}\sum_{i=1}^{n}  E\left\{(y_i - \mu_1)^2\mathbbm{1}\left\{y_i\in R_1(t)\right\}\right\} + \frac{2}{n}\sum_{i=1}^{n}E\left\{(y_i - \mu_k)^2\mathbbm{1}\left\{y_i\in R_k(t)\right\}\right\}.\nonumber
\end{eqnarray}
On the other hand, by elementary calculations we have
\begin{eqnarray*}
&&\frac{2}{n}\sum_{i=1}^{n}  E\left\{(y_i - \mu_1)^2\mathbbm{1}\left\{y_i\in R_1(t)\right\}\right\} = \frac{2}{n}\sum_{i=1}^{n}\int_{-\infty}^{t_1}(z-\mu_1)^2p\left(\frac{z-g_0(i/n)}{\sigma}\right)\sigma^{-1}dz,\\
&&\frac{2}{n}\sum_{i=1}^{n}E\left\{(y_i - \mu_k)^2\mathbbm{1}\left\{y_i\in R_k(t)\right\}\right\} = \frac{2}{n}\sum_{i=1}^{n}\int_{t_{k-1}}^{\infty}(z-\mu_k)^2p\left(\frac{z-g_0(i/n)}{\sigma}\right)\sigma^{-1}dz,
\end{eqnarray*}
where $p(\cdot)$ is the distribution of $\epsilon$. Combining the above, we get 
\begin{equation} \label{eq_proof_G_c_k}
E\|\widehat{g}^{\textrm{B}}_{\mu, t, c}
-\widehat{g}^{\textrm{ss}}_c\|^2 \le G_{c,k}(t).    
\end{equation}

Next, we analyze the mean square error of the second term $\|\widehat{g}^{\textrm{ss}}_c-g_0\|^2$. For the sake of theoretical investigation, we introduce the following function,
\[
g_\textrm{new} = \sum_{i=1}^c\theta_{\textrm{new}, i}  K_{i/c},
\]
where $(\theta_{\textrm{new}, 1},\ldots,\theta_{\textrm{new}, c})^T=c^{-1}(\Sigma_c+\lambda I_c)^{-1}\widetilde{f}$ with $\widetilde{f}=(f(1/c),\ldots,f(c/c))^T\in \mathbb{R}^c$, and $f(x)$ is the integral function of $g_0$ as defined in (\ref{eq:f0}), i.e.,
\[
f(x)=\frac{1}{2\Delta}\int_{\max(x-\Delta, 0)}^{\min(x+\Delta, 1)}g_0(s)ds, \text{ where } x\in[0,1], \text{ and }\Delta=\frac{1}{c}.
\]
Recall that $\widehat{g}^{\textrm{ss}}_c = \sum_{i=1}^c\widetilde{\theta}_i K_{i/c} = c^{-1}\sum_{\nu=1}^\infty\frac{\Phi_\nu^T(\Sigma_c+\lambda I_c)^{-1}\widetilde{y}}{\gamma_\nu}\varphi_\nu$, where  $(\widetilde{\theta}_1,\ldots,\widetilde{\theta}_c)^T=c^{-1}(\Sigma_c+\lambda I_c)^{-1}\widetilde{y}$ with $\widetilde{y}=(\widetilde{y}_1,\ldots,\widetilde{y}_c)^T$, $\varphi_{2k-1}(x)=\sqrt{2}\cos(2\pi kx),\,\,\,\,
\varphi_{2k}(x)=\sqrt{2}\sin(2\pi kx)$ are the trigonometric basis functions, $\gamma_{2k-1}=\gamma_{2k}=(2\pi k)^{2m}$, and $\Phi_\nu=(\varphi_\nu(1/c),\varphi_\nu(2/c),\ldots,\varphi_\nu(c/c))^T$.
Therefore, we have
\begin{eqnarray}
\|E(\widehat{g}^{\textrm{ss}}_c) - g_\textrm{new}\|^2&=&c^{-1}(\widetilde{g} - \widetilde{f})^T(\Sigma_c+\lambda
I_c)^{-1}\Omega_c(\Sigma_c+\lambda I_c)^{-1}(\widetilde{g} - \widetilde{f}) \nonumber\\
&\leq& c^{-1}(\widetilde{g} - \widetilde{f})^T(\widetilde{g} - \widetilde{f})\nonumber\\
&=& O(n^{-2}), \label{thm1:eq:g_new_E_g_hat}
\end{eqnarray}
where $\widetilde{g} = (\widetilde{g}_1, \ldots, \widetilde{g}_c)^T$ and $\widetilde{g}_i = \frac{\sum_{j=(i-1)\widetilde{n}+1}^{i\widetilde{n}}g_0(j/n)}{\widetilde{n}}, i=1,\ldots,c$. Next, we evaluate $E(\|\widehat{g}^{\textrm{ss}}_c-E(\widehat{g}^{\textrm{ss}}_c)\|^2)$. Note that
$\Sigma_c, \Omega_c$ can be decomposed as $ \Sigma_c=M\Lambda_{d^{\prime}} M^\ast,\,\,\Omega_c=M\Lambda_dM^\ast$, as defined in (\ref{Sigma:Sigma2:decompose}). Furthermore, we let $\widetilde{y}^0 = (\widetilde{y}^0_1, \ldots, \widetilde{y}^0_c)^T$, where $\widetilde{y}^0_i = \frac{1}{\widetilde{n}}\sum_{j=(i-1)\widetilde{n}+1}^{i\widetilde{n}}\sigma\epsilon_j$. Hence, we obtain

\begin{align*}
E(\|\widehat{g}^{\textrm{ss}}_c-E(\widehat{g}^{\textrm{ss}}_c)\|^2)&=\frac{1}{c^{2} \gamma_{\nu}^{2}}\sum_{\nu=1}^{\infty}  E\left\{\left|\Phi_{\nu}^T\left(\Sigma_c+\lambda I_{c}\right)^{-1} \widetilde{y}^0\right|^{2}\right\}\\
&=\frac{\sigma^2}{c^{2} \gamma_{\nu}^{2}}\sum_{\nu=1}^{\infty}  \operatorname{trace}\left(\left(\Sigma_c+\lambda I_{c}\right)^{-1} \Phi_{\nu} \Phi_{\nu}^{T}\left(\Sigma_c+\lambda I_{c}\right)^{-1}\right)E\widetilde{\epsilon}^2\\
&=\frac{\sigma^2}{n} \operatorname{trace}\left(\left(\Sigma_c+\lambda I_{c}\right)^{-1} \sum_{\nu=1}^{\infty} \frac{\Phi_{\nu} \Phi_{\nu}^T / c}{\gamma_{\nu}^{2}}\left(\Sigma_c+\lambda I_{c}\right)^{-1}\right)\\
&=\frac{\sigma^2}{n}\operatorname{trace}\left(\left(\Sigma_c+\lambda I_{c}\right)^{-1} \Omega_c\left(\Sigma_c+\lambda I_{c}\right)^{-1}\right)\\
&=\frac{\sigma^2}{n} \operatorname{trace}\left(M\left(\Lambda_{d^{\prime}}+\lambda I_{c}\right)^{-1} \Lambda_{d}\left(\Lambda_{d^{\prime}}+\lambda I_{c}\right)^{-1} M^{*}\right)\\
&=\frac{\sigma^2}{n} \sum_{l=0}^{c-1} \frac{\lambda_{d, l}}{\left(\lambda+\lambda_{d^{\prime}, l}\right)^{2}}.
\end{align*}
By expressions of $\lambda_{d,l}$'s, the above is upper bounded by the following
\begin{align*}
&\frac{2 \sigma^2\bar{d}_{m}}{n\left(2 \bar{d}^{\prime}_{m}+(2 \pi c)^{2 m} \lambda\right)^{2}}+\sigma^2\left(1+\bar{d}_{m}\right) n^{-1} \sum_{l=1}^{c-1} \frac{(2 \pi(c-l))^{-4 m}+(2 \pi l)^{-4 m}}{\left(\lambda+(2 \pi(c-l))^{-2 m}+(2 \pi l)^{-2 m}\right)^{2}}\\
&\leq \frac{2 \sigma^2\bar{d}_{m}}{n\left(2 \bar{d}^\prime_{m}+(2 \pi c)^{2 m} \lambda\right)^{2}}+2\sigma^2\left(1+\bar{d}_{m}\right) n^{-1} \sum_{1 \leq l \leq n / 2} \frac{(2 \pi l)^{-4 m}+(2 \pi(c-l))^{-4 m}}{\left(\lambda+(2 \pi l)^{-2 m}+(2 \pi(c-l))^{-2 m}\right)^{2}}\\
&\leq \frac{2 \sigma^2\bar{d}_{m}}{n\left(2 \bar{d}^{\prime}_{m}+(2 \pi c)^{2 m} \lambda\right)^{2}}+4\sigma^2\left(1+\bar{d}_{m}\right) n^{-1} \sum_{1 \leq l \leq n / 2} \frac{(2 \pi l)^{-4 m}}{\left(\lambda+(2 \pi l)^{-2 m}\right)^{2}}\\
&\leq \frac{2 \sigma^2\bar{d}_{m}}{n\left(2 \bar{d}^{\prime}_{m}+(2 \pi c)^{2 m} \lambda\right)^{2}}+\frac{2\sigma^2\left(1+\bar{d}_{m}\right)}{\pi n h} \int_{0}^{\pi c h} \frac{1}{\left(1+x^{2 m}\right)^{2}} d x\\
&\leq b_{m}\left(\frac{\sigma^2}{n}+\frac{\sigma^2}{n h} \int_{0}^{\pi c h} \frac{1}{\left(1+x^{2 m}\right)^{2}} d x\right),
\end{align*}
where $b_m \ge 1$ is a constant only depending on $m$. From the above analysis, we obtain
\begin{align*}
E(\|\widehat{g}^{\textrm{ss}}_c-E(\widehat{g}^{\textrm{ss}}_c)\|^2) \lesssim n^{-1} + (nh)^{-1}.
\end{align*}
Using above analysis and (\ref{thm1:eq:g_new_E_g_hat}), we have
\begin{align} \label{thm1:term1}
E(\|\widehat{g}^{\textrm{ss}}_c-g_{\textrm{new}}\|^2) \lesssim n^{-1} + (nh)^{-1}.
\end{align}

Now, we consider the difference between original regression function $g_0$ and the integral function $f$ defined in (\ref{eq:f0}), i.e., $\|f-g_0\|^2$. By definition, for $t \in [\Delta, 1-\Delta]$, there exists  $t^\prime$  between $t-\Delta$ and $t+\Delta$ such that
\begin{eqnarray*}
f(t) - g_0(t) &=& \frac{1}{2\Delta}\int_{t-\Delta}^{t+\Delta}\left(g_0(s) - g_0(t)\right) ds \\
&=& \frac{1}{2\Delta}\int_{t-\Delta}^{t+\Delta}\left(g_0'(t)(s-t) + \frac{g_0''(t^\prime)}{2}(s-t)^2\right) ds\\
&=& \frac{1}{2\Delta}\int_{t-\Delta}^{t+\Delta}\frac{g_0''(t^\prime)}{2}(s-t)^2ds = \frac{\Delta^2}{3}g''_0(t^\prime).
\end{eqnarray*}
On the other hand, for $t \in [0, \Delta]$, there exists $t^\prime$  between $0$ and $t+\Delta$ such that
\begin{eqnarray*}
f(t) - g_0(t) &=& \frac{1}{2\Delta}\int_{0}^{t+\Delta}\left(g_0(s) - g_0(t)\right) ds \\
&=& \frac{1}{2\Delta}\int_{0}^{t+\Delta}\left(g_0'(t^\prime)(s-t)\right) ds \leq  \frac{\Delta}{4}g'_0(t^\prime).
\end{eqnarray*}
In a similar way, we obtain $f(t) - g_0(t) \leq  \frac{\Delta}{4}g'_0(t^\prime)$ for $t \in [t-\Delta, 1]$ and some $t^{\prime} \in [t-\Delta, 1]$. Therefore, by Sobolev inequality, we know $\int^1_0|g_0''(t)|^2dt\leq\int^1_0|g_0^{(m)}(t)|^2dt<\infty$ and $\int^1_0|g_0'(t)|^2dt\leq\int^1_0|g_0^{(m)}(t)|^2dt<\infty$, which implies
\begin{align} \label{thm1:term3}
\|f-g_0\|^2 &= \int_0^1|f(t) - g(t)|^2dt \nonumber\\
 &= \int_{t-\Delta}^{t+\Delta}|f(t) - g(t)|^2dt + \int_{0}^{t+\Delta}|f(t) - g(t)|^2dt + \int_{t-\Delta}^{1}|f(t) - g(t)|^2dt  \nonumber \\
&= O(c^{-3}).
\end{align}

In the end, because both $g_{\textrm{new}}$ and $f$ belong to Sobolev space, and $g_{\textrm{new}}$ can be viewed as the approximate error of spline estimates with respect to $f$ without random error. By classical spline theory (\citep{W90}), we know
\begin{align} \label{thm1:term2}
\|g_{\textrm{new}}-f\|^2 = O(c^{-2m} + \lambda).
\end{align}
As a consequence, from (\ref{thm1:term1}), (\ref{thm1:term3}), and (\ref{thm1:term2}), we have $E\|\widehat{g}^{\textrm{ss}}_c-g_0\|^2 \leq 3E(\|\widehat{g}^{\textrm{ss}}_c-g_{\textrm{new}}\|^2)+ 3\|g_{\textrm{new}}-f\|^2 + 3\|f-g_0\|^2=O\left((nh)^{-1} + c^{-3} + c^{-2m} + \lambda\right)$. Combining the result in (\ref{eq_proof_G_c_k}), we get the desired result.
\end{proof}

\begin{proof}\textbf{of Corollary \ref{coro2}:}
Because as $|t_1| \to \infty$,
\begin{align*}
G_{c,k,1}(t)&=\frac{1}{n}\sum_{i=1}^{n}\int_{-\infty}^{t_1}(z-\mu_1)^2\sigma^{-1}p\left(\frac{z-g_0(i/n)}{\sigma}\right)dz\\
&\lesssim 2n^{-1}\sigma^{-1}\int_{-\infty}^{\frac{t_1-g_0(i/n)}{\sigma}}\sum_{i=1}^nz^2p(z)dz + 2n^{-1}\sigma^{-1}\int_{-\infty}^{\frac{t_1-g_0(i/n)}{\sigma}}\sum_{i=1}^n \mu_1^2p(z)dz.
\end{align*}
The first term in the above equation is bounded by $n^{-2m/(2m+1)}$ because $p(z)$ satisfies $\int_{|z|\ge T} z^2p(z)dz=O(\exp(-T^d))$ and $d \geq \frac{4m}{2m+1}$, $|t_1| \asymp \sqrt{\log(n)}$. Due to Condition (\textbf{B}), we know $G_{c,k,1}(t) = O(n^{-2m/(2m+1)})$. Similarly, we know $G_{c,k,2}(t) = O(n^{-2m/(2m+1)})$. Hence $G_{c,k}(t)=C_k(t)^2+G_{c,k,1}(t)+G_{c,k,2}(t)=O(n^{-2m/(2m+1)})$.
The result follows by Theorem \ref{upper:bound:hat:bb:f} and $\lambda \asymp n^{-2m/(2m+1)}$, $c\asymp n^{\frac{\max\{1,2m/3\}}{2m+1}}$.
\end{proof}

\begin{proof}\textbf{of Theorem \ref{asymp:distribution}:}
 Suppose $z^\star_i$'s are the quantized samples corresponding to $\mu_j=\mu^\star_j$ for $1\le j\le k$, where $\mu_j^\star$ are defined by
\begin{equation}\label{mu:j:estimation}
\mu^\star_j=\frac{\sum_{i=1}^n E\{y_iI(y_i\in R_j(t))\}}{\sum_{i=1}^nP(y_i\in R_j(t))}.
\end{equation}
For $p>0$, define the $p$th order moment of the standardized $z^\star_i$:
\begin{equation}\label{trivial:fact}
m_p=E_{H_0}\{|z^\star_i/\tau^\star_k|^p\},
\end{equation}
where $E_{H_0}$ denotes the expectation under $H_0$ and $\tau^{\star2}_k = Var(z^\star_i|H_0)$. Because $|\mu_j - \mu_j^\star| \leq C_k(t)$ for $j=2, \ldots, k-1$, and under Condition (\textbf{B}), we have that $\tau_k^{\star2} = O(cn^{-1})$. Furthermore, since $b\gg\log_2\left(\sqrt{n\T_nh^{1/2}}\right)$, which implies that
$C_k(t)^4 \asymp \frac{\T_n^2}{2^{4b}} \ll \frac{\T_n^2}{n^2h\T^2_n} = (n^2h)^{-1} = o(c^2n^{-2})$ and the assumption that $E([\widetilde{n}^{-1}\sum_{j=1}^{\widetilde{n}}Q(\epsilon_j)]^4)=O(c^2n^{-2})$, one has that $m_p = O(1)$ for $p=3, 4$.

Define $z^{\textrm{sd}}_i=z^\star_i/\tau^\star_k$ for $i=1,\ldots,c$. Then $z^{\textrm{sd}}_i$ are \textit{iid} variables
with zero-mean and unit variance. Define $z^\star=(z^\star_1,\ldots,z^\star_c)^T$ and $z^{\textrm{sd}}=(z^{\textrm{sd}}_1,\ldots,z^{\textrm{sd}}_c)^T$.
Define $A_0=\textrm{diag}(a_{1,1},\ldots,a_{c,c})$
and $A_1=A-A_0$. Let $B=A_1/s_c$.
Define $\alpha_l=\frac{\lambda_{d,l}}{(\lambda+\lambda_{d^{\prime},l})^2}, l=0, \ldots, c-1$. Immediately,
for all $i=1,\ldots,c$,
$a_{i,i}=c^{-1}\sum_{l=0}^{c-1}\alpha_l\asymp 1/(ch)$,
therefore,
\begin{eqnarray*}
\sum_{i\neq i'}a_{i,i'}^2&=&\sum_{i,i'=1}^c a_{i,i'}^2-\sum_{i=1}^c a_{i,i}^2\\
&=&\textrm{trace}(A^2)-c\left(c^{-1}\sum_{l=0}^{c-1}\alpha_l\right)^2\\
&=&\sum_{l=0}^{c-1}\alpha_l^2-c\left(c^{-1}\sum_{l=0}^{c-1}\alpha_l\right)^2\\
&\asymp&h^{-1}(1-1/(ch))\asymp h^{-1},
\end{eqnarray*}
where the last ``$\asymp$" follows from condition $(ch)^{-1}=o(1)$.
This implies that $s_c^2\asymp h^{-1}$.
Furthermore,
\begin{equation}\label{asymp:dist:eqn:0}
\textrm{trace}(A^2)=\sum_{l=0}^{c-1}\alpha_l^2\asymp h^{-1}\,\,\,\,
\textrm{and}\,\,\,\,\textrm{trace}(A^4)=\sum_{l=0}^{c-1}\alpha_l^4\asymp h^{-1}.
\end{equation}

Let $T^\star_{\mu^\star, t, c}$ be the test statistic corresponding to $z^\star_i$'s. By (\ref{thm1:eqn1}) it can be shown that
$cT^\star_{\mu^\star,t,c}=z^{\star T} Az^\star$, which leads to that
\begin{eqnarray*}
\frac{cT^\star_{\mu^\star,t,c}-\sum_{i=1}^c a_{i,i}\tau^{\star2}_k}{s_c\tau^{\star2}_k}
&=&\frac{z^{\star T}Az^\star-\sum_{i=1}^ca_{i,i}\tau^{\star2}_k}{s_c\tau^{\star2}_k}\\
&=&\frac{(z^{\textrm{sd}})^TAz^{\textrm{sd}}-\sum_{i=1}^ca_{i,i}}{s_c}\\
&=&\frac{\sum_{i=1}^ca_{i,i}((z^{\textrm{sd}}_i)^2-1)+\sum_{1\le i\neq i'\le c}a_{i,i'}z^{\textrm{sd}}_iz^{\textrm{sd}}_{i'}}{s_c}\\
&=&\frac{\sum_{i=1}^ca_{i,i}((z^{\textrm{sd}}_i)^2-1)}{s_c}+\sum_{i\neq i'}b_{i,i'}z^{\textrm{sd}}_iz^{\textrm{sd}}_{i'}\equiv Q_1+Q_2.
\end{eqnarray*}

We first look at $Q_1$.
By (\ref{trivial:fact}) we have
\begin{eqnarray*}
E\{|\sum_{i=1}^ca_{i,i}((z^{\textrm{sd}}_i)^2-1)|^2\}/s_c^2=s_c^{-2}\sum_{i=1}^ca_{i,i}^2 (m_4-1)
\asymp \frac{m_4-1}{ch}=o(1),
\end{eqnarray*}
which leads to $Q_1=o_P(1)$.

Define $b_{i,i}=0$ for $i=1,\ldots,c$ and $B=[b_{i,i'}]_{1\le i,i'\le c}$.
We next analyze $Q_2$.
Note that $Q_2=(z^{\textrm{sd}})^T Bz^{\textrm{sd}}$.
Let $(\widetilde{z}^{\textrm{sd}}_1,\ldots,\widetilde{z}^{\textrm{sd}}_c)^T$ be an independent copy of $z^{\textrm{sd}}=(z^{\textrm{sd}}_1,\ldots,z^{\textrm{sd}}_c)^T$.
Let $I$ be uniform distributed on $\{1,2,\ldots,c\}$.
Throughout, we let $\widetilde{z}^{\textrm{sd}}_i$, $z^{\textrm{sd}}_i$ and $I$
be mutually independent.
Define $\widetilde{z}^{\textrm{sd}}=(z^{\textrm{sd}}_1,\ldots,z^{\textrm{sd}}_{I-1},\widetilde{z}^{\textrm{sd}}_I,z^{\textrm{sd}}_{I+1},\ldots,z^{\textrm{sd}}_c)^T$.
So $(z^{\textrm{sd}},\widetilde{z}^{\textrm{sd}})$ is an exchangeable pair (see \cite{RR09}),
and $\widetilde{z}^\textrm{sd}=z^\textrm{sd}+e_I(\widetilde{z}^\textrm{sd}_I-z^\textrm{sd}_I)$,
where $e_j=(0,\ldots,0,1,0,\ldots,0)^T$ with 1 being at the $j$th position
for $j=1,\ldots,c$. Let $Q_2'=((\widetilde{z}^\textrm{sd})^TB\widetilde{z}^\textrm{sd}$.
By a simple calculation it can be shown that
$Q_2'-Q_2=(\widetilde{z}^\textrm{sd})^TB\widetilde{z}^\textrm{sd}-(z^\textrm{sd})^TBz^\textrm{sd}=2(\widetilde{z}^\textrm{sd}_I-z^\textrm{sd}_I)e_I^TBz^\textrm{sd}$.
So it follows that
\begin{equation}\label{asymp:dist:eqn:1}
E\{Q_2'-Q_2|z^\textrm{sd}\}
=E\{2(\widetilde{z}^\textrm{sd}_I-z^\textrm{sd}_I)e_I^TBz^\textrm{sd}|z^\textrm{sd}\}
=\frac{2}{c}\sum_{j=1}^cE\{(\widetilde{z}^\textrm{sd}_j-z^\textrm{sd}_j)e_j^TBz^\textrm{sd}|z^\textrm{sd}\}
=-\frac{2}{c}Q_2.
\end{equation}

Let $g^\star_0:\bbR\rightarrow[0,1]$ be a $C^3$-function
such that $g^\star_0(s)=1$ for $s\le0$ and $g^\star_0(s)=0$ for $s\ge1$.
Let $G_u(s)=g^\star_0(\psi_c(s-u))$ for $u\in\bbR$,
where $\psi_c$ is a positive sequence tending to infinity and satisfying
\begin{equation}\label{asymp:dist:eqn:2}
(m_4^2+m_3^2+m_4)\psi_c^2h=o(1),\,\,\,\,
m_3\psi_c^3h^{1/2}=o(1).
\end{equation}
The existence of such $\psi_c$ follows by (\ref{trivial:fact}).

Next we will approximate $E\{G_u(Q_2)-G_u(V)\}$ where $V\sim N(0,1)$.
Consider Stein's equation
\begin{equation}\label{stein:eqn}
G_u(s)-E\{G_u(V)\}=\ddot{g}(s)-z^\textrm{sd}\dot{g}(s),
\end{equation}
where $\dot{g}$ and $\ddot{g}$ represent first- and second-order derivatives
of $g$. By \cite{GR96}, a solution to (\ref{stein:eqn}) is
\begin{equation}\label{stein:solution}
g(s)=-\int_0^1\frac{1}{2t}[E\{G_u(\sqrt{t}s+\sqrt{1-t}V)\}-E\{G_u(V)\}]dt.
\end{equation}

Let $C_1=\|\dot{g}^\star_0\|_{\sup}$, $C_2=\|\ddot{g}^\star_0\|_{\sup}$, and $C_3=\|\dddot{g}^\star_0\|_{\sup}$,
where $\dddot{g}^\star_0$ is the third-order derivative of $g^\star_0$.
It is easy to see that
\[
\ddot{g}(s)=-\frac{1}{2}\int_0^1E_U\{\ddot{G}_u(\sqrt{t}s+\sqrt{1-t}U)\}dt,
\]
\[
\dddot{g}(s)=-\frac{1}{2}\int_0^1\sqrt{t}E_U\{\dddot{G}_u(\sqrt{t}s+\sqrt{1-t}U)\}dt.
\]
Clearly, it holds that $\|\ddot{g}\|_{\sup}\le\|\ddot{G}_u\|_{\sup}\le C_2\psi_c^2$
and $\|\dddot{g}\|_{\sup}\le\|\dddot{G}_u\|_{\sup}\le C_3\psi_c^3$.

By exchangeability, $\frac{1}{2}E\{(Q_2'-Q_2)(\dot{g}(Q_2')+\dot{g}(Q_2))\}=0$.
So $E\{(Q_2'-Q_2)\dot{g}(Q_2)\}+\frac{1}{2}E\{(Q_2'-Q_2)(\dot{g}(Q_2')-\dot{g}(Q_2))\}=0$.
Since $E\{(Q_2'-Q_2)\dot{f}(Q_2)\}=
E\{E\{Q_2'-Q_2|w\}\dot{g}(Q_2)\}=-\frac{2}{c}E\{Q_2\dot{g}(Q_2)\}$,
we have
\begin{eqnarray} \label{J1J2}
&&E\{Q_2\dot{g}(Q_2)\}-E\{\ddot{g}(Q_2)\}\nonumber\\
&=&\frac{c}{4}E\{(Q_2'-Q_2)(\dot{g}(Q_2'))-\dot{g}(Q_2))\}-E\{\ddot{g}(Q_2)\}\nonumber\\
&=&\frac{c}{4}E\{\ddot{g}(Q_2)(Q_2'-Q_2)^2\}-E\{\ddot{g}(Q_2)\}+\frac{c}{4}\int_0^1(1-t)\times E\{\dddot{g}(Q_2+t(Q_2'-Q_2))(Q_2'-Q_2)^3\}dt\nonumber\\
&=&E\{\ddot{g}(Q_2)(\sum_{i=1}^c(\widetilde{z}^\textrm{sd}_i-z^\textrm{sd}_i)^2(e_i^TBz^\textrm{sd})^2-1)\}\nonumber \\
&\quad&+2\int_0^1(1-t)E\{\dddot{g}(Q_2+t(Q_2'-Q_2))\sum_{i=1}^c(\widetilde{z}^\textrm{sd}_i-z^\textrm{sd}_i)^3(e_i^TBz^\textrm{sd})^3\}\nonumber\\
&\equiv& J_1+J_2.
\end{eqnarray}
Next, we analyze $J_1$ and $J_2$ separately. Let $M_p=E\{(z^{\textrm{sd}}_i)^p\}$ for $p\ge1$. For $J_1$, by direct examinations we have
\begin{eqnarray*}
|J_1|\le C_2\psi_c^2E\{|\sum_{i=1}^c(1+(z^\textrm{sd}_i)^2)(e_i^TBz^\textrm{sd})^2-1|\}
\le C_2\psi_c^2E\{|\sum_{i=1}^cD_i-1|^2\}^{1/2},
\end{eqnarray*}
where $D_i=(1+(z^\textrm{sd}_i)^2)(e_i^TBz^\textrm{sd})^2$. Since $\sum_{i=1}^c E\{D_i\}=1$,
we get that
\begin{align}\label{lemma:J1:eqn:0}
E\{(\sum_{i=1}^cD_i-1)^2\}
&=E\{|\sum_{i=1}^c[D_i-E\{D_i\}]|^2\}\nonumber\\
&=\sum_{i=1}^cE\{(D_i-E\{D_i\})^2\}+\sum_{i\neq i'}E\{(D_i-E\{D_i\})(D_{i'}-E\{D_{i'}\})\}.
\end{align}
The first term of (\ref{lemma:J1:eqn:0}) is equal to
\begin{align*}
\sum_{i=1}^c[3(3+M_4)(e_i^TB^2e_i)^2-4(e_i^TB^2e_i)^2]&=(5+3M_4)\sum_{i=1}^c(e_i^TB^2e_i)^2\\
&=(5+3M_4)\sum_{i=1}^c\left(\sum_{l=1}^c b_{i,l}^2\right)^2\\
&\le(5+3M_4)\sum_{i=1}^c\left(\sum_{l=1}^c a_{i,l}^2\right)^2/s_c^4\\
&\le(5+3M_4)\textrm{trace}(A^4)/s_{c}^4\asymp (5+3M_4)h,
\end{align*}
where the last ``$\asymp$" follows by (\ref{asymp:dist:eqn:0}).

The second term of (\ref{lemma:J1:eqn:0}) is equal to
\begin{eqnarray*}
&&\sum_{i\neq i'}\left(E\{(1+(z^\textrm{sd}_i)^2)
(1+(z^\textrm{sd}_{i'})^2)(e_i^TBz^\textrm{sd})^2(e_{i'}^TBz^\textrm{sd})^2\}\right.\\
&&\left.-E\{(1+(z^\textrm{sd}_i)^2)(e_i^TBz^\textrm{sd})^2\}E\{(1+(z^\textrm{sd}_{i'})^2)(e_{i'}^TBz^\textrm{sd})^2\}\right)\\
&=&\sum_{i\neq i'}\left(E\{(1+(z^\textrm{sd}_i)^2)
(1+(z^\textrm{sd}_{i'})^2)E\{(e_i^TBz^\textrm{sd})^2(e_{i'}^TBz^\textrm{sd})^2|z^\textrm{sd}_i,z^\textrm{sd}_{i'}\}\}\right.\\
&&\left.-E\{(1+(z^\textrm{sd}_i)^2)(e_i^TBz^\textrm{sd})^2\}E\{(1+(z^\textrm{sd}_{i'})^2)(e_{i'}^TBz^\textrm{sd})^2\}\right).
\end{eqnarray*}

We have that
\begin{eqnarray*}
E\{(e_i^TBz^\textrm{sd})^2(e_{i'}^TBz^\textrm{sd})^2|z^\textrm{sd}_i,z^\textrm{sd}_{i'}\}
&=&E\{(b_{i,i'}z^\textrm{sd}_{i'}+\sum_{l\neq i,i'}b_{i,l}z^\textrm{sd}_l)^2(b_{i',i}z^\textrm{sd}_i+\sum_{l\neq i,i'}b_{i',l}z^\textrm{sd}_l)^2|z^\textrm{sd}_i,z^\textrm{sd}_{i'}\}\\
&=&E\{(N_1+N_2+N_3)(N_1'+N_2'+N_3')|z^\textrm{sd}_i,z^\textrm{sd}_{i'}\}\\
&=&E\{N_1N_1'|z^\textrm{sd}_i,z^\textrm{sd}_{i'}\}+E\{N_1N_2'|z^\textrm{sd}_i,z^\textrm{sd}_{i'}\}+E\{N_1N_3'|z^\textrm{sd}_i,z^\textrm{sd}_{i'}\}\\
&&+E\{N_2N_1'|z^\textrm{sd}_i,z^\textrm{sd}_{i'}\}+E\{N_2N_2'|z^\textrm{sd}_i,z^\textrm{sd}_{i'}\}
+E\{N_2N_3'|z^\textrm{sd}_i,z^\textrm{sd}_{i'}\}\\
&&+E\{N_3N_1'|z^\textrm{sd}_i,z^\textrm{sd}_{i'}\}+E\{N_3N_2'|z^\textrm{sd}_i,z^\textrm{sd}_{i'}\}+E\{N_3N_3'|z^\textrm{sd}_i,z^\textrm{sd}_{i'}\},
\end{eqnarray*}
where $N_1=(\sum_{l\neq i,i'}b_{i,l}z^\textrm{sd}_l)^2$,
$N_2=2\sum_{l\neq i,i'}b_{i,l}z^\textrm{sd}_lb_{i,i'}z^\textrm{sd}_{i'}$,
$N_3=b_{i,i'}^2(z^\textrm{sd}_{i'})^2$,
$N_1'=(\sum_{l\neq i,i'}b_{i',l}z^\textrm{sd}_l)^2$,
$N_2'=2\sum_{l\neq i,i'}b_{i',l}z^\textrm{sd}_l b_{i',i}z^\textrm{sd}_i$,
$N_3'=b_{i',i}^2(z^\textrm{sd}_i)^2$.
By direct calculations, it is easy to see that
\begin{eqnarray*}
E\{N_1N_1'|z^\textrm{sd}_i,z^\textrm{sd}_{i'}\}
&=&M_4\sum_{l\neq i,i'}b_{i,l}^2b_{i',l}^2+\sum_{\substack{l_1,l_2\neq i,i'\\ l_1\neq l_2}}
b_{i,l_1}^2b_{i',l_2}^2
+2\sum_{\substack{l_1,l_2\neq i,i'\\ l_1\neq l_2}}b_{i,l_1}b_{i,l_2}b_{i',l_1}b_{i',l_2},\\
E\{N_1N_2'|z^\textrm{sd}_i,z^\textrm{sd}_{i'}\}&=&2M_3b_{i',i}z^\textrm{sd}_i\sum_{l\neq i,i'}b_{i,l}^2b_{i',l},\\
E\{N_1N_3'|z^\textrm{sd}_i,z^\textrm{sd}_{i'}\}&=&b_{i',i}^2(z^\textrm{sd}_i)^2\sum_{l\neq i,i'}b_{i,l}^2,\\
E\{N_2N_1'|z^\textrm{sd}_i,z^\textrm{sd}_{i'}\}&=&2M_3b_{i,i'}z^\textrm{sd}_{i'}\sum_{l\neq i,i'}b_{i',l}^2b_{i,l},\\
E\{N_2N_2'|z^\textrm{sd}_i,z^\textrm{sd}_{i'}\}&=&4b_{i,i'}^2z^\textrm{sd}_{i'}z^\textrm{sd}_i\sum_{l\neq i,i'}b_{i,l}b_{i',l},\\
E\{N_2N_3'|z^\textrm{sd}_i,z^\textrm{sd}_{i'}\}&=&E\{N_3N_2'|z^\textrm{sd}_i,z^\textrm{sd}_{i'}\}=0,\\
E\{N_3N_1'|z^\textrm{sd}_i,z^\textrm{sd}_{i'}\}&=&b_{i,i'}^2(z^\textrm{sd}_{i'})^2\sum_{l\neq i,i'}b_{i',l}^2,\\
E\{N_3N_3'|z^\textrm{sd}_i,z^\textrm{sd}_{i'}\}&=&b_{i,i'}^2b_{i',i}^2(z^\textrm{sd}_i)^2(z^\textrm{sd}_{i'})^2.
\end{eqnarray*}
Therefore, it can be shown that
\begin{eqnarray*}
&&\sum_{i\neq i'}[E\{D_{i}D_{i'}\}-E\{D_i\}E\{D_{i'}\}]\\
&=&4M_4\sum_{i\neq i'}\sum_{l\neq i,i'}b_{i,l}^2b_{i',l}^2+
4\sum_{i\neq i'}\sum_{\substack{l_1,l_2\neq i,i'\\ l_1\neq l_2}}b_{i,l_1}^2b_{i',l_2}^2\\
&&+8\sum_{i\neq i'}\sum_{\substack{l_1,l_2\neq i,i'\\ l_1\neq l_2}}b_{i,l_1}b_{i,l_2}b_{i',l_1}b_{i',l_2}
+4M_3^2\sum_{i\neq i'}b_{i,i'}\sum_{l\neq i,i'}b_{i,l}^2b_{i',l}\\
&&+4M_3^2\sum_{i\neq i'}b_{i',i}\sum_{l\neq i,i'}b_{i,l}b_{i',l}^2
+2(1+M_4)\sum_{i\neq i'}b_{i,i'}^2\sum_{l\neq i,i'}b_{i,l}^2\\
&&+2(1+M_4)\sum_{i\neq i'}b_{i,i'}^2\sum_{l\neq i,i'}b_{i',l}^2
+4M_3^2\sum_{i\neq i'}b_{i',i}^2\sum_{l\neq i,i'}b_{i,l}b_{i',l}\\
&&+(1+M_4)^2\sum_{i\neq i'}b_{i,i'}^4-4\sum_{i\neq i'}\sum_{l\neq i}b_{i,l}^2\sum_{l\neq i'}b_{i',l}^2\\
&\le&4M_4\sum_{i\neq i'}\sum_{l\neq i,i'}b_{i,l}^2b_{i',l}^2
+8\sum_{i\neq i'}\sum_{\substack{l_1,l_2\neq i,i'\\ l_1\neq l_2}}b_{i,l_1}b_{i,l_2}b_{i',l_1}b_{i',l_2}
+4M_3^2\sum_{i\neq i'}b_{i,i'}\sum_{l\neq i,i'}b_{i,l}^2b_{i',l}\\
&&+4M_3^2\sum_{i\neq i'}b_{i',i}\sum_{l\neq i,i'}b_{i,l}b_{i',l}^2
+2(1+M_4)\sum_{i\neq i'}b_{i,i'}^2\sum_{l\neq i,i'}b_{i,l}^2\\
&&+2(1+M_4)\sum_{i\neq i'}b_{i,i'}^2\sum_{l\neq i,i'}b_{i',l}^2
+4M_3^2\sum_{i\neq i'}b_{i',i}^2\sum_{l\neq i,i'}b_{i,l}b_{i',l}
+(1+M_4)^2\sum_{i\neq i'}b_{i,i'}^4\\
&\le&(M_4^2+12M_3^2+10M_4+13)\textrm{trace}(B^4).
\end{eqnarray*}
The last inequality holds because
each term in the summation is bounded by $\textrm{trace}(B^4)$
multiplied by suitable constants.

Since $B=(A-A_0)/s_c$, we have
$B^2\le 2(A^2+A_0^2)/s_c^2$ and $B^4\le8(A^4+A_0^4)/s_c^4$. So it holds that
\[
\textrm{trace}(B^4)\le 16s_c^{-4}\textrm{trace}(A^4),
\]
where the last inequality follows from the trivial fact
$\textrm{trace}(A^4)\ge\sum_{i=1}^ca_{i,i}^4$.
From the above analysis, we get that
\begin{eqnarray}\label{J1}
|J_1|\le C_2(16M_4^2+192M_3^2+163M_4+213)\psi_c^2s_c^{-4}\textrm{trace}(A^4).
\end{eqnarray}

For $J_2$, it holds that
\begin{eqnarray*}
|J_2|&=&\big|2\sum_{i=1}^c\int_0^1(1-t)E\{\dddot{g}(Q_2+t(Q_2'-Q_2))(\widetilde{z}^\textrm{sd}_i-z^\textrm{sd}_i)^3(e_i^TBz^\textrm{sd})^3\}\big|\\
&\le&2\|\dddot{g}\|_{\sup}\sum_{i=1}^cE\{|\widetilde{z}^\textrm{sd}_i-z^\textrm{sd}_i|^3|e_i^TBz^\textrm{sd}|^3\}\\
&=&2\|\dddot{g}\|_{\sup}\sum_{i=1}^cE\{|\widetilde{z}^\textrm{sd}_i-z^\textrm{sd}_i|\}E\{|e_i^TBz^\textrm{sd}|^3\}\\
&\le&32C_3E\{|z^\textrm{sd}_i|^3\}\psi_c^3\sum_{i=1}^cE\{|e_i^TBz^\textrm{sd}|^3\}\\
&\le&32C_3E\{|z^\textrm{sd}_i|^3\}\psi_c^3
\sum_{i=1}^cE\{|e_i^TBz^\textrm{sd}|^4\}^{1/2}E\{|e_i^TBz^\textrm{sd}|^2\}^{1/2}\\
&=&32\sqrt{3}C_3E\{|z^\textrm{sd}_i|^3\}\psi_c^3\sum_{i=1}^c(e_i^TB^2e_i)^{3/2}\\
&\le&32\sqrt{3}C_3E\{|z^\textrm{sd}_i|^3\}\psi_c^3\sqrt{\sum_{i=1}^c(e_i^TB^2e_i)^2\sum_{i=1}^c e_i^TB^2e_i}\\
&\le&32\sqrt{3/2}C_3E\{|z^\textrm{sd}_i|^3\}\psi_c^3\sqrt{\textrm{trace}(B^4)}\\
&\le&128\sqrt{3/2}C_3E\{|z^\textrm{sd}_i|^3\}\psi_c^3s_c^{-2}\sqrt{\textrm{trace}(A^4)}.
\end{eqnarray*}

By (\ref{asymp:dist:eqn:0}), (\ref{J1}) and $s_c^2\asymp h^{-1}$,
we have
\begin{eqnarray*}
|J_1|&\lesssim&C_2(16M_4^2+192M_3^2+163M_4+213)\psi_c^2h,\\
|J_2|&\lesssim& 128\sqrt{3/2}C_3E\{|z^\textrm{sd}_i|^3\}\psi_c^3h^{1/2}.
\end{eqnarray*}
By (\ref{asymp:dist:eqn:2}) the following holds uniformly for $u\in\bbR$:
\begin{equation}\label{asymp:dist:eqn:3}
E\{G_u(Q_2)\}-E\{G_u(V)\}\rightarrow0,\,\,\,\,c\rightarrow\infty.
\end{equation}
Similarly, for $\widetilde{G}_u(s)=g^\star_0(\psi_n(s-u)+1)$,
it can be shown that the following statement holds uniformly for $u\in\bbR$:
\begin{equation}\label{asymp:dist:eqn:4}
E\{\widetilde{G}_u(Q_2)\}-E\{\widetilde{G}_u(V)\}\rightarrow0,\,\,\,\,c\rightarrow\infty.
\end{equation}
By elementary facts, we have
\begin{eqnarray}\label{asymp:dist:eqn:5}
P(Q_2\le u)&\le& E\{G_u(Q_2)\}\le P(Q_2\le u+\psi_c^{-1}),\nonumber\\
P(V\le u)&\le& E\{G_u(V)\}\le P(V\le u+\psi_c^{-1}),\nonumber\\
P(Q_2\le u-\psi_c^{-1})&\le&E\{\widetilde{G}_u(Q_2)\}\le P(Q_2\le u),\nonumber\\
P(V\le u-\psi_c^{-1})&\le&E\{\widetilde{G}_u(V)\}\le P(V\le u).
\end{eqnarray}
By (\ref{asymp:dist:eqn:3}), (\ref{asymp:dist:eqn:4}) and (\ref{asymp:dist:eqn:5}),
the following statements hold uniformly for $u\in\bbR$,
\begin{eqnarray*}
P(Q_2\le u)-P(V\le u)&\le& E\{G_u(Q_2)-G_u(V)\}+P(V\le u+\psi_c^{-1})-P(V\le u)\rightarrow0,\\
P(V\le u)-P(Q_2\le u)&\le& E\{\widetilde{G}_u(V)-\widetilde{G}_u(Q_2)\}+P(V\le u)-P(V\le u-\psi_c^{-1})\rightarrow0.
\end{eqnarray*}
Hence, as $c$ tends to infinity,
\[
\sup_{u\in\bbR}|P(Q_2\le u)-P(V\le u)|\rightarrow0.
\]
This, together with $Q_1=o_P(1)$, proves
\begin{equation} \label{eq limit distribution}
\frac{cT^\star_{\mu^\star,t,c}-\textrm{trace}(A)\tau^{\star2}_k}{s_c\tau^{\star2}_k}\overset{d}{\longrightarrow}N(0,1),
\,\,\textrm{as $c\rightarrow\infty$.}
\end{equation}

Let $z_i$'s and $T_{\mu, t, c}$ be the quantized samples and testing statistics in Theorem \ref{asymp:distribution}, then we have
\begin{eqnarray*}
\frac{cT_{\mu,t,c}-\sum_{i=1}^c a_{i,i}\widehat{\tau}_k^2}{s_c\widehat{\tau}_k^2}
&=& \frac{cT_{\mu,t,c}-cT^\star_{\mu^\star,t,c}}{s_c\widehat{\tau}_k^2} + \frac{cT^\star_{\mu^\star,t,c}-\sum_{i=1}^c a_{i,i}\widehat{\tau}_k^2}{s_c\widehat{\tau}_k^2}\\
&=& R_1 + R_2.
\end{eqnarray*}
We will analyze these two terms separately. For $R_1$, one has that


\begin{eqnarray}
\label{eq:a}
\frac{cT_{\mu,t,c}-cT^\star_{\mu^\star,t,c}}{s_c\widehat{\tau}_k^2} &=& \frac{z^TAz-z^{^\star T}Az^\star }{s_c\widehat{\tau}_k^2}= \frac{\Delta^{\star T} A \Delta^{\star}}{s_c\widehat{\tau}_k^2} + \frac{2\Delta^{\star T} A z^\star}{s_c\widehat{\tau}_k^2},
\end{eqnarray}
where $\Delta^{\star} = (\Delta^\star_1, \ldots, \Delta^\star_c)^T$ with $\Delta^\star_i=z_i-z^\star_i$ which satisfies $E(\Delta^{\star2}_i)\le C_k^2(t) + 1/n$ under Condition~{\bf B}.

For the first term, since $\|\Delta^\star\|^2\leq O_p\left(cC_k(t)^2 + cn^{-1}\right)$ and $A\leq I_c$, it follows that
\[
\frac{\Delta^{\star T} A \Delta^{\star}}{s_c\widehat{\tau}_k^2}\le \frac{ \|\Delta^\star\|^2}{s_c\widehat{\tau}_k^2}\le \frac{c}{s_c\widehat{\tau}_k^2} O_p\left(C_k(t)^2 + n^{-1}\right)= O_p\left(\frac{ cC_k(t)^2}{h^{-1/2}c/n}\right)= O_p\left(n h^{1/2}C_k(t)^2\right)=o_p(1),
\]
where the last equality follows from the condition $b\gg\log_2\left(\sqrt{n\T_nh^{1/2}}\right)$, which implies that
\begin{align*}
C_k(t)^2 \asymp \frac{\T_n}{2^{2b}} \ll \frac{\T_n}{(nh^{1/2})\T_n} = (nh^{1/2})^{-1}.     
\end{align*}

For the second term in~\eqref{eq:a}, using the fact that $(\Delta^\star_i,z_i^\star)^T$,  $(\Delta^\star_j,z_j^\star)^T$ are independent if $i\ne j$, and $Ez^\star=0$,  it is straightforward to show that
\[
\begin{split}
    E\left((\Delta^{\star}-E\Delta^\star)^T A z^{\star}\right)^2=&\sum_{i=1}^c a_{ii}^2 E((\Delta_i^{\star}-E\Delta_i^\star) z_i^\star)^2+ \sum_{i=1}^c\sum_{j=1}^c a_{ij}^2 E((\Delta_i^{\star}-E\Delta_i^\star)(\Delta_i^{\star}-E\Delta_i^\star) z_j^{\star} z_j^*)\\
    &+\sum_{i=1}^c\sum_{j=1}^c a_{ij}a_{ji} E((\Delta_i^{\star}-E\Delta_i^\star) z_j^\star(\Delta_j^{\star}-E\Delta_j^\star)z_i^\star)\\&+\sum_{i=1}^c\sum_{j=1}^c a_{ii}a_{jj} E((\Delta_i^{\star}-E\Delta_i^\star) z_i^\star(\Delta_j^{\star}-E\Delta_j^\star) z_j^\star)\\
    &+\sum_{i=1}^c\sum_{j=1}^c a_{jj}a_{ii} E((\Delta_j^{\star}-E\Delta_j^\star) z_j^\star(\Delta_i^{\star}-E\Delta_i^\star) z_i^\star)\\
    \le&\sum_{i=1}^c a_{ii}^2C_k(t)^2\tau_k^{\star 2}+ \sum_{i=1}^c\sum_{j=1}^c a_{ij}^2 C_k(t)^2\tau_k^{\star 2}\\&+\sum_{i=1}^c\sum_{j=1}^c a_{ij}a_{ji}C_k(t)^2\tau_k^{\star 2}+\sum_{i=1}^c\sum_{j=1}^c a_{ii}a_{jj} C_k(t)^2\tau_k^{\star 2}+\sum_{i=1}^c\sum_{j=1}^c a_{jj}a_{ii} C_k(t)^2\tau_k^{\star 2}\\
    =&\left(\sum_{i=1}^c a_{ii}^2\right)C_k(t)^2\tau_k^{\star 2}+2\textrm{trace}(A^2) C_k(t)^2\tau_k^{\star 2}+2[\textrm{trace}(A)]^2 C_k(t)^2\tau_k^{\star 2}.
\end{split}
\]
In the proof of \eqref{asymp:dist:eqn:0}, we have shown that $a_{ii}\asymp(ch)^{-1}$, $\textrm{trace}(A)\asymp \textrm{trace}(A^2)\asymp h^{-1}$, thence we have that
\[
E\left((\Delta^{\star}-E\Delta^\star)^T A z^{\star}\right)^2\lesssim h^{-2}C_k(t)^2\tau_k^{\star 2},
\]
which implies that
\[
\frac{2(\Delta^{\star}-E\Delta^\star)^T A z^\star}{s_c\widehat{\tau}_k^2}=O_p\left(\frac{h^{-1}C_k(t)\tau_k^{\star}}{s_c\widehat{\tau}_k^2}\right)=O_p\left(\frac{h^{-1}C_k(t)\sqrt{c/n}}{h^{-1/2} c/n}\right) =O_p\left(C_k(t)\sqrt{\frac{n}{ch}}\right).
\]
Furthermore, since $A\le I_c$, we have that
\[
E\left[(E\Delta^{\star})^T A z^{\star}\right]^2=E\left[(E\Delta^{\star})^T A z^{\star}z^{\star T}A(E\Delta^{\star})\right]=\tau_k^{\star 2}(E\Delta^{\star})^TA^2(E\Delta^{\star})\le \tau_k^{\star 2}\|E\Delta^{\star}\|^2\le c C_k(t)^2 \tau_k^{\star 2},
\]
which implies that
\[
\frac{2(E\Delta^\star)^T A z^\star}{s_c\widehat{\tau}_k^2}=O_p\left(\frac{\sqrt{c}C_k(t)\tau_k^{\star}}{s_c\widehat{\tau}_k^2}\right)=O_p\left(\frac{\sqrt{c}C_k(t)\sqrt{c/n}}{h^{-1/2} c/n}\right) =O_p\left(C_k(t)\sqrt{nh} \right),
\]
and consequently,
\begin{equation}\label{eq:boundtmp}
\frac{2 \Delta^{\star T}A z^\star}{s_c\widehat{\tau}_k^2} =O_p\left(C_k(t)\sqrt{\frac{n}{ch}}+C_k(t)\sqrt{nh}\right).    
\end{equation}

Using the condition $b\gg\log_2\left(\sqrt{(nh^{1/2}+n(ch)^{-1})\T_n}\right)$, and the fact that $h\to 0$, one has that
\begin{align*}
C_k(t)^2 \asymp \frac{\T_n}{2^{{2b}}} \ll \frac{\T_n}{(nh^{1/2}+n(ch)^{-1})\T_n}\le \frac{\T_n}{(nh+n(ch)^{-1})\T_n} =(nh+n(ch)^{-1})^{-1},     
\end{align*}
which gives that $\frac{2\Delta^{\star T} A z^\star}{s_c\widehat{\tau}_k^2}=o_p(1)$. Plugging this back to equation~\eqref{eq:a}, we have that $R_1=o_p(1)$.


Now we analyze $R_2$, by Lemma \ref{hat_tau_k}, we have 
\begin{align*}
\frac{cT^\star_{\mu^\star,t,c}-\sum_{i=1}^c a_{i,i}\widehat{\tau}_k^2}{s_c\widehat{\tau}_k^2} &= \frac{cT^\star_{\mu^\star,t,c}-\textrm{trace}(A)\tau^{\star2}_k+\textrm{trace}(A)\tau^{\star2}_k-\textrm{trace}(A)\widehat{\tau}_k^2}{s_c\tau^{\star2}_k}\frac{\tau^{\star2}_k}{\widehat{\tau}_k^2}\\
& = \frac{cT^\star_{\mu^\star,t,c}-\textrm{trace}(A)\tau^{\star2}_k}{s_c\tau^{\star2}_k} + \frac{\textrm{trace}(A)(\tau_k^{\star2}-\widehat{\tau}_k^2)}{s_c\tau^{\star2}_k} + o_p(1)\\
&=\frac{cT^\star_{\mu^\star,t,c}-\textrm{trace}(A)\tau^{\star2}_k}{s_c\tau^{\star2}_k} + O_p(C_k(t)h^{-1/2}).
\end{align*}
Since $k\gg \sqrt{(nh^{1/2}+n(ch)^{-1})\T_n}$, one has that
\[
C_kh^{-1/2}\ll \frac{\sqrt{\T_n}}{\sqrt{(nh^{1/2}+\frac{n}{ch})\T_n h}}=\frac{1}{\sqrt{(nh^{3/2}+\frac{n}{c}) }}\le \sqrt{c/n}=O(1).
\]
From (\ref{eq limit distribution}), we get the desired result.


\end{proof}

\begin{proof}\textbf{of Proposition \ref{tau_k^2:lemma}:}
Suppose $p_\sigma(\cdot)$ is the density of $\sigma\epsilon_1$. By direct calculations, we have
\begin{equation}\label{eq:properity4}
E([\widetilde{n}^{-1}\sum_{j=1}^{\widetilde{n}}Q(\epsilon_j)]^4) = 3\sigma^4(\frac{1}{\widetilde{n}^2}-\frac{1}{\widetilde{n}^3})E^2\{Q(\sigma\epsilon_1)^2\} + \frac{\sigma^2E\{Q(\sigma\epsilon_1)^4\}}{\widetilde{n}^3}.    
\end{equation}
For the first term, under Condition (B), we know $E\{Q(\sigma\epsilon_1)^2\} = O(1)$, which implies $3\sigma^4(\frac{1}{\widetilde{n}^2}-\frac{1}{\widetilde{n}^3})E^2\{Q(\sigma\epsilon_1)^2\} = O(c^2n^{-2})$. For the second term, we have that
\begin{eqnarray*}
E\{Q(\sigma\epsilon_1)^4\} &=& \sum_{j=2}^{k-1} \mu_j^4P(\sigma\epsilon_1\in R_j) + \mu_1^4P(\sigma\epsilon_1\in R_1) + \mu_k^4P(\sigma\epsilon_1\in R_k)\\
&=&\sum_{j=2}^{k-1}\int_{R_j}\mu_j^4p_\sigma(x)dx + \mu_1^4P(\sigma\epsilon_1\in R_1) + \mu_k^4P(\sigma\epsilon_1\in R_k)\\
&\le&\sum_{j=2}^{k-1}\int_{R_j}(|x| + C_k(t))^4p_\sigma(x)dx + \mu_1^4P(\sigma\epsilon_1\in R_1) + \mu_k^4P(\sigma\epsilon_1\in R_k)\\
&\le&8\sum_{j=1}^{k}\int_{R_j}\sigma^4x^4p_\sigma(x)dx + 8C_k(t)^4 + \mu_1^4P(\sigma\epsilon_1\in R_1) + \mu_k^4P(\sigma\epsilon_1\in R_k)\\
&=&8\sigma^4E\{\epsilon_1^4\} + 8C_k(t)^4 + \mu_1^4P(\sigma\epsilon_1\in R_1) + \mu_k^4P(\sigma\epsilon_1\in R_k).
\end{eqnarray*}
Since $C_k(t)^4=o(1)$, $E\{\epsilon_1^4\}= O(nc^{-1})$ and $\mu_j^4P(\sigma\epsilon_1\in R_j(t)) = O(nc^{-1})$ for $j=1,k$, one has that $\frac{\sigma^2E\{Q(\sigma\epsilon_1)^4\}}{\widetilde{n}^3} = O(c^2n^{-2})$. Plugging this back to equation (\ref{eq:properity4}), we get the desired result.
\end{proof}

\begin{proof}\textbf{of Theorem \ref{power:thm}:}  Without loss of generality, we only consider the case $g_*(x)=0$ in~\eqref{H0}.
By Condition (\textbf{B}),
we have that $\min\{t_1^2,t_{k-1}^2\}>4c_s^2\rho^2$, as $n \rightarrow\infty$.
Consider the following event:
\begin{equation}\label{event:E}
\mathcal{E}_1=\{\textrm{$\sigma|\epsilon_i|+c_s\rho\le\sqrt{\T_n}$ for all $1\le i\le n$}\}.
\end{equation}
It is easy to show that $P(\mathcal{E}_1)\rightarrow1$ as $n\rightarrow\infty$ under Condition (\textbf{B}).
Thus, we choose $N_\eta'$ s.t. $P(\mathcal{E}_1)\ge1-\eta/3$ if $c\ge N_\eta'$.
Throughout the proof, we suppose that $g\in S_\rho^m(\mathbb{I})$ is the function that generates the samples and $f$ is the integral function of $g$ defined in (\ref{eq:f0}). Let $\omega_i=\widetilde{z}_i-\widetilde{z}_i^0, \omega=(\omega_1, \ldots, \omega_c)^T$. It is straightforward to see that $\omega_i={z}_i-{z}_i^0$ under event $\mathcal{E}_1$. Because 
\begin{align*}
C_k(t)^2 \asymp \frac{\T_n}{2^{{2b}}} \ll \frac{\T_n}{(nh^{1/2}+n(ch)^{-1})\T_n}\le \frac{\T_n}{(nh+n(ch)^{-1})\T_n} =(nh+n(ch)^{-1})^{-1} \ll \tau_k^2,     
\end{align*}
it follows by Lemma \ref{power:thm:lemma:2} that there exists $N''$ s.t., when $c\ge N''$, the following equation holds
\begin{equation}
\sup_{g\in S_\rho^m(\mathbb{I})}\frac{E\{|(\widetilde{z}-\widetilde{z}^0)^TAz^0|^2\}}
{(1+(ch^2)^{-1})\tau_k^2\sum_{i=1}^c(|f(i/c)|+4C_k(t) + \zeta)^2}\le8,\,\,\textrm{as $c\rightarrow\infty$}.
\end{equation}
Consider the event
\[
\mathcal{E}_{2}=\left\{|\omega^T A z^0|\le C_\eta'\sqrt{1+(ch^2)^{-1}}\tau_k\sqrt{\sum_{i=1}^c(|f(i/c)|+4C_k(t)+\zeta)^2}\right\},
\]
where $C_\eta'=\sqrt{24/\eta}$. 

Then
\[
1-P(\mathcal{E}_{2})\le \frac{E\{|\omega^T A z^0|^2\}}{(C_\eta')^2(1+(ch^2)^{-1})\tau_k^2\sum_{i=1}^c(|f(i/c)|+4C_k(t)+\zeta)^2}
\le \eta/3,
\]
which implies that $P\left(\mathcal{E}_{2}\right)\ge 1-\eta/3$.

Let $\widehat{\tau}_{k,0}^2$ be the estimated variance under the null. Then one has that
\begin{align*}
\frac{(z^0)^TAz^0-\textrm{trace}(A)\widehat{\tau}_{k}^2}{s_c\widehat{\tau}_{k}^2} &= \frac{(z^0)^TAz^0-\textrm{trace}(A)\widehat{\tau}_{k,0}^2+ \textrm{trace}(A)\widehat{\tau}_{k,0}^2 - \textrm{trace}(A)\widehat{\tau}_{k}^2}{s_c\widehat{\tau}_{k,0}^2}\frac{\widehat{\tau}_{k,0}^2}{\widehat{\tau}_{k,0}^2}\\
& = \frac{(z^0)^TAz^0-\textrm{trace}(A)\widehat{\tau}_{k,0}^2}{s_c\widehat{\tau}^{2}_{k,0}} + \frac{\textrm{trace}(A)(\widehat{\tau}_{k,0}^{2}-\widehat{\tau}_k^2)}{s_c\widehat{\tau}^{2}_{k,0}} + o_p(1)\\
&=\frac{(z^0)^TAz^0-\textrm{trace}(A)\widehat{\tau}_{k,0}^2}{s_c\widehat{\tau}^{2}_{k,0}} + O_p(C_k(t)h^{-1/2}).
\end{align*}
Since $k\gg \sqrt{(nh^{1/2}+n(ch)^{-1})\T_n}$, one has that
\[
C_kh^{-1/2}\ll \frac{\sqrt{\T_n}}{\sqrt{(nh^{1/2}+\frac{n}{ch})\T_n h}}=\frac{1}{\sqrt{(nh^{3/2}+\frac{n}{c}) }}\le \sqrt{c/n}=O(1).
\]
It follows from Theorem \ref{asymp:distribution} that 
\[
\frac{(z^0)^TAz^0-\textrm{trace}(A)\widehat{\tau}_{k}^2}{s_c\widehat{\tau}_{k}^2}=O_P(1).
\]
Hence, there exists $C_\eta''>0$ s.t. $P(\mathcal{E}_3)\ge 1-\eta/3$ for all $c\ge N_\eta'$ and $N''$, where
\[
\mathcal{E}_3=\left\{\bigg|\frac{(z^0)^TAz^0-\textrm{trace}(A)\widehat{\tau}_k^2}{s_c\widehat{\tau}_k^2}\bigg|\le C_\eta''\right\}.
\]
Let $\mathcal{E}=\mathcal{E}_1\cap\mathcal{E}_{2}\cap\mathcal{E}_3$,
then $P(\mathcal{E})\ge 1-\eta$ for any $c\ge N_\eta'$ and $N''$.

Suppose $g\in S_\rho^m(\mathbb{I})$ satisfies $\|g\|_c\ge C_\eta\delta_*$,
where
\begin{align}\label{C:eta}
C_\eta=\max\left\{6\varrho\rho^2,384,(72C_\eta')^2,6(C_\eta''+z_{1-\alpha/2}+1)\right\},
\end{align}
\begin{align} \label{delta}
\delta_* = \sqrt{c^{-1}\tau_k^2(1+s_c+(ch^2)^{-1})+\lambda+c^{-2m}+C_k(t)^2+\zeta^2},
\end{align}
where $\tau_k^2=O(\widetilde{n}^{-1})$, $\zeta= \max\limits_{i = 1,\ldots,c}\big|f(i/c) - \frac{1}{\widetilde{n}}\sum_{j=(i-1)\widetilde{n}+1}^{i\widetilde{n}}g(j/n)\big|=O(n^{-1})$.

It follows from Lemma \ref{power:thm:lemma:1} that, on $\mathcal{E}$, $|\omega_i-f(i/c)|\le 4C_k(t)+\zeta$.
Since $A\le I_c$, we get that
\[
(\omega-\textbf{f})^TA(\omega-\textbf{f})\le\sum_{i=1}^c(\omega_i-f(i/c))^2\le 32cC_k(t)^2+2c\zeta^2,
\]
which, together with Lemma \ref{power:thm:lemma:4}, leads to that
\begin{eqnarray*}
&&\omega^T A\omega\ge\frac{1}{2}\textbf{f}^TA\textbf{f}-(\omega-\textbf{f})^TA(\omega-\textbf{f})\\
&=&\frac{c}{2}\|g\|_c^2-\frac{1}{2}\textbf{f}^T(I_c-A)\textbf{f}-(\omega-\textbf{f})^TA(\omega-\textbf{f})\\
&\ge&\frac{c}{2}\|g\|_c^2-\frac{1}{2}\varrho c(\lambda+c^{-2m})\rho^2-32cC_k(t)^2-2c\zeta^2.
\end{eqnarray*}
Therefore, on $\mathcal{E}$, we have
\begin{eqnarray}
&&\frac{cT_{\mu,t, c}-\textrm{trace}(A)\widehat{\tau}_k^2}{s_c\widehat{\tau}_k^2}\nonumber\\
&=&\frac{z^TAz-(z^0)^TAz^0}{s_c\widehat{\tau}_k^2}+\frac{(z^0)^TAz^0-\textrm{trace}(A)\widehat{\tau}_k^2}{s_c\widehat{\tau}_k^2}\nonumber\\
&\ge&\frac{z^TAz-(z^0)^TAz^0}{s_c\widehat{\tau}_k^2}-C_\eta''\nonumber\\
&=&\frac{\omega^TA\omega+2\omega^TAz^0}{s_c\widehat{\tau}_k^2}-C_\eta''\nonumber\\
&\ge&\textstyle\frac{\frac{c}{2}\|g\|_c^2-\frac{1}{2}\varrho c(\lambda+c^{-2m})\rho^2-32cC_k(t)^2-2c\zeta^2-
2C_\eta'\tau_k\sqrt{(1+\frac{1}{ch^2})\sum_{i=1}^c(|f(i/c)|+4C_k(t)+\zeta)^2}}{s_c\widehat{\tau}_k^2}-C_\eta''\nonumber\\
&\ge&\textstyle\frac{\frac{c}{2}\|g\|_c^2-\frac{1}{2}\varrho c(\lambda+c^{-2m})\rho^2-32cC_k(t)^2-2c\zeta^2-6C_\eta'\sqrt{1+(ch^2)^{-1}}\tau_k\sqrt{c}\|g\|_c}
{s_c\widehat{\tau}_k^2}-C_\eta''\label{pow:thm:point:1}\\
&=&\frac{\frac{c}{2}\|g\|_c^2\left(
1-\frac{\frac{1}{2}\varrho c(\lambda+c^{-2m})\rho^2}{\frac{c}{2}\|g\|_c^2}
-\frac{32cC_k(t)^2}{\frac{c}{2}\|g\|_c^2}-\frac{2c\zeta^2}{\frac{c}{2}\|g\|_c^2}-
\frac{6C_\eta'\sqrt{1+(ch^2)^{-1}}\tau_k\sqrt{c}\|g\|_c}{\frac{c}{2}\|g\|_c^2}\right)}{s_c\widehat{\tau}_k^2}-C_\eta''\nonumber\\
&\ge&\frac{\frac{c}{6}\|g\|_c^2}{s_c\widehat{\tau}_k^2}-C_\eta''>z_{1-\alpha/2},\label{pow:thm:point:2}
\end{eqnarray}
where (\ref{pow:thm:point:1}) follows from $C_\eta>12$ (see (\ref{C:eta})), i.e.,
\[
\sum_{i=1}^c f(i/c)^2=c\|g\|_c^2\ge cC_\eta\delta_{*}^2\ge 48cC_k(t)^2, ,\,\,\, \sum_{i=1}^c f(i/c)^2=c\|g\|_c^2\ge cC_\eta\delta_{*}^2\ge 3c\zeta^2,
\]
which leads to
\[
\sum_{i=1}^c(|f(i/c)|+4C_k(t)+\zeta)^2\le 3\sum_{i=1}^cf(i/c)^2+48cC_k(t)^2 +3c\zeta^2\le 9\sum_{i=1}^cf(i/c)^2=9c\|g\|_c^2;
\]
and (\ref{pow:thm:point:2}) follows from (\ref{C:eta}), i.e.,
\begin{eqnarray*}
&&\frac{\frac{1}{2}\varrho c(\lambda+c^{-2m})\rho^2}{\frac{c}{2}\|g\|_c^2}\le 1/6,\\
&&\frac{32cC_k(t)^2}{\frac{c}{2}\|g\|_c^2}\le 1/6,\\
&&\frac{2c\zeta^2}{\frac{c}{2}\|g\|_c^2}\le 1/6,\\
&&\frac{6C_\eta'\sqrt{1+(ch^2)^{-1}}\tau_k\sqrt{c}\|g\|_c}{\frac{c}{2}\|g\|_c^2}\le 1/6.
\end{eqnarray*}
Then for any $g\in S_\rho^m(\mathbb{I})$ satisfying $\|g\|_c\ge C_\eta\delta_*$, where $C_\eta, \delta_*$ are defined in (\ref{C:eta}) (\ref{delta}), there exist $N_\eta\equiv\max\{N_\eta',N''\}$ such that for any $c\ge N_\eta$, we have
\begin{eqnarray*}
&&P(\textrm{reject $H_0$}|\textrm{$H_1$ is true})\\
&\ge&P\left(\textrm{$\mathcal{E}$ and $\bigg|\frac{cT_{\mu,t,c}-\textrm{trace}(A)\widehat{\tau}_k^2}{s_c\widehat{\tau}_k^2}\bigg|
\ge z_{1-\alpha/2}$}\right)\\
&=&P(\mathcal{E})\ge 1-\eta.
\end{eqnarray*}
In the end, since $h = \lambda^{1/(2m)}, nh^{1/2}C_k(t)^2=o(1), ch\rightarrow \infty, \zeta = O(n^{-1})$, immediately, one has that $\|g\|_c\ge C_\eta\delta_*$ is equivalent as $\|g\|_c\ge C_\eta\delta_{n, c, \lambda}$. This proves the desired result.
\end{proof}

\begin{proof}\textbf{of Theorem \ref{corollary linear}:}
Suppose $g = \beta x + \alpha$ is the ``true'' function under $H^\textrm{linear}_0$ and $y_i = g(i/n) + \sigma\epsilon_i$, $i=1, \ldots, n$. We use $\widehat{g}$ to denote the least-square estimator of $g$ based on $Q(y_i)$'s. Consider the following two events:
\begin{eqnarray*}
\mathcal{E}_1 &=&\{\textrm{$\sigma|\epsilon_i|+c_s\rho\le\sqrt{\T_n}$ for all $1\le i\le n$}\},\\
\mathcal{E}_2 &=&\{\textrm{$|g(i/n) - \widehat{g}(i/n)|\le\sqrt{\T_n}$ for all $1\le i\le n$}\}.
\end{eqnarray*}
It is easy to show that $P(\mathcal{E}_1 \cap \mathcal{E}_2) \rightarrow 1$, as $n\rightarrow \infty$. Since $\min\{t_1^2,t_{k-1}^2\}=\T_n>4c_s^2\rho^2$ as $n \rightarrow\infty$, under event $\mathcal{E}_1 \cap \mathcal{E}_2$, for $j=1,\ldots,c$, one has that 
\begin{equation*}
    \sigma|\epsilon_j|\le\min\{|t_1|,|t_{k-1}|\} ,\,\,\, |y_j| \le \min\{|t_1|,|t_{k-1}|\}, \,\,\,\, |g(j/n) - \widehat{y}_j| \le \min\{|t_1|,|t_{k-1}|\}.
\end{equation*}
Furthermore, we have
\begin{eqnarray}
z^\textrm{linear}_i&\le&\frac{\sum_{j=(i-1)\widetilde{n}+1}^{i\widetilde{n}}Q(y^\textrm{linear}_j)}{\widetilde{n}} + C_k(t) \label{eq:linear1}\\
&=&\frac{\sum_{j=(i-1)\widetilde{n}+1}^{i\widetilde{n}}Q\big(Q(y_j) - \widehat{y}_j\big)}{\widetilde{n}} + C_k(t) \nonumber \\
&=& \frac{\sum_{j=(i-1)\widetilde{n}+1}^{i\widetilde{n}}Q\big(y_j - g(j/n) + Q(y_j) - y_j + g(j/n) - \widehat{y}_j\big)}{\widetilde{n}}+ C_k(t) \nonumber \\
&\leq& \frac{\sum_{j=(i-1)\widetilde{n}+1}^{i\widetilde{n}}Q\big(y_j - g(j/n)\big) + Q\big(Q(y_j) - y_j\big) + Q\big(g(j/n) - \widehat{y}_j\big)}{\widetilde{n}}  + 7C_k(t)\nonumber\\
&\leq& \frac{\sum_{j=(i-1)\widetilde{n}+1}^{i\widetilde{n}}Q(\sigma\epsilon_j)}{\widetilde{n}} + \frac{\sum_{j=(i-1)\widetilde{n}+1}^{i\widetilde{n}}Q\big(g(j/n) - \widehat{y}_j\big)}{\widetilde{n}} + 9C_k(t)\nonumber\\
&\le& z_i^0 + \varsigma_i +10C_k(t), \label{eq:linear2}
\end{eqnarray}
where $\varsigma_i = \frac{\sum_{j=(i-1)\widetilde{n}+1}^{i\widetilde{n}}Q\big(g_0(j/n) - \widehat{y}_j\big)}{\widetilde{n}}$ satisfying $\varsigma_i=O_p(1/n + C_k(t))$, and equations (\ref{eq:linear1}), (\ref{eq:linear2}) follow from (\ref{eq:quant:2}), (\ref{eq:quant_z0}). Let $z^0 = (z_1^0, \ldots, z_c^0)^T$, $\varsigma = (\varsigma_1, \ldots, \varsigma_c)^T$. Therefore, the test statistic
\begin{eqnarray*}
cT_\textrm{linear} &=& c\|\widehat{g}_{\textrm{linear}, \mu, t, c}^{\textrm{B}}\|^2=(z^{\textrm{linear}})^TAz^\textrm{linear} \\
&\leq & (z^0)^TAz^0 + \varsigma^TA\varsigma + 100\overrightarrow{C_k(t)}^T A \overrightarrow{C_k(t)} + 2\big((z^0)^TA\varsigma+10(z^0)^TA\overrightarrow{C_k(t)} + 10\varsigma^TA\overrightarrow{C_k(t)}\big)\\
&=& T_1 + T_2 + T_3 + T_4,
\end{eqnarray*}
where $\overrightarrow{C_k(t)}=(C_k(t), \ldots, C_k(t))^T$. Now we proceed to prove that $cT_{\textrm{linear}}$ is dominated by $T_1$. Using the fact that $z_i^0$,  $z_j^0$ are independent of each other if $i\ne j$, and $E(\{z_i^0\}^2) = O(\tau_k^2)=O(c/n), E(\{z_i^0\}^4) = O(c^2/n^2) $, then for the first term $T_1$, it is straightforward to show that
\[
\begin{split}
    E\left((z^0)^TAz^0\right)^2=&\sum_{i=1}^c a_{ii}^2 E(z_i^0)^4+ \sum_{i=1}^c\sum_{j=1}^c a_{ij}^2 E(z_i^0z_i^0 z_j^0z_j^0) +\sum_{i=1}^c\sum_{j=1}^c a_{ij}a_{ji} E(z_i^0 z_j^0z_j^0z_i^0)\\
    &+\sum_{i=1}^c\sum_{j=1}^c a_{ii}a_{jj} E(z_i^0 z_i^0z_j^0z_j^0) +\sum_{i=1}^c\sum_{j=1}^c a_{jj}a_{ii} E(z_j^0z_j^0z_i^0z_i^0)\\
    \asymp& c^2n^{-2}\left(\sum_{i=1}^c a_{ii}^2+ \sum_{i=1}^c\sum_{j=1}^c a_{ij}^2 +\sum_{i=1}^c\sum_{j=1}^c a_{ij}a_{ji}+\sum_{i=1}^c\sum_{j=1}^c a_{ii}a_{jj} +\sum_{i=1}^c\sum_{j=1}^c a_{jj}a_{ii}\right)\\
    =&c^2n^{-2}\left(\sum_{i=1}^c a_{ii}^2+2\textrm{trace}(A^2) +2[\textrm{trace}(A)]^2 \right).
\end{split}
\]
In the proof to achieve equation~\eqref{asymp:dist:eqn:0}, we have shown that $a_{ii}\asymp(ch)^{-1}$, $\textrm{trace}(A)\asymp \textrm{trace}(A^2)\asymp h^{-1}$, thence we have that
\begin{equation} \label{eq:linearT1}
    E\left((z^0)^TAz^0\right)^2\asymp c^2(nh)^{-2}.
\end{equation}
Furthermore, since $A\le I_c$, we have that
\[
\varsigma^TA\varsigma \le \sum_{i=1}^c\varsigma^2_i = \sum_{i=1}^c\Big[\frac{\sum_{j=(i-1)\widetilde{n}+1}^{i\widetilde{n}}Q\big(g_0(j/n) - \widehat{y}_j\big)}{\widetilde{n}}\Big]^2 = O_p\big(c/n + cC_k(t)^2\big),
\]
\[
\overrightarrow{C_k(t)}^T A \overrightarrow{C_k(t)}\le cC_k(t)^2.
\]
Since $C_k(t)^2 \asymp \frac{\T_n}{2^{{2b}}} \ll \frac{\T_n}{(nh^{1/2}+n(ch)^{-1})\T_n}\le \frac{\T_n}{(nh+n(ch)^{-1})\T_n} =(nh+n(ch)^{-1})^{-1}$, one has that
\[
cC_k(t)^2 \ll c(nh+n(ch)^{-1})^{-1}\le c(nh)^{-1}.
\]
Together with the fact that $cn^{-1} \ll c(nh)^{-1}$ and equation (\ref{eq:linearT1}), one has that
\[
T_2 \ll T_1, T_3 \ll T_1.
\]
Similarly, it can be shown that $T_4 \ll T_1$. Therefore, $cT_\textrm{linear}\asymp T_1$.  The dominated term $T_1$ in $cT_\textrm{linear}$ is nothing but $cT_{\mu, t, c}$ for testing $H_0: g_0(x) = 0$ based on $z^0$. Therefore, in keep with Lemma \ref{hat_tau_k}, the limiting distribution of $T_\textrm{linear}$ under $H_0^{\textrm{linear}}$ should have the same limiting distribution as $T_{\mu, t, c}$ under $H_0: g_0(x) = 0$. Thus, according to Theorem \ref{asymp:distribution} and Lemma \ref{hat_tau_k_linear}, the result is proved.
\end{proof}

\begin{proof}\textbf{of Theorem \ref{power:linearity}:}
The proof of Theorem \ref{power:linearity} is similar to Theorem \ref{power:thm}. Let $g$ be the function which generates the observations and $\mathcal{P}_{\mathcal{L}(\mathbb{I})}(g) = \argmin_{f\in \mathcal{L}(\mathbb{I})} \|g - f\|^2$ be the projection of $g(\cdot)$ to $\mathcal{L}(\mathbb{I})$. We further define $f(x)$ be the integral function associated with $g-\mathcal{P}_{\mathcal{L}(\mathbb{I})}(g)$, as defined in (\ref{eq:f0}), that is,
\[
f(x)=\frac{1}{2\Delta}\int_{\max(x-\Delta, 0)}^{\min(x+\Delta, 1)}[g-\mathcal{P}_{\mathcal{L}(\mathbb{I})}(g)](s)ds, \textrm{ for } x\in[0,1], \,\, \Delta=\frac{1}{c}.
\]
Therefore, $\|g-\mathcal{P}_{\mathcal{L}(\mathbb{I})}(g)\|_c = \sqrt{\sum_{i=1}^c f^2(i/c)/c}$.
To proceed, we first define $\widehat{g}^\star$ be the least squared estimator based on $Q\big(\mathcal{P}_{\mathcal{L}(\mathbb{I})}(g)(j/n) + \sigma\epsilon_j\big)$ which satisfies $|\widehat{g}(i/n) - \widehat{g}^\star(i/n)|\le C_k(t)$. Let $z^\textrm{linear}_{i,0}$ be the $z^\textrm{linear}_{i}$ based on $\mathcal{P}_{\mathcal{L}(\mathbb{I})}(g)$. According to (\ref{eq:quant:2}), onw has that
\begin{align*}
\left|z^\textrm{linear}_{i,0} - \widetilde{n}^{-1}\sum_{j=(i-1)\widetilde{n}+1}^{i\widetilde{n}}Q\Big( Q\big(\mathcal{P}_{\mathcal{L}(\mathbb{I})}(g)(j/n) + \sigma\epsilon_j\big)- \widehat{g}^\star(j/n)\Big)\right| \le C_k(t),\,\, \textrm{ for $i=1,\ldots,c$.}
\end{align*}
Let $z^\textrm{linear}_0=(z^\textrm{linear}_{1,0}, \ldots, z^\textrm{linear}_{c,0})^T$. Before proceeding, we first define some notations to ease the calculations. Define
\footnotesize
\begin{eqnarray*}
    \widetilde{z}^\textrm{linear}_{i,0} = z^\textrm{linear}_{i,0}\mathbbm{1}\Big(\max\{\sigma|\epsilon_j|+c_s\rho, |g(j/n) - \widehat{g}(j/n)|, |\mathcal{P}_{\mathcal{L}(\mathbb{I})}(g)(j/n) - \widehat{g}^\star(j/n)|\} \le \sqrt{\T_n}\textrm{ for all }j=(i-1)\widetilde{n}+1, \dots i\widetilde{n}\Big),\\
    \widetilde{z}^\textrm{linear}_{i} = z^\textrm{linear}_{i}\mathbbm{1}\Big(\max\{\sigma|\epsilon{j}|+c_s\rho, |g(j/n) - \widehat{g}(j/n)|, |\mathcal{P}_{\mathcal{L}(\mathbb{I})}(g)(j/n) - \widehat{g}^\star(j/n)|\} \le \sqrt{\T_n} \textrm{ for all }j=(i-1)\widetilde{n}+1, \dots i\widetilde{n})\Big).
\end{eqnarray*}
\normalsize
Similar to Lemma \ref{power:thm:lemma:1}, we want to find an upper bound of $|\widetilde{z}^\textrm{linear}_i - \widetilde{z}^\textrm{linear}_{i,0}|$. It is straightforward to show that
\begin{align}\label{eq_linear_proof_diff_z}
&|\widetilde{z}^\textrm{linear}_{i} - \widetilde{z}^\textrm{linear}_{i,0}|\\
&\le \Big|\widetilde{n}^{-1}\sum_{j=(i-1)\widetilde{n}+1}^{i\widetilde{n}}Q(y^\textrm{linear}_{j}) - Q\Big(Q\big( \mathcal{P}_{\mathcal{L}(\mathbb{I})}(g)(j/n) + \sigma\epsilon_j\big)  - \widehat{g}^\star(j/n)\Big)\Big| + 2C_k(t)\nonumber\\
&=\Big|\widetilde{n}^{-1}\sum_{j=(i-1)\widetilde{n}+1}^{i\widetilde{n}}Q\big(Q(y_j) - \widehat{g}(j/n)\big) - Q\Big( Q\big(\mathcal{P}_{\mathcal{L}(\mathbb{I})}(g)(j/n) + \sigma\epsilon_j\big)  - \widehat{g}^\star(j/n)\Big)\Big|+2C_k(t) \nonumber\\
&\leq \Big|\widetilde{n}^{-1}\sum_{j=(i-1)\widetilde{n}+1}^{i\widetilde{n}}Q(y_j) - \widehat{g}(j/n) -  \mathcal{P}_{\mathcal{L}(\mathbb{I})}(g)(j/n) - \sigma\epsilon_j  + \widehat{g}^\star(j/n)\Big|+ 6C_k(t) \nonumber\\
&\leq \Big|\widetilde{n}^{-1}\sum_{j=(i-1)\widetilde{n}+1}^{i\widetilde{n}}y_j  -  \mathcal{P}_{\mathcal{L}(\mathbb{I})}(g)(j/n) - \sigma\epsilon_j + \widehat{g}^\star(j/n) - \widehat{g}(j/n)\Big|+7C_k(t) \nonumber\\
&\le  \bigg|\widetilde{n}^{-1}\sum_{j=(i-1)\widetilde{n}+1}^{i\widetilde{n}}\Big(g(j/n) - \mathcal{P}_{\mathcal{L}(\mathbb{I})}(g)(j/n)\Big)\bigg|  + 8C_k(t)\nonumber\\
&\leq |f(i/c)| + 8C_k(t) + \zeta \nonumber,
\end{align}
where
\[
\zeta = \max_{1\le i\le c}\Big|f(i/c)-\frac{c}{n}\sum_{j=(i-1)\widetilde{n}+1}^{i\widetilde{n}}\big(g(j/n)-\mathcal{P}_{\mathcal{L}(\mathbb{I})}(g)(j/n)\big)\Big| = O(n^{-1}).
\]
Suppose that Condition (\textbf{B}) holds, and consider events
\begin{eqnarray*}
\mathcal{E}^*_1 &=&\{\textrm{$\sigma|\epsilon_i|+c_s\rho\le\sqrt{\T_n}$ for all $1\le i\le n$}\},\\
\mathcal{E}^*_2 &=&\{\textrm{$|g(i/n) - \widehat{g}(i/n)|\le\sqrt{\T_n}$ for all $1\le i\le n$}\},\\
\mathcal{E}^*_3 &=&\{\textrm{$|\mathcal{P}_{\mathcal{L}(\mathbb{I})}(g)(i/n) - \widehat{g}^\star(i/n)|\le\sqrt{\T_n}$ for all $1\le i\le n$}\}.
\end{eqnarray*}
It is easy to show that $P(\mathcal{E}^*_1 \cap \mathcal{E}^*_2 \cap \mathcal{E}^*_3)\rightarrow 1$ as $n\rightarrow \infty$. Let $\mathcal{E}^*=\mathcal{E}^*_1 \cap \mathcal{E}^*_2 \cap \mathcal{E}^*_3$. Thus, we choose $N_\eta'$ s.t. $P(\mathcal{E}^*)\ge1-\eta/3$ if $n\ge N_\eta'$. Define $\omega^\textrm{linear}=(\omega^\textrm{linear}_1, \ldots, \omega^\textrm{linear}_c)^T$, where $\omega^\textrm{linear}_i = \widetilde{z}^\textrm{linear}_i - \widetilde{z}^\textrm{linear}_{i, 0}$. Obviously, $\omega^\textrm{linear}_i = z^\textrm{linear}_i - z^\textrm{linear}_{i, 0}$ under event $\mathcal{E}^*$. Therefore, by (\ref{eq_linear_proof_diff_z}), under event $\mathcal{E}^*$, we have $|\omega^\textrm{linear}_i-f(i/c)|\leq 8C_k(t) + \zeta$. Since $A \le I_c$, we get that
\begin{align}\label{proof_linearity_eq1}
(\omega^\textrm{linear} - \textbf{f})^TA(\omega^\textrm{linear} - \textbf{f})\le\sum_{i=1}^c(\omega^\textrm{linear}_i-f(i/c))^2 \leq 128cC_k(t)^2 + 2c\zeta^2,
\end{align}
where $\textbf{f} = (f(1/c), \ldots, f(c/c))^T$. By  Lemma \ref{power:thm:lemma:4} and (\ref{proof_linearity_eq1}), we get the following lower bound:
\begin{align}\label{proof_linearity_eq2}
(\omega^\textrm{linear})^T A\omega^\textrm{linear}&\ge\frac{1}{2}\textbf{f}^TA\textbf{f}-(\omega^\textrm{linear}-\textbf{f})^TA(\omega^\textrm{linear}-\textbf{f})\nonumber \\
&=\frac{c}{2}\|g-\mathcal{P}_{\mathcal{L}(\mathbb{I})}(g)\|_c^2-\frac{1}{2}\textbf{f}^T(I_c-A)\textbf{f}-(\omega^\textrm{linear}-\textbf{f})^TA(\omega^\textrm{linear}-\textbf{f}) \nonumber\\
&\ge\frac{c}{2}\|g-\mathcal{P}_{\mathcal{L}(\mathbb{I})}(g)\|_c^2-\frac{1}{2}\varrho c(\lambda+c^{-2m})\rho^2-128cC_k(t)^2 -2c\zeta^2.
\end{align}

Using a similar argument in Lemma \ref{power:thm:lemma:2}, and the facts that $Var(z^\textrm{linear}_{i,0}) = Var(z^0_i)(1+o_p(1)) = \tau_k^2(1+o_p(1))$, $C_k(t)^2 \ll \tau_k^2$, as $n, c \rightarrow\infty$, one has that
\begin{align*}
\frac{E\{|(\omega^\textrm{linear})^TAz^\textrm{linear}_0|^2\}}
{(1+(ch^2)^{-1})\tau_k^2\sum_{i=1}^c(|f(i/c)|+8C_k(t) + \zeta)^2}\le8.
\end{align*}
Therefore, there exists $N''$ s.t., when $c\ge N'', P(\mathcal{E}_2)\ge 1-\eta/3$, where event $\mathcal{E}_2$ is defined as
\begin{align}\label{proof_linearity_eq3}
\mathcal{E}_2 = \left\{|(\omega^\textrm{linear})^T A z^\textrm{linear}_0|\le C_\eta'\sqrt{1+(ch^2)^{-1}}\tau_k\sqrt{\sum_{i=1}^c(|f(i/c)|+8C_k(t)+\zeta)^2}\right\},
\end{align}
and $C_\eta' = \sqrt{24/\eta}$.

From Theorem \ref{corollary linear}, it is straightforward to show that
\[
\frac{(z^\textrm{linear}_0)^TAz^\textrm{linear}_0-\textrm{trace}(A)\widehat{\tau}_k^2}{s_c\widehat{\tau}_k^2} = O_p(1).
\]
Thus, there exists $C_\eta''>0$ s.t. $P(\mathcal{E}_3)\ge 1-\eta/3$ for all $c\ge N_\eta'$ and $N''$, where
\[
\mathcal{E}_3=\left\{\bigg|\frac{(z^\textrm{linear}_0)^TAz^\textrm{linear}_0-\textrm{trace}(A)\widehat{\tau}_k^2}{s_c\widehat{\tau}_k^2}\bigg|\le C_\eta''\right\}.
\]
Then $P(\mathcal{E}^*\cap\mathcal{E}_2\cap\mathcal{E}_3)\ge 1-\eta$ for any $c\ge N_\eta'$ and $N''$.

Suppose $g\in S_\rho^m(\mathbb{I})$ satisfies
\begin{align*}
\|g-\mathcal{P}_{\mathcal{L}(\mathbb{I})}(g)\|_c = \sqrt{ c^{-1}\sum_{i=1}^c \big[\frac{c}{2}\int_{\max(i/c-c^{-1}, 0)}^{\min(i/c+c^{-1}, 1)}[g-\mathcal{P}_{\mathcal{L}(\mathbb{I})}(g)](s)ds\big]^2}=\sqrt{\sum_{i=1}^cf^2(i/c)/c}>C_\eta\delta^\star_{\textrm{linear}},
\end{align*}
where
\begin{align} \label{linear:C:eta}
C_\eta=\max\left\{6\varrho\rho^2, \sqrt{1536},(72C_\eta')^2,6(C_\eta''+z_{1-\alpha/2}+1)\right\},
\end{align}
\begin{align*}
\delta^\star_{\textrm{linear}} = \sqrt{c^{-1}\tau_\textrm{linear}^2(1+s_c+(ch^2)^{-1})+\lambda+c^{-2m}+C_k(t)^2+\zeta^2}.
\end{align*}
Then, under event $\mathcal{E}^*\cap\mathcal{E}_2\cap\mathcal{E}_3$, we have
\begin{align}
&\frac{cT_{\textrm{linear},\mu,t,c}-\textrm{trace}(A)\widehat\tau_k^2}{s_c\widehat\tau_k^2}\nonumber \\
=& \frac{(z^\textrm{linear})^TAz^\textrm{linear}-(z^\textrm{linear}_0)^TAz^\textrm{linear}_0}{s_c\widehat{\tau}_k^2}+\frac{(z^\textrm{linear}_0)^TAz^\textrm{linear}_0-\textrm{trace}(A)\widehat{\tau}_k^2}{s_c\widehat{\tau}_k^2}\nonumber\\
\ge&\frac{(z^\textrm{linear})^TAz^\textrm{linear}-(z^\textrm{linear}_0)^TAz^\textrm{linear}_0}{s_c\widehat{\tau}_k^2}-C_\eta''\nonumber\\
=&\frac{(\omega^\textrm{linear})^TA\omega^\textrm{linear}+2(\omega^\textrm{linear})^TAz^\textrm{linear}_0}{s_c\widehat{\tau}_k^2}-C_\eta''\nonumber\\
\ge&\textstyle\frac{\frac{c}{2}\|g-\mathcal{P}_{\mathcal{L}(\mathbb{I})}(g)\|_c^2-\frac{1}{2}\varrho c(\lambda+c^{-2m})\rho^2-128cC_k(t)^2-2c\zeta^2-
2C_\eta'\tau_{k}\sqrt{(1+\frac{1}{ch^2})\sum_{i=1}^c(|f(i/c)|+8C_k(t)+\zeta)^2}}{s_c\widehat{\tau}_k^2}-C_\eta''\nonumber\\
\ge&\textstyle\frac{\frac{c}{2}\|g-\mathcal{P}_{\mathcal{L}(\mathbb{I})}(g)\|_c^2-\frac{1}{2}\varrho c(\lambda+c^{-2m})\rho^2-128cC_k(t)^2-2c\zeta^2-6C_\eta'\sqrt{1+(ch^2)^{-1}}\tau_k\sqrt{c}\|g-\mathcal{P}_{\mathcal{L}(\mathbb{I})}(g)\|_c}
{s_c\widehat{\tau}_k^2}-C_\eta'' \label{linear:thm:point:1}\\
=&\textstyle\frac{\frac{c}{2}\|g-\mathcal{P}_{\mathcal{L}(\mathbb{I})}(g)\|_c^2\left(
1-\frac{\frac{1}{2}\varrho c(\lambda+c^{-2m})\rho^2}{\frac{c}{2}\|g-\mathcal{P}_{\mathcal{L}(\mathbb{I})}(g)\|_c^2}
-\frac{128cC_k(t)^2}{\frac{c}{2}\|g-\mathcal{P}_{\mathcal{L}(\mathbb{I})}(g)\|_c^2}-\frac{2c\zeta^2}{\frac{c}{2}\|g-\mathcal{P}_{\mathcal{L}(\mathbb{I})}(g)\|_c^2}-
\frac{6C_\eta'\sqrt{1+(ch^2)^{-1}}\tau_k\sqrt{c}\|g-\mathcal{P}_{\mathcal{L}(\mathbb{I})}(g)\|_c}{\frac{c}{2}\|g-\mathcal{P}_{\mathcal{L}(\mathbb{I})}(g)\|_c^2}\right)}{s_c\hat{\tau}_k^2}-C_\eta''\nonumber\\
\ge&\frac{\frac{c}{6}\|g-\mathcal{P}_{\mathcal{L}(\mathbb{I})}(g)\|_c^2}{s_c\widehat{\tau}_k^2}-C_\eta''>z_{1-\alpha/2}.\label{linear:thm:point:2}
\end{align}
where (\ref{linear:thm:point:1}) follows from $C_\eta>192$ (see (\ref{linear:C:eta})), i.e.,
\[
\sqrt{\sum_{i=1}^c f(i/c)^2}=\sqrt{c}\|g-\mathcal{P}_{\mathcal{L}(\mathbb{I})}(g)\|_c\ge \sqrt{c}C_\eta\delta^\star_{\textrm{linear}}\ge \sqrt{192c}C_k(t),
\]
\[
\sqrt{\sum_{i=1}^c f(i/c)^2}=\sqrt{c}\|g-\mathcal{P}_{\mathcal{L}(\mathbb{I})}(g)\|_c\ge \sqrt{c}C_\eta\delta^\star_{\textrm{linear}}\ge \sqrt{3c}\zeta,
\]
which leads to
\[
\sum_{i=1}^c(|f(i/c)|+8C_k(t)+\zeta)^2\le 3\sum_{i=1}^cf(i/c)^2+192cC_k(t)^2 +3c\zeta^2\le 9\sum_{i=1}^cf(i/c)^2=9c\|g-\mathcal{P}_{\mathcal{L}(\mathbb{I})}(g)\|_c^2;
\]
and (\ref{linear:thm:point:2}) follows from (\ref{linear:C:eta}), i.e.,
\begin{eqnarray*}
&&\frac{\frac{1}{2}\varrho c(\lambda+c^{-2m})\rho^2}{\frac{c}{2}\|g-\mathcal{P}_{\mathcal{L}(\mathbb{I})}(g)\|_c^2}\le 1/6,\\
&&\frac{128cC_k(t)^2}{\frac{c}{2}\|g-\mathcal{P}_{\mathcal{L}(\mathbb{I})}(g)\|_c^2}\le 1/6,\\
&&\frac{2c\zeta^2}{\frac{c}{2}\|g-\mathcal{P}_{\mathcal{L}(\mathbb{I})}(g)\|_c^2}\le 1/6,\\
&&\frac{6C_\eta'\sqrt{1+(ch^2)^{-1}}\tau_{k}\sqrt{c}\|g-\mathcal{P}_{\mathcal{L}(\mathbb{I})}(g)\|_c}{\frac{c}{2}\|g-\mathcal{P}_{\mathcal{L}(\mathbb{I})}(g)\|_c^2}\le 1/6.
\end{eqnarray*}

Then for any $c\ge N_\eta\equiv\max\{N_\eta',N''\}$, we have
\begin{eqnarray*}
P(\textrm{reject $H_0$}|\textrm{$H_1$ is true})\ge P\left(\textrm{$\mathcal{E}^\star\cap\mathcal{E}_2\cap\mathcal{E}_3$ and $\bigg|\frac{cT_{\textrm{linear},\mu,t,c}-\textrm{trace}(A)\widehat\tau_k^2}{s_c\widehat\tau_k^2}\bigg|
\ge z_{1-\alpha/2}$}\right)\ge 1-\eta.
\end{eqnarray*}
In the end, by direct calculations, we know that $\|g-\mathcal{P}_{\mathcal{L}(\mathbb{I})}(g)\|_c \ge C_\eta\delta^\star_{\textrm{linear}}$ is equivalent as $\|g-\mathcal{P}_{\mathcal{L}(\mathbb{I})}(g)\|_c\ge C_\eta\delta_{n,c,\lambda}^{\textrm{linear}}$. This proves the desired result.
\end{proof}

\begin{proof}\textbf{of Theorem \ref{adaptive:size}:}
Let $\epsilon = (\epsilon_1, \ldots, \epsilon_n)^T\sim N(0, I_n)$ and $\widetilde{y}^0 = (\widetilde{y}_1^0, \ldots, \widetilde{y}_c^0)^T$, where $\widetilde{y}_i^0 = \frac{\sum_{j=(i-1)\widetilde{n}+1}^{i\widetilde{n}}\sigma\epsilon_j}{\widetilde{n}}$ follows a normal distribution. We further define $\varsigma_i = z_i^0 -\widetilde{y}^0_i$ be the difference of $z^0_i$ and $\widetilde{y}^0_i$. Let $\varsigma=(\varsigma_1, \ldots, \varsigma_c)^T$. Consider event
 \[ \mathcal{E}_1=\{\textrm{$\sigma|\epsilon_i|+c_s\rho\le\sqrt{\T_n}$ for all $1\le i\le n$}\}.
 \]
 Then $P(\mathcal{E}_1) \rightarrow 1$ as $n \rightarrow \infty$, and under event $\mathcal{E}_1$, $|\widetilde{y}_i^0 - \frac{1}{\widetilde{n}}\sum_{j=(i-1)\widetilde{n}+1}^{i\widetilde{n}}Q(\sigma\epsilon_j)| \le C_k(t)$. According to (\ref{eq:quant_z0}), one has that $|\varsigma_i| = O_p(C_k(t))$. Note that for any given $m_l \leq m\leq m_u\rightarrow \infty$, the standardized testing statistic
\begin{eqnarray*}
\xi_m &=& \frac{cT_{m}-\textrm{trace}(A_m)\widehat{\tau}_k^2}{s_{c,m}\widehat{\tau}_k^2}\\
&=&\frac{(\varsigma+\widetilde{y}^0)^TA_m(\varsigma+\widetilde{y}^0)-\textrm{trace}(A_m)\widehat{\tau}_k^2}{s_{c,m}\widehat{\tau}_k^2}\\
&=&\frac{\varsigma^TA_{m}\varsigma}{s_{c,m}
\widehat{\tau}_k^2}+\frac{2\varsigma^TA_{m}\widetilde{y}^0}{s_{c,m}\widehat{\tau}_k^2}+\frac{(\widetilde{y}^0)^TA_{m}\widetilde{y}^0-\textrm{trace}(A_m)\widehat{\tau}_k^2}{s_{c,m}\widehat{\tau}_k^2}\\
&=& J_1 + J_2 +J_3.
\end{eqnarray*}
For $J_1$, notice $\varsigma^TA_{m}\varsigma\leq \sum_{i=1}^c\varsigma_i^2=O_p(cC_k(t)^2)$, and $C_k(t)^2 \ll (nh_m^{1/2})^{-1}$ for any $m_l\le m \le m_u$, one has $J_1=o_p(1)$. For $J_2$, using a similar argument as (\ref{eq:boundtmp}), we have
\[
 J_2 =O_p\left(C_k(t)\sqrt{\frac{n}{ch_m}}+C_k(t)\sqrt{nh_m}\right) \rightarrow 0,\,\,\,\, \textrm{ for any } m_l\le m \le m_u.
\]
For $J_3$, we need to use Lemma \ref{adaptive}. Let $\widetilde{A}_m = A_m/s_{c,m}, \boldsymbol{\Psi}_{c} = \widetilde{n}^{-1/2}(\widetilde{y}^0_1, \ldots, \widetilde{y}^0_c)^T$. Define $F_{c, m}:=\boldsymbol{\Psi}_{c}^{\top} \widetilde{A}_{m} \boldsymbol{\Psi}_{c}-E\left[\boldsymbol{\Psi}_{c}^{\top} \widetilde{A}_{m} \boldsymbol{\Psi}_{c}\right]$. Define $Z_{c}=\left(Z_{c, m_{l}}, \ldots, Z_{c, m_{u}}\right)^{\top}$ be an $m_{u}-m_{l}+1$ -dimensional centered Gaussian vector with covariance matrix $\mathfrak{C} = I_{m_{u}-m_{l}+1}$ Next we need to verify the conditions in Lemma \ref{adaptive}.

By direct calculations, we have $E\left[F_{c, m}^{4}\right]-3 E\left[F_{c, m}^{2}\right]^{2} = 48\textrm{trace}(\widetilde{A}^4_m)\asymp\frac{h_{m}}{s^4_{c,m}} \asymp h^3_{m}.$
Then we have $\max _{m_l \leq m \leq m_{u}}\left(E\left[F_{c, m}^{4}\right]-3 E\left[F_{c, m}^{2}\right]^{2}\right) \log ^{6} m_{u} \rightarrow 0$.

On the other hand,
\begin{eqnarray}\label{eq:tmp:adptive}
\max _{m_l \leq m, m' \leq m_{u}}\left|\mathfrak{C}(m, m')-E\left[F_{c, m} F_{c, m'}\right]\right| \leq \sum_{m=m_l}^{m_u}|1-E\left[F_{c, m}^{2}\right]| + 2\sum_{m_l\leq m<m'\leq m_u}|E\left[F_{c, m}F_{c,m'}\right]|.
\end{eqnarray}
For the first term in (\ref{eq:tmp:adptive}), recall that $s_{c,m}^2=2\sum_{1\le i\neq i'\le c}a_{i,i'}^2$  with
$a_{i,i'}$ being the $(i,i')$th entry of $A_m$. Then by Lemma \ref{adaptive_rate}, we have,
\begin{align} \label{adptive eq1}
\sum_{m=m_l}^{m_u}|1-E\left[F_{c, m}^{2}\right]|\log ^{2} m_{u} = \sum_{m=m_l}^{m_u}|1-2\textrm{trace}(\widetilde{A}^2_m)|\log ^{2} m_{u}\asymp \frac{m_u\log^2(m_u)}{ch_{m}} \rightarrow 0.
\end{align}

For the second term in (\ref{eq:tmp:adptive}), we need to find a bound of $|E\left[F_{c, m}F_{c,m'}\right]|$ for $m'>m$. It follows that
\begin{eqnarray*}
E\left[F_{c, m}F_{c,m'}\right] &=& 2\textrm{trace}(A_mA_{m'})/(s_{c,m}s_{c,m'})\\
&\asymp& \frac{1}{\left(h_{m} h_{m'}\right)^{-1 / 2}}\sum_{\nu=1}^{n} \frac{\nu^{-4m}}{(\lambda_m+ \nu^{-2m})^2}\frac{\nu^{-4m'}}{(\lambda_{m'}+ \nu^{-2m'})^2}\\
&=& \frac{1}{\left(h_{m} h_{m'}\right)^{-1 / 2}}\sum_{\nu=1}^{n} \frac{1}{(1+ \lambda_m\nu^{2m})^2(1+\lambda_{m'}\nu^{2m'})^2}\\
&\leq& \frac{1}{\left(h_{m} h_{m'}\right)^{-1 / 2}}\sum_{\nu=1}^{n} \frac{1}{\left(1+\left(h_{m} \nu\right)^{2 m}\right)\left(1+\left(h_{m'} \nu\right)^{2 m'}\right)} \\
&\leq & \frac{1}{\left(h_{m} h_{m'}\right)^{-1 / 2}}\sum_{\nu=1}^{n} \int_{\nu-1}^{\nu} \frac{1}{\left(1+\left(h_{m} x\right)^{2 m}\right)\left(1+\left(h_{m'} x\right)^{2 m'}\right)} d x \\
&=& \frac{1}{\left(h_{m} h_{m'}\right)^{-1 / 2}}\int_{0}^{\infty} \frac{1}{\left(1+\left(h_{m} x\right)^{2 m}\right)\left(1+\left(h_{m'} x\right)^{2 m'}\right)} d x \\
&=&\int_{0}^{\infty} \frac{\sqrt{h_{m} h_{m'}}}{\left(1+\left(h_{m} x\right)^{2 m}\right)\left(1+\left(h_{m'} x\right)^{2 m'}\right)} d x \\
&=&\int_{0}^{\infty} \frac{1}{\left(1+\left(x \sqrt{h_{m} / h_{m'}}\right)^{2 m}\right)\left(m'+\left(x \sqrt{h_{m'} / h_{m}}\right)^{2 m'}\right)} d x \\
&\leq &\int_{0}^{\infty} \frac{1}{1+\left(x \sqrt{h_{m'} / h_{m}}\right)^{2 m'}} d x \\
&=&\left(h_{m'} / h_{m}\right)^{-1 / 2} \int_{0}^{\infty} \frac{1}{1+x^{2 m'}} d x \\
&=&\left(h_{m'} / h_{m}\right)^{-1 / 2}\left(\int_{0}^{1} \frac{1}{1+x^{2 m'}} d x+\int_{1}^{\infty} \frac{1}{1+x^{2 m'}} d x\right) \\
&\leq & C_{0}\left(h_{m'} / h_{m}\right)^{-1 / 2}. \\
\end{eqnarray*}

Therefore, using Lemma \ref{adaptive_rate}, we have
\[2\sum_{m_l\leq m<m'\leq m_u}|E\left[F_{c, m}F_{c,m'}\right]|\log^2(m_u)\lesssim \left(h_{m'} / h_{m}\right)^{-1 / 2}m_u^2\log^2(m_u)\rightarrow 0.
\]
Together with (\ref{adptive eq1}), we have
$\max _{m_l \leq m, m' \leq m_{u}}\left|\mathfrak{C}(m, m')-E\left[F_{c, m} F_{c, m'}\right]\right| \log ^{2} m_{u} \rightarrow 0$.

Therefore, by Lemma \ref{adaptive}, we have
$$
\sup _{x \in \mathbb{R}}\left|P\left(C_n(\max _{m_l \leq m_* \leq m_{u}}\xi_m -C_n)\leq x\right)-P\left(C_n(\max _{m_l \leq m\leq m_{u}} Z_{c, m} -C_n) \leq x\right)\right| \rightarrow 0, \,\,\textrm{as $c\rightarrow\infty$.}
$$

By \cite{extreme}, we know $C_n(\max _{m_l \leq m\leq m_{u}} Z_{c, m} -C_n)$ follows an extreme value distribution. Proof is complete.
\end{proof}

\begin{proof}\textbf{of Theorem \ref{adaptive:power}:}
The proof of Theorem \ref{adaptive:power} is similar to Theorem \ref{power:thm} and Theorem \ref{power:linearity}. We use the same notations as in the proof of Theorem \ref{power:thm}. Suppose $g\in S_\rho^{m_*}(\mathbb{I})$ is the function which generates the samples and $f$ is the corresponding integral function as defined in (\ref{eq:f0}). We consider the following three events as defined in the proof of Theorem \ref{power:thm}.
\begin{align*}
\mathcal{E}_1&=\{\textrm{$\sigma|\epsilon_i|+c_s\rho\le\sqrt{\T_n}$ for all $1\le i\le n$}\},\\
\mathcal{E}_{2}&=\left\{|\omega^T A_{m_*} z^0|\le C_\eta'\sqrt{1+(ch_{m_*}^2)^{-1}}\tau_k\sqrt{\sum_{i=1}^c(|f(i/c)|+4C_k(t)+\zeta)^2}\right\}, \,\,\, C_\eta'=\sqrt{24/\eta},\\
\mathcal{E}_3&=\left\{\bigg|\frac{(z^0)^TA_{m_*}z^0-\textrm{trace}(A_{m_*})\widehat{\tau}_k^2}{s_{c, m_*}\widehat{\tau}_k^2}\bigg|\le C_\eta''\right\}.
\end{align*}
Since $P(\mathcal{E}_1) \rightarrow 1$ as $c\rightarrow \infty$, there exist $N_\eta' > 0$, such that $P(\mathcal{E}_1) \geq 1-\eta/3$ for all $c\geq N_\eta'$. Follows from Lemma \ref{power:thm:lemma:2} that there exists $N''$ s.t., when $c\ge N''$, $P(\mathcal{E}_2) \geq 1-\eta/3$. Furthermore, using Theorem \ref{asymp:distribution}, there exists $C_\eta''>0$ s.t. $P(\mathcal{E}_3)\ge 1-\eta/3$ for all $c\ge \max\{N_\eta', N''\}$.

Suppose $g\in S_\rho^{m_*}(\mathbb{I})$ satisfies $\|g\|_c\ge C_\eta\delta_*$,
where
\begin{align*}
C_\eta&=\max\left\{6\varrho\rho^2,384,(72C_\eta')^2,6(C_\eta''+z_{1-\alpha/2}+1)\right\},\\
\delta_* &= \sqrt{c^{-1}\tau_k^2(1+s_{c, m_*}\log^{1/2}(m_u)+(ch_{m_*}^2)^{-1})+\lambda+c^{-2m_*}+C_k(t)^2+\zeta^2}, \\
\tau_k^2 &= Var(z_1|H_0) = Var(z_i^0)=O(\widetilde{n}^{-1}), \,\,\, \zeta= \max\limits_{i = 1,\ldots,c}\big|f(i/c) - \frac{1}{\widetilde{n}}\sum_{j=(i-1)\widetilde{n}+1}^{i\widetilde{n}}g(j/n)\big|=O(n^{-1}).
\end{align*}

Since $m_u \rightarrow\infty, m_* \leq m_u$ eventually. So we assume $m_*\leq m_u$. Then it holds that
\begin{eqnarray*}
\inf_{\substack{g\in S_\rho^{m_*}(\mathbb{I})\\
\|g\|_c\ge C_\eta\delta_*}} P(\xi_*>q_\alpha) &=& \inf_{\substack{g\in S_\rho^{m_*}(\mathbb{I})\\
\|g\|_c\ge C_\eta\delta_*}} P(\xi_{\textrm{max}}>C_n + q_\alpha/C_n)\\
&\geq& \inf_{\substack{g\in S_\rho^{m_*}(\mathbb{I})\\
\|g\|_c\ge C_\eta\delta_*}} P(\xi_{m_*}>C_n + q_\alpha/C_n)\\
&=& \inf_{\substack{g\in S_\rho^{m_*}(\mathbb{I})\\
\|g\|_c\ge C_\eta\delta_*}} P(\frac{cT_{m_*}-\textrm{trace}(A_{m_*})\widehat{\tau}_k^2}{s_{c,m_*}\widehat{\tau}_k^2}>C_n + q_\alpha/C_n).
\end{eqnarray*}

Similar to the proof of Theorem \ref{power:thm}, we know with probability approaching one
\begin{eqnarray*}
&&\frac{cT_{m_*}-\textrm{trace}(A_{m_*})\hat{\tau}_k^2}{s_{c,m_*}\hat{\tau}_k^2}\nonumber\\
&=&\frac{z^TA_{m_*}z-(z^0)^TA_{m_*}z^0}{s_{c,m_*}\hat{\tau}_k^2}+\frac{(z^0)^TA_{m_*}z^0-\textrm{trace}(A_{m_*})\hat{\tau}_k^2}{s_{c,m_*}\hat{\tau}_k^2}\nonumber\\
&=&\frac{z^TA_{m_*}z-(z^0)^TA_{m_*}z^0}{s_{c,m_*}\hat{\tau}_k^2}+O_p(1)\nonumber\\
&=&\frac{\omega^TA_{m_*}\omega+2\omega^TA_{m_*}z^0}{s_c\hat{\tau}_k^2}+O_P(1)\nonumber\\
&\geq&\frac{\frac{c}{2}\|g\|_c^2\left(
1-\frac{\frac{1}{2}c_{m_*}c(\lambda_{m_*}+c^{-2m_*})\rho^2}{\frac{c}{2}\|g\|_c^2}
-\frac{32cC_k(t)^2}{\frac{c}{2}\|g\|_c^2}-\frac{2c\zeta^2}{\frac{c}{2}\|g\|_c^2}-
\frac{6C_\eta'\sqrt{1+(ch_{m_*}^2)^{-1}}\tau_k\sqrt{c}\|g\|_c}{\frac{c}{2}\|g\|_c^2}\right)}{s_{c,m_*}\hat{\tau}_k^2}+O_p(1).
\end{eqnarray*}
Since $ C_n \asymp (\log(m_u))^{1/2}$
we have
\[
\frac{cT_{m_*}-\textrm{trace}(A_{m_*})\hat{\tau}_k^2}{s_{c,m_*}\hat{\tau}_k^2}\ge C_n + q_\alpha/C_n.
\]
Therefore, for any $c\ge N_\eta\equiv\max\{N_\eta',N''\}$, we have
\begin{eqnarray*}
P(\textrm{reject $H_0$}|\textrm{$H_1$ is true})\ge P\left(\textrm{$\mathcal{E}_1\cap\mathcal{E}_2\cap\mathcal{E}_3$ and $\bigg|\frac{cT_{m_*}-\textrm{trace}(A_{m_*})\hat{\tau}_k^2}{s_{c,m_*}\hat{\tau}_k^2}\bigg|
\ge C_n + q_\alpha/C_n$}\right)\ge 1-\eta.
\end{eqnarray*}

In the end, by direct calculations, we know that $\|g\|_c\ge C_\eta\delta_*$ is equivalent as $\|g\|_c\ge C_\eta\delta_{n, c, a_n}$. This proves the desired result.
\end{proof}

\end{document}